\documentclass[a4paper,reqno]{amsart}

\usepackage[dvips]{graphicx}

\usepackage[font=small]{caption}

\usepackage{setspace}
\usepackage[totalwidth=14cm,totalheight=22cm,hmarginratio={1:1},vmarginratio={1:1},heightrounded]{geometry}
\usepackage[T1]{fontenc}
\usepackage{amsmath,amssymb,amscd,amsthm}
\usepackage{amsfonts}

\usepackage[latin1]{inputenc}
\usepackage{latexsym}
	
\usepackage{colortbl}
\usepackage{color}
\usepackage{ae}
\usepackage{enumitem}

\usepackage{tikz}
\usepackage{cancel}
\usepackage{bbm}

\usepackage{mathrsfs}
\usepackage{relsize}

\usepackage[toc]{appendix}

\usepackage{tikz-cd}
\usepackage{hyperref}

\newtheorem{theorem}{Theorem}[section]

\newtheorem{thmroman}{Theorem}

\newtheorem{manualtheoreminner}{Theorem}
\newenvironment{manualtheorem}[1]{%
  \renewcommand\themanualtheoreminner{#1}%
  \manualtheoreminner
}{\endmanualtheoreminner}

\newtheorem{lemma}[theorem]{Lemma}
\newtheorem{prop}[theorem]{Proposition}
\newtheorem{cor}[theorem]{Corollary}
\newtheorem{defi}[theorem]{Definition}
\newtheorem{notation}[theorem]{Notation}
\newtheorem{remark}[theorem]{Remark}
\newtheorem*{remark*}{Remark}

\newtheorem{cla}{Claim}

\newcommand{\e}{\varepsilon}
\newcommand{\pt}{\partial}
\newcommand{\mc}[1]{\mathcal{#1}}

\newcommand{\mb}[1]{\mathbf{#1}}
\newcommand{\ms}[1]{\mathscr{#1}}
\newcommand{\ol}[1]{\overline{#1}}
\newcommand{\rar}{\rightarrow}
\newcommand{\R}{\mathbb{R}}

\renewcommand{\H}{\mathbb{H}}

\renewcommand{\S}{\mathbb{S}}
\newcommand{\len}{{\rm len}}

\newcommand{\diam}{{\rm diam}}
\newcommand{\area}{{\rm area}}
\newcommand{\clconv}{{\rm conv}}

\renewcommand{\area}{{\rm area}}

\newcommand{\inter}{{\rm int}}

\renewcommand{\hat}{\widehat}
\newcommand{\la}{\langle}
\newcommand{\ra}{\rangle}

\renewcommand{\S}{\mathbb{S}}
\renewcommand{\H}{\mathbb{H}}

\newcommand{\CH}{\clconv}
\newcommand{\CT}{\textsc{ct}}

\renewcommand{\L}{\operatorname{L}}

\renewcommand{\div}{\operatorname{div}}

\newcommand{\Isom}{\mathrm{Isom}}

\newcommand{\tr}{\mbox{\rm tr}}

\newcommand{\grad}{\operatorname{grad}}

\renewcommand{\P}{\mc P}
\newcommand{\p}{\mathbf{p}} 

\newcommand{\M}{\mc M}

\def\wti#1{\widetilde{#1}}

\def\ol#1{\overline{#1}}

\DeclareFontFamily{OT1}{pzc}{}
\DeclareFontShape{OT1}{pzc}{m}{it}{<-> s * [0.99] pzcmi7t}{}

\providecommand{\keywords}[1]
{
  \small	
  \textbf{\textit{Keywords---}} #1
}

\begin{document}

\setcounter{secnumdepth}{3}
\setcounter{tocdepth}{2}

\title[Polyhedra in flat spacetimes]{Polyhedral surfaces in flat (2+1)-spacetimes and balanced cellulations on hyperbolic surfaces}

\author{
  Fran\c{c}ois Fillastre 
and
Roman Prosanov
}

\begin{abstract}
We first prove that given a hyperbolic metric $h$ on a closed surface $S$, any flat metric on $S$ with negative singular curvatures isometrically embeds as a convex polyhedral Cauchy surface in a unique future-complete flat globally hyperbolic maximal (2+1)-spacetime whose linear part of the holonomy is given by $h$.
The Gauss map allows to translate this statement to a purely 2-dimensional problem of finding a balanced geodesic cellulation on the hyperbolic surface, from which the flat metric can be easily recovered.

We show next that given two such flat metrics on the surface, there exists a unique  pair of future- and past-complete flat globally hyperbolic maximal (2+1)-spacetimes with the same holonomy, in which the flat metrics embed respectively as convex polyhedral Cauchy surfaces. The proof follows from convexity properties of the total length of the associated balanced geodesic cellulations over Teichm\"uller space.
\end{abstract}

\maketitle

\medskip


\setcounter{secnumdepth}{4}
\setcounter{tocdepth}{4}
\tableofcontents

\section{Introduction}

\subsection{Main statements}
The Gauss map encodes an intricate information about how a surface is curved in space. For convex surfaces the Gauss map can be defined without further regularity assumptions, but by dropping the smoothness we accept that the Gauss map becomes not really a map, but a multivalued map. The first concern of the current article is a study of the Gauss map from convex polyhedral surfaces in flat (2+1)-spacetimes. In particular, we are interested how it translates their extrinsic geometry into the intrinsic geometry of the target, which in our setting is a hyperbolic surface. Let us first recall some details in the classical case of Euclidean 3-space $\mathbb{R}^3$.

Consider a (compact) convex polyhedron $P$ in $\mathbb{R}^3$. The Gauss map of the polyhedron is a multivalued map that takes a boundary point of $P$ and associates it with the set of outward unit normals to the support planes of $P$ at this point. 
Therefore, the Gauss map transforms the polyhedron $P$ to a convex geodesic cellulation of the 
round sphere $\mathbb{S}^2\subset \mathbb{R}^3$. This cellulation, denoted as $\ms G$, is inherently dual to the face cellulation of the original polyhedron $P$. Here is how this duality is manifested:
\begin{itemize}[nolistsep]
\item The Gauss image of a face of $P$ is a vertex of $\ms G$.
\item For an edge $E$ of $P$ its Gauss image is an edge $e$ in $\ms G$, and the length of $e$ in $\mathbb{S}^2$ is equal to the exterior dihedral angle of $P$ at $E$.
\item Finally, under the Gauss map a vertex $V$ of the polyhedron $P$ becomes a convex spherical polygon. The area of this spherical polygon is precisely the singular curvature present at the vertex $V$.
\end{itemize}

We assign a positive weight, denoted as $w_e$, to each edge $e$
 of $\ms G$, equal to the length of the respective edge in the original polyhedron $P$.
Considering a vertex $v$ of $\ms G$, we see that the weights $w_e$ of the incident edges satisfy the following balance condition:

\begin{equation}\label{eq:balance 1}
\sum_{e\ni v} w_e U_e = 0~.
\end{equation}

Here $U_e$ is the unit tangent vector of $\mathbb{S}^2$ to the edge $e$ at the vertex $v$. We will refer to this balanced convex cellulation $(\ms G,w)$, as ``the Gauss image of $P$'', although this is a slight abuse of terminology.
We will call such a geodesic convex cellulation over $\mathbb{S}^2$ with positive weights satisfying the balance condition \eqref{eq:balance 1} a \emph{balanced cellulation} over $\mathbb{S}^2$.

In the other direction, from a balanced cellulation $(\ms G,w)$ over $\mathbb{S}^2$ it is not hard to construct a convex polyhedron $P$ such that its Gauss image is $(\ms G,w)$. Also $P$ is unique up to Euclidean isometries, see e.g.   \cite{sch77,schneider}.

The induced intrinsic metric $d$ on the boundary of $P$ is a flat metric with conical singularities with positive curvature on the topological sphere $S^2$. This metric can be also recovered from the Gauss image $(\ms G,w)$ of $P$. Indeed, the balance condition \eqref{eq:balance 1} means  that we may associate to each vertex of $(\ms G,w)$ a flat polygon so that the edges have outward unit normals $U_e$ and lengths $w_e$. Gluing these polygons along the combinatorics dual to $\ms G$, we obtain a flat surface with conical singularities, which is isometric to $d$ by construction. We will say that the flat metric $d$ is \emph{dual} to $(\ms G,w)$, see Figure~\ref{fig:vertex-to-face}.

\begin{figure}[h]
\begin{center}

\includegraphics[scale=0.2]{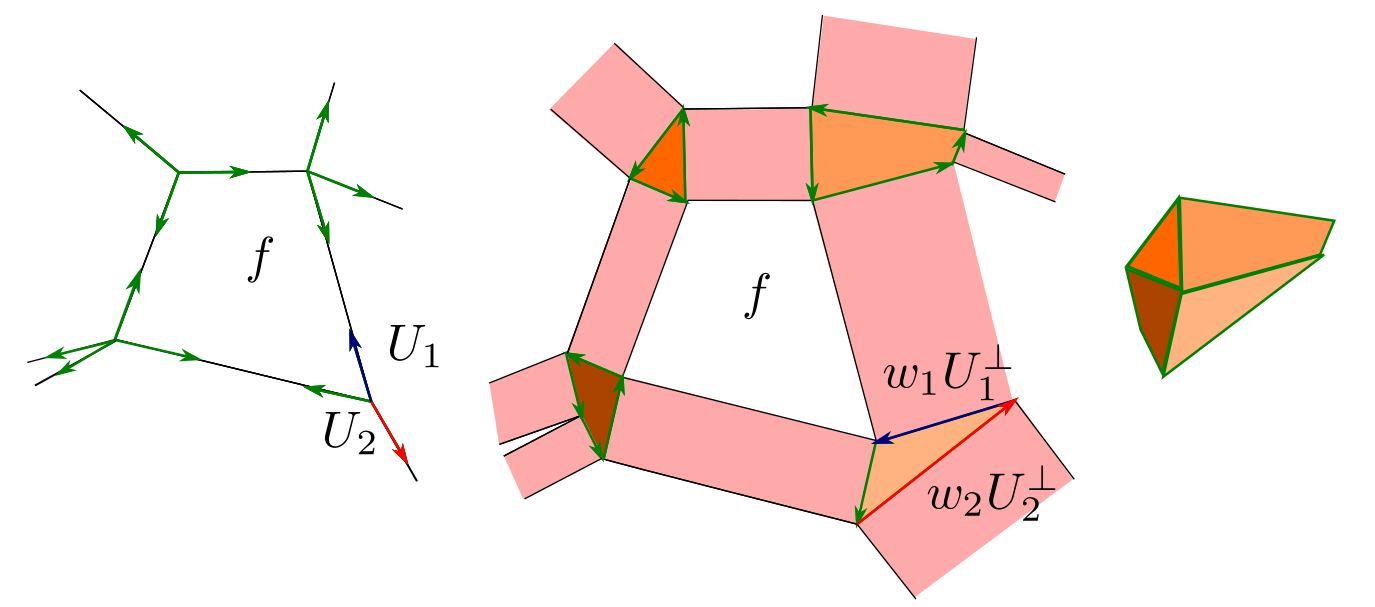}
\caption{ \footnotesize	The left hand side of the picture is a part of a balanced cellulation, around a face $f$. The green arrows in this picture are unit tangent vectors to the edges of the cellulation. In the picture in the middle, Euclidean strips are glued to every edge appearing in the left-hand side picture, in a direction orthogonal to the direction of the edge, and of width the weight of the edge. By the balance condition, the empty parts between the strips may be filled by convex flat polygons (colored brown). Gluing those flat polygons by identifying two segments that bound the same strip, produces a flat metric on the sphere (picture on the right hand side). These faces are glued around a vertex, and the sum of the interior angles of the faces at this vertex is equal to the sum of the exterior angles of the spherical convex polygon $f$. By the Gauss--Bonnet Formula, this sum plus the area of $f$ is equal to $2\pi$.}
\label{fig:vertex-to-face}
\end{center}
\end{figure}

At first glance, it might appear that retrieving the flat metric from a balanced cellulation provides less comprehensive information compared to recovering the entire convex polyhedron. However, the renowned theorem of Alexandrov demonstrates that surprisingly it is sufficient to recover the metric alone.

\begin{theorem}[Alexandrov Theorem, extrinsic version]\label{theorem: alex eucl extr}
Let $d$ be a flat metric with positive singular curvatures on the sphere $S^2$. Then there exists a unique (up to isometries) convex polyhedron $P$ in Euclidean space such that the induced intrinsic distance on $\partial P$ is isometric to $d$.
\end{theorem}

From what was said above, Theorem~\ref{theorem: alex eucl extr} is equivalent to the following statement:

\begin{theorem}[Alexandrov Theorem, intrinsic version]\label{theorem: alex eucl intr}
Let $d$ be a flat metric with positive singular curvatures on the sphere $S^2$. Then there exists a unique (up to isometries) 
balanced cellulation $(\ms G,w)$ on $\mathbb{S}^2$ whose dual metric is isometric to $d$.
\end{theorem}

The first aim of the present article is to generalize the theorems above to higher genus. 
The generalization of the definition of balanced cellulation is immediate to any constant curvature surface. The dual metric of a balanced cellulation over a flat torus is a flat metric (without cone-singularities), and we won't consider this case.

We will say that a surface $S$ is of \emph{hyperbolic type} if $S$ is a connected  closed oriented surface of genus $>1$, and that $S$ is a \emph{hyperbolic surface} if it is of hyperbolic type and is endowed with a hyperbolic metric.
For a hyperbolic surface, due to the Gauss--Bonnet Formula for hyperbolic convex polygons,  
the curvature at the cone points of the dual flat metric is negative. The first aim of the present article is to prove the following  hyperbolic version of 
 Theorem~\ref{theorem: alex eucl intr}.

\begin{thmroman}[Hyperbolic Alexandrov Theorem, intrinsic version]\label{thmI}
Let $d$ be a flat metric with negative singular curvatures on a surface $S$ of hyperbolic type, and let $h$ be a hyperbolic metric on $S$. Then there exists a unique balanced cellulation $(\ms G,w)$ on $(S,h)$ whose dual metric is isometric to $d$ via an isometry isotopic to the identity.
\end{thmroman}

A simplest example of the dualization is given by the convex cellulation of the hyperbolic genus $2$ surface obtained by identifying the side of a regular hyperbolic octagon. If weights $1$ are provided on any of the edges, then in the tangent space at the unique vertex of the cellulation we obtain a flat regular octagon. Identifying its sides, we get a flat metric with a unique cone singularity of curvature $4\pi$.

There is also an extrinsic version of Theorem~\ref{thmI}, which is actually the main point of our focus. We want to embed the flat metric $d$ isometrically as a convex polyhedral surface $P$ in a natural 3-dimensional space $M$. To have abundance of polyhedral surfaces, $M$ should have constant curvature. As the metric is flat, $M$ should be flat, and as the singular curvature are negative, $M$ should be Lorentzian.
The convexity allows to embed the universal covering of $P$ into Minkowski space $\R^{2,1}$ so that the embedding is equivariant with respect to a holonomy map $\pi_1S \rar {\rm Isom}_0(\R^{2,1})$. As the Gauss image of $P$ must be isometric to $(S,h)$,  the linear part $\rho$ of the holonomy is prescribed by $h$. Here we are identifying  the hyperbolic plane with its hyperboloid model in Minkowski space $\R^{2,1}$, and choosing a developing map, so the holonomy of $(S,h)$ has values in the subgroup  ${\rm SO}_0(2,1)$ of the group of linear isometries of Minkowski space. 

We will consider spacetimes $M$ whose holonomy is obtained in the following way. For any element $\gamma\in \pi_1S$ we define $\tau(\gamma)\in \R^{2,1}$ and set $\rho_\tau(\gamma)x:=\rho(\gamma)x+\tau(\gamma)$ for all $x \in \R^{2,1}$. The image $\rho_\tau(\pi_1 S)$ is a subgroup of isometries of Minkowski space if and only if $\tau$ satisfies a cocycle condition.  From a work of G. Mess \cite{mess,mess+} it is known that there exists a maximal open set in Minkowski space on which $\rho_\tau(\pi_1 S)$  acts freely and properly.  This maximal open set has two convex connected components, one future-complete and one past-complete. The quotient of such a connected component gives an example of what is known as flat globally hyperbolic maximal Cauchy compact (2+1)-spacetime, which is a concept coming from General Relativity (in short, flat GHMC \((2+1)\) spacetime ---we will often skip the reference to the dimension in the text). Let us denote 
these quotients by $\Omega^+(\rho,\tau)$ and $\Omega^-(\rho,\tau)$ respectively.

For the details of this construction, we refer the reader to \cite{mess,mess+,bonsante,barbot} (the last two consider higher-dimensional generalizations).

The following theorem implies Theorem~\ref{thmI} after an application of the Gauss map.

\begin{manualtheorem}{I'}[Hyperbolic Alexandrov Theorem, extrinsic version]\label{thmI'}
Let $d$ be a  flat metric with negative singular curvatures on a  surface $S$ of  hyperbolic type, and let $\rho$ be the holonomy of a hyperbolic metric on $S$.  
Then there exists a unique convex Cauchy polyhedron $P$ in a unique (up to marked isometry) flat GHMC (2+1)-spacetime $\Omega^+(\rho,\tau)$ such that the induced intrinsic metric on $\partial P$ is marked isometric to $d$. 
\end{manualtheorem} 

By construction, (2+1)-spacetimes $\Omega^{\pm}(\rho, \tau)$  are marked in the sense that for every Cauchy surface $\Sigma \subset \Omega^{\pm}(\rho, \tau)$ there exists a natural homeomorphism $S \rar \Sigma$ determined up to isotopy. The marked isometry between $(S, d)$ and $\pt P$ in Theorem~\ref{thmI'} means that this isotopy class contains an isometry (which is then unique).

Actually, Theorem~\ref{thmI} and Theorem~\ref{thmI'} are equivalent, as from the balanced cellulation it is easy to construct the spacetime and the polyhedral surface, in a way similar  to the Euclidean case, see  \cite{FV,FS}.

The second aim of the present article is to prove the following Simultaneous Uniformization Theorem:

\begin{manualtheorem}{II'}\label{thmII}
Let $d_+$ and $d_-$ be   flat metrics with negative singular curvatures on a  surface $S$ of  hyperbolic type.
 Then there exists a  unique (up to marked isometries) pair of flat GHMC (2+1)-spacetimes $\Omega^+(\rho,\tau)$ and $\Omega^-(\rho,\tau)$ containing unique convex Cauchy polyhedra with induced intrinsic metrics on the boundary  marked isometric to $d_+$ and $d_-$ respectively. 
 \end{manualtheorem}
 
The meaning of Theorem~\ref{thmII} is that  to a discrete and faithful representation $\rho_\tau: \pi_1S \rar {\rm Isom}_0(\R^{2,1})$ we attach a geometric object, a pair of convex polyhedral surfaces embedded into the respective spacetimes $\Omega^\pm(\rho, \tau)$, and show that our representation plus the geometric object is completely determined by the intrinsic geometry of the surfaces. In Section~\ref{reviewII} we review analogous results from other settings serving as a source of our motivation.

The cocycles $\tau$ introduced above have values in Minkowski space. But there is a natural identification $\Lambda$ between Minkowski space and the Lie algebra of ${\rm SO}(2,1)$.  If the Teichm\"uller space $\mc T(S)$ of $S$ is seen as  the space of the equivalence classes of discrete and faithful representations of $\pi_1 S$ into ${\rm SO}_0(2,1)$, then, applying the identification $\Lambda$, a cocycle is identified with a tangent vector to the Teichm\"uller space. 
 In turn, Theorem \ref{thmI'} says that any  flat metric $d$ over $S$ with negative singular curvatures defines a vector field $X_d$ over $\mc T(S)$, which will be shown to be $C^1$. In the present article we will prove the following theorem, which implies Theorem~\ref{thmII}. It is worth noting that if a convex polyhedral  set $P$ in Minkowski space is $\rho_\tau$-invariant then  $-P$ is $\rho_{-\tau}$-invariant, and of course the induced intrinsic metrics over $\partial P$ and $\partial (-P)$ are related by an isometry.

\begin{thmroman}\label{thmII'}
Let $d_1$ and $d_2$ be two   flat metrics over $S$ with negative singular curvatures. Then $X_{d_1}$ and $-X_{d_2}$ intersect in the tangent bundle of Teichm\"uller space at a unique point, and the intersection is transverse.
\end{thmroman}

Theorem~\ref{thmII'} is proved by associating to the  metric $d$ the \emph{total length function} $\L_d$ over the Teichm\"uller space. For any hyperbolic metric over $S$ it is equal to the sum of the weighted  lengths of the edges of the corresponding balanced convex cellulation. It will turn out that $\L_d$ is strictly convex and proper, and that its Weil--Petersson symplectic  gradient is $X_d$.

\subsection{Sketch of the proof of Theorem~\ref{thmI'} and related results}

The original proof by A.~Alexandrov of Theorem~\ref{theorem: alex eucl extr} introduced the now classical deformation method. Roughly speaking, the space of isometric classes of flat metrics with a given number $n$ of vertices and the space of isometric classes of
convex polyhedra (in $\R^3$) with $n$ vertices are endowed with  suitable topologies that make them manifolds of the same finite dimension. A natural map $\mathcal{I}$ from the latter to the former is given by considering the induced intrinsic metric on the boundary of a polyhedron. Theorem~\ref{theorem: alex eucl extr} says that $\mathcal{I}$ is a bijection,  and furthermore, Alexandrov proves that it is a homeomorphism \cite{AlexandrovBook}. A key point is to prove that $\mathcal{I}$ is injective, which is a refinement of the famous Cauchy Theorem.
The proofs of the latter results heavily use special topological properties of the sphere and cannot be directly generalized to surfaces of higher genus. Instead of directly proving  that $\mathcal{I}$ is injective, one may show that it is locally injective, by using a first-order version of the Cauchy Theorem and the Local Inverse Theorem. 
In the case of the sphere, the first-order version of the Cauchy Theorem can be formulated as follows: it is not possible to deform at first order a convex geodesic cellulation of $\mathbb{S}^2$ without changing the face angles at first order \cite[Theorem 10.3.2]{AlexandrovBook}. 

In the hyperbolic case, the analog of the 
 first order Cauchy theorem is the following statement.
\begin{theorem}[{Iskhakov, \cite{iskhakov}}]\label{thm:iskhakov}
It is not possible to deform at first order a convex geodesic cellulation of a hyperbolic surface without changing the face angles at first order. \end{theorem}

This theorem is proven in~\cite{iskhakov}. It seems to us, however, that the proof, unfortunately, contains a mistake, which, however, is fixable for convex cellulations. (I.~Iskhakov in~\cite{iskhakov} claims to prove Theorem~\ref{thm:iskhakov} also for non-convex cellulations, but in this case we, unfortunately, do not see how to make his proof working.) Because of this and because for our purposes we need to slightly generalize this result, we give our exposition of the proof of Theorem~\ref{thm:iskhakov} in Section~\ref{sec:iskhakov} together with more details on the mistake in the original argument.

Now we describe more precisely our proof strategy for Theorem~\ref{thmI'}.
Fix a holonomy $\rho:\pi_1S\to {\rm SO}_0(2,1)$ of a hyperbolic metric $h$ over $S$, and let $V \subset S$ be a finite set of cardinality $n$. By $\mathcal{P}_c(\rho,V)$ we denote the space of configurations of points in $\Omega^+(\rho,\tau)$ for some $\tau$ that are in bijection with $V$ and are in strictly convex position. The convex polyhedra we are looking for are the convex hulls of such configurations.
We will see in Lemma~\ref{open} and Lemma~\ref{lem: p rho connect}  that
$\P_c(\rho, V)$ is a connected analytic manifold of dimension $6\mathbf{g}-6+3n$, where $\mathbf{g}$ is the genus of $S$. Heuristically, $3n$ corresponds to the position of the points, and $6\mathbf{g}-6$ to the choice of the cocycle, since the space of equivalence classes of cocycles can be identified with the tangent space at $\rho_0$ of the Teichm\"uller space of $S$.

On the other hand, let $\M_-(S,V)$ be the space of flat metrics over $S$ with the set of cone points in bijection with $V$, and with negative singular curvatures, considered up to marked isometries fixing $V$ and isotopic to identity.
It is well-known that such a (isotopy class of) metric is determined by a point in $\mc T(S)$, the position of the elements of $V$ and their cone angles. In turn, $\M_-(S,V)$ 
is a connected analytic manifold of dimension $6\mathbf{g}-6+3n$, see Section~\ref{sec: the realization map}
 and Lemma~\ref{lem:dim flat}. 

We define the map $\mathcal{I}_\rho$ from $\mathcal{P}_c(\rho,V)$ to $\M_-(S,V)$ that associates to any configuration of points in 
$\mathcal{P}_c(\rho,V)$ the induced intrinsic metric on the boundary of the convex hull of the marked points. It appears that it is a $C^1$-map (Lemma~\ref{diff}), and Theorem~\ref{thm:iskhakov} implies that $d\mathcal{I}_\rho$ is non-degenerate. By the Local Inverse Theorem, $\mathcal{I}_\rho$ is a local diffeomorphism.

Section~\ref{sec:compactness} establishes that $\mathcal{I}_\rho$ is proper. The compactness result for the ambient spacetimes is inspired by the techniques of F. Bonsante in \cite{bonsante}, and the compactness result for the position of vertices is based on a result of T. Barbot, F. B\'eguin and A. Zeghib that relates the systole of a Cauchy surface to its distance to the initial singularity  \cite{BBZ}.

With the help of few additional topological arguments we obtain (see Section~\ref{sec:proof thm I'''}).

\begin{theorem}\label{thm:1}
$\mc I_\rho:\P_c(\rho, V)\to \mc \M_-(S,V)$ is a $C^1$-diffeomorphism.
\end{theorem}


There is the following cousin of  Theorem~\ref{thmI'}.

\begin{theorem}\label{thm:cas fuchsien}
 Let $d$ be a  flat metric with negative singular curvatures on a  surface $S$ of  hyperbolic type.  Then there exists a unique  convex Cauchy polyhedron $P$ in a unique (up to marked isometries) $\Omega^+(\rho,0)$  such that the induced intrinsic metric on $\partial P$ is marked isometric to $d$. 
\end{theorem}

The spacetimes $\Omega^+(\rho,0)$ are often called \emph{Fuchsian}.
Theorem~\ref{thm:cas fuchsien} was proved in \cite{filMA} by the first author. Note that in this reference the fact that the map induced metric is $C^1$ is not established, but it is covered by Lemma~\ref{diff} of the present article. Instead of Iskhakov Theorem~\ref{thm:iskhakov}, the proof of Theorem~\ref{thm:cas fuchsien} is based on a first-order rigidity result proved in \cite[Theorem~6.2]{sch-polygons}. Later, a variational proof of Theorem~\ref{thm:cas fuchsien} was done in \cite{brunswic}, using the \emph{Hilbert--Einstein functional}. This variational approach was first used in the Euclidean case (see the introduction of \cite{ivan}), and in many other situations. However, we do not know if the variational approach can be implemented in our general non-Fuchsian case. Note that~\cite{filMA} considers similar results for spherical and hyperbolic cone-metrics with negative singular curvatures, while the Fuchsian realizations of hyperbolic cone-metrics of positive singular curvatures were investigated in~\cite{schlenker2002hyperbolic,filAF} ---which were thought as a polyhedral analogue of \cite{lab-sch}.

There is a smooth analogue of Theorem~\ref{thmI'} where the flat metric is replaced by a Riemannian metric of negative curvature. The equivalent statement similar to Theorem~\ref{thmI} is about the existence of a \emph{Codazzi tensor} on a hyperbolic surface, determined by the negatively curved metric on the same surface. It is proved in \cite{trapani-valli}, we refer to \cite{smith} for more details. An interesting point is that, even if the proof of \cite{trapani-valli} is by a deformation method, the solution is identified with a critical point of a functional.

In Theorem~\ref{thmI} we start from a flat metric and a hyperbolic metric, and look for a suitable balanced cellulation. There is a related problem, when one starts from a topological cellulation with positive weights on the edges, and for any hyperbolic metric on the surface finds a homotopic balanced cellulation, which can be considered as a discrete harmonic map.  A proof is done for a different balance condition by minimizing an energy functional \cite{CdV,KT}. The balance condition appearing there for a vertex $v$ is given by
$$\sum_{e \ni v} w_e U_e \ell_e=0, $$
where $\ell_e$ is the length of $e$ for the hyperbolic metric, and it can be equal to $0$.  Note that the balance condition forces the geodesic realization to be convex: if it is not convex at a vertex, then all the $U_e$ will point to a same half space, and as the weights are taken positive, the balance condition cannot be satisfied. Also some edges of the cellulations may degenerate to a point.
When this does not happen, we call the weighted cellulation \emph{admissible}. It is known that 
triangulations are admissible  \cite{CdV,LWZ}. It is possible to see that \cite[Theorem 2.1]{KT} translates as a statement similar to our Theorem~\ref{thmI'}, at least in the admissible case, but instead of the induced metrics one prescribes combinatorics of the polyhedron and for each edge prescribes the value of the edge-length divided by dihedral angle.
See \cite{GLM} for another discretization of the energy functional, ans \cite{lam} for related results.

We  don't know if the result of Theorem~\ref{thmI} corresponds to some minimum of a functional. However, we will obtain as a byproduct of our  proof of Theorem~\ref{thmI'} some regularity properties that are needed in order to state Theorem~\ref{thmII'}, as we will discuss in the next section.

Note that an alternative to prescribe a weighted cellulation is to prescribe a weighted multi-curve, or, more generally, a measured lamination on a  surface. To obtain a measured geodesic lamination from it once a hyperbolic metric is chosen on the surface is classical,  and  it has a clear $3d$ interpretation as it is the Gauss image of the initial singularity of a flat GHMC (2+1)-spacetime, which was first noted by  Mess \cite{mess}. 

Theorem~\ref{thmI'} can be considered as a part of a general paradigm that geometry of spacelike future-convex equivariant convex sets in $\R^{n,1}$ might be seen as a Lorentizan analogue of theory of Euclidean convex bodies, see e.g.~\cite{FV, BF}.
See the next section for similar results to Theorem~\ref{thmI'}  in the hyperbolic or anti-de Sitter setting.


\subsection{Sketch of the proof of Theorem~\ref{thmII'} and related results}
\label{reviewII}

From the discussion above, one can associate to a balanced cellulation $(\ms G, w)$ over a hyperbolic surface $(S,h)$ an element of $T_\rho \mc T(S)$, where $\rho$ is the class of the holonomy of $h$. Indeed, one considers the cohomology class associated to the flat GHMC spacetime, into which the  flat metric dual to $(\ms G,w)$ embeds. Theorem~\ref{thmI'} says that then any flat metric $d$ over $S$, with negative singular curvatures, defines a tangent vector field $X_d$ over Teichm\"uller space. From the regularity of maps involved in the proof of 
Theorem~\ref{thmI'}, it will follow that $X_d$ is $C^1$ (Lemma~\ref{lem:Xd C1}).

We will consider the \emph{total length function} $\L_d$ over $\mc T(S)$ that associates to each hyperbolic metric $h$ the weighted sum of the lengths of the balanced convex cellulation defined by $d$ and $h$. The first point  is that the Weil--Petersson symplectic gradient of $\L_d$ is $-X_d$ (Proposition~\ref{prop:WP gradient}). Note that this immediately implies a \emph{reciprocity formula} as in \cite[Theorem 2.11]{wolpert82}.

The main point will be Lemma~\ref{prop:derivve seconde}, which says that the (Weil--Petersson) Hessian of $\L_d$ is positive definite. The argument is inspired by \cite{wolf}. In this article M. Wolf  computed the Hessian of the length of a closed geodesic, where the length is also considered as a function over Teichm\"uller space (that this function is convex was previously observed by S. Wolpert in \cite{wolpert87}). 

The proof of Theorem~\ref{thmII'} is established by the analysis of the function $L_{d_1}+L_{d_2}$. This argument is inspired by a work of F.~Bonahon \cite{bonahon}, where an analogous result is shown with flat metrics replaced by measured laminations.

In the realm of flat GHMC \((2+1)\)-spacetimes this connection was further extended in \cite{fixedpoint}, where it was proved that such spacetimes can be parameterized by filling pairs of measured laminations (and moreover, it was also extended to the anti-de Sitter case, i.e., the case of constant curvature $-1$). A measured lamination essentially serves as the Gauss image of the initial singularity of a spacetime. It can be seen as a degenerate analogue of Theorem~\ref{thmI'}.

The  \emph{energy} of a closed geodesic was studied in \cite{yamada1999} by S.~Yamada.  For a weighted graph filling the surface the fact that the energy has a minimum was proved in \cite{KT} by T.~Kajigaya and R.~Tanaka. 
The results in 
\cite{KT} (especially Theorem 2.3 and Theorem 2.5), in the case we called admissible in the end of the preceding section,
 imply an analogue of Theorem~\ref{thmII} with weighted filling graphs instead of flat metrics, and the prescribed weights give the edges lengths of the polyhedron divided by the dihedral angles.
 What worked without restriction and what we are doing in the present article is to prescribe a flat metric instead of a weighted cellulation.

The smooth analogue of Theorem~\ref{thmII} was done in \cite{smith} by G.~Smith (see \cite{GF} for a partial result). The proof uses a functional, which is heuristically the same as ours, in the sense that when a respective surface is realized in a flat spacetime, the value of both functionals is the total mean curvature of the surface. In the case when $d$ is a hyperbolic metric, most of the relevant properties of a similar function $\L_d$ were established in~\cite{BMS}. It is curious to notice that despite the general outline of our proof is similar to the one in~\cite{smith}, the details of the arguments are quite different. 

An interesting direction of further research would be to obtain a common generalization of the present article and~\cite{smith}, which would be about isometric embeddings of CAT(0)-metrics. The existence part of Theorem~\ref{thm:cas fuchsien} in this setting was obtained in~\cite{filslu} by the first author and D.~Slutskiy. The existence is obtained by approximation, and thus do not rely on uniqueness results. Uniqueness results for induced metrics on the boundary of an arbitrary convex set is a hard problem, even in Euclidean space, see discussion in  \cite{roman-rig}.

As we already mentioned, Theorem~\ref{thmII} belongs to the stream of works describing a representation of a discrete group into a Lie group, endowed with some kind of a geometric object. In our very brief review we mention only some results on representations of $\pi_1 S$ (for a longer albeit still non-exhausting survey we refer to~\cite{FSSurvey}). A model result is the famous Bers Simultaneous Uniformization Theorem~\cite{bers} stating that for any two Riemann surfaces there exists a essentially unique Kleinian surface group producing these Riemann surfaces simultaneously as the quotients of its two domains of discontinuity in $\mathbb{CP}^1$. 

With a Kleinian surface group one also associates its quotient of $\H^3$, a quasi-Fuchsian hyperbolic 3-manifold. W.~Thurston conjectured that it is also uniquely determined by the bending lamination on the boundary of convex core. The realization part of this was established in~\cite{BO} by F. Bonahon and J.-P. Otal, while the uniqueness part was just recently obtained in~\cite{DS} by B.~Dular and J.-M.~Schlenker. In a quasi-Fuchsian hyperbolic 3-manifold one can consider a totally convex subset, which is homeomorphic to $S \times [-1, 1]$. In the case of smooth strictly convex boundaries such a set (together with the ambient manifold, and hence with the Kleinian group) is uniquely determined by either the first fundamental form of the boundary, or by the third fundamental form as shown by Schlenker in~\cite{schlenker}. The respective realization result on the first fundamental form was shown earlier by F.~Labourie in \cite{labourie}. These results can be considered as generalizations of the classical \emph{Weyl Problem} to the case of non-trivial topology (the Weyl Problem is a smooth counterpart to Theorem~\ref{theorem: alex eucl extr}, see e.g. \cite{HH} for more information on it). The polyhedral counterparts to the results of Schlenker were recently established by the second author in~\cite{roman-dualhyp, roman-hyp}. It is interesting that in contrast to the flat ``quasi-Fuchsian'' case presented here, in the hyperbolic case cone-metrics may admit not so polyhedral, somewhat degenerate embeddings, which make the analysis more intricate. Particularly, a hyperbolic analogue of Theorem~\ref{thmI'} is still open both in smooth and in polyhedral settings: the main progress is a related result on the realization of hyperbolic metrics without cone-singularities, which corresponds to the famous Grafting Problem, see~\cite{SW}, and see~\cite{DW} for the dual problem. We mention also that the recent article~\cite{CS} mixes the results of Labourie--Schlenker with the result of Bers.

The last thing we want to remark is that a curious direction of current research deals with similar problems in the context of \emph{anti-de Sitter geometry}, i.e., Lorentzian geometry of constant curvature $-1$. The interest to it is spanned by deep connections to Teichm\"uller theory, see e.g.~\cite{mess, bsSurvey}. Plenty of mentioned problems still remain open in the anti-de Sitter context, though we refer to the work~\cite{Tam} of A.~Tamburelli, settling the anti-de Sitter version of the mentioned result of Labourie. Also, very recently Q.~Chen and J.-M.~Schlenker resolved in~\cite{CS2} an anti-de Sitter counterpart to Theorem~\ref{thmI'} in the smooth setting. As for polyhedral surfaces in anti-de Sitter spacetimes, they exhibit similar degenerations as in the hyperbolic case (contrary to the flat case), which is additionally aggravated by the lack of some necessary toolbox in the anti-de Sitter setting. Thereby, a thorough understanding of anti-de Sitter polyhedral surfaces is still the matter of further investigations.

Let us emphasize that in our Minkowski case, convex hulls of finitely many points are polyhedral, see Definition~\ref{def:polyhedral} and Lemma~\ref{lem:boudary polyhedral}. As mentioned above, this is not true in the hyperbolic or Anti-de Sitter setting, mainly because the convex hull may meet the boundary of the convex core. In the hyperbolic case however, if the \emph{dual} induced metric is polyhedral, then the resulting convex hull of vertices is polyhedral in the sense of the present article \cite{roman-dualhyp}. A way to explain this is to consider our theorems from the present article not as results for induced metric on surfaces in Minkowski spacetimes, but results for dual induced metrics on surfaces in \emph{co-Minkowski quasi-Fuchsian manifolds}, where a convex core enters the picture, see \cite{barbot-fillastre}. Such manifolds are modeled on \emph{half-pipe geometry}, which was popularized by J. Danciger \cite{jeff}, and which is roughly the geometry of space-like planes of Minkowski space. We, however, chose not to use this point of view in the present article.

\subsection*{Acknowledgments}
The first author was supported by the ANR G\'eom\'etrie et Analyse dans le cadre Pseudo-Riemannien -- GAPR. 

This research of the second author was funded in whole by the Austrian Science Fund (FWF) \url{https://doi.org/10.55776/ESP12}. For open access purposes, the author has applied a CC BY public copyright license to any author-accepted manuscript version arising from this submission.

This article was initiated during a visit of the second author to the University of Montpellier. It was finished during authors' stay at the Erwin Schr\"odinger International Institute for Mathematics and Physics, Vienna, at the thematic program Geometry beyond Riemann: Curvature and Rigidity. They thank the institutions for their hospitality.

The first author wishes to express his gratitude to Andrea Seppi for fruitful discussions related to the content of the present article. We are also grateful to the anonymous referee for helpful comments.

\section{Background}

\subsection{Flat GHMC spacetimes}

Recall the following definitions from the introduction.

\begin{defi}
A surface $S$ is of \emph{hyperbolic type} if $S$ is a connected closed oriented surface of genus $\mathbf{g}>1$, and $S$ is a \emph{hyperbolic surface} if it is of hyperbolic type and endowed with a hyperbolic metric~$h$.
\end{defi}

We consider the coordinates $(x_0, x_1, x_2)$ in the Minkowski space $\R^{2,1}$ such that for two points $x, y \in \R^{2,1}$ their scalar product is expressed as
\[\la x, y\ra=-x_0y_0+x_1y_1+x_2y_2~.\]

Let $S$ be a surface of hyperbolic type.  
In all the article we identify the hyperbolic plane $\H^2$ with the future component of the 
two-sheeted hyperboloid in $\R^{2,1}$. For a hyperbolic metric $h$ over $S$, a developing map provides a homeomorphism from $\tilde S$ onto $\H^2 \subset \R^{2,1}$, and its holonomy  gives a representation $\rho:\pi_1S\to {\rm SO}_0(2,1)$, where ${\rm SO}_0(2,1)$ is the connected component of identity in ${\rm SO}(2,1)$.

The group $\Isom_0(\R^{2,1})$ is ${\rm SO}_0(2,1)\ltimes \R^{2,1}$, and for any map
$$\tau:\pi_1S\to \R^{2,1} $$
satisfying the \emph{cocycle} condition (twisted by $\rho$):
\begin{equation}\label{eq:cocycle condition twisted}
\tau(\gamma\mu)=\rho(\gamma)(\tau(\mu))+\tau(\gamma)~,
\end{equation}
we obtain a representation 
\begin{equation*}\rho_\tau:\pi_1S\to \Isom_0(\R^{2,1})~.\end{equation*}

A $\rho$-cocycle $\tau$ is a \emph{$\rho$-coboundary} if there exists $v\in \R^{2,1}$ such that
$$\tau(\gamma)=\rho(\gamma)(v)-v~.$$
We then write this cocycle as $\tau_v$.
If two $\rho$-cocycles $\tau, \tau'$ differ by a $\rho$-coboundary $\tau_v$, then the actions of $\rho_\tau$ and $\rho_{\tau'}$ are conjugated by the translation by vector $v$, see e.g. \cite{FV}. We denote by $\mc Z^1(\rho)$ the space of $\rho$-cocycles and by $\mc B^1(\rho) \subset \mc Z^1(\rho)$ the subspace of $\rho$-coboundaries. The quotient is denoted by  $\mc H^1(\rho)$. Abusing notation, we will denote by the same letter a cocycle $\tau$ and its equivalence class in $\mc H^1(\rho)$, but the context should always clarify this ambiguity.

We denote by $\mc R(S)$ the  space of Fuchsian representations of $\pi_1S$ into ${\rm SO}_0(2,1)$, and we see its quotient by the conjugate action of   ${\rm SO}_0(2,1)$ as the  Teichm\"uller space $\mc T(S)$ of $S$. Abusing again notation, we will denote by the same letter such a representation $\rho$ and its equivalence class in Teichm\"uller space.

The tangent space of $\mc R(S)$ is standardly identified with the set of $\mathfrak{so}(2,1)$-valued cocycles twisted by the adjoint action of $\rho$, and the tangent space of $\mc T(S)$ at $\rho$ can be considered as the quotient by the respective space of coboundaries, see e.g. \cite{goldman-symp}. There is a natural identification between $\mathfrak{so}(2,1)$ and $\R^{2,1}$: for example a spacelike vector $v$ is identified with an infinitesimal hyperbolic translation along the geodesic in the hyperbolic plane obtained by the intersection of the hyperboloid and the plane orthogonal to the vector, with the amount of displacement defined by the length of $v$. See e.g. \cite{FS} for an explicit description.
This identification is ${\rm SO}_0(2,1)$-equivariant, where ${\rm SO}_0(2,1)$ acts on $\mathfrak{so}(2,1)$ via the adjoint representation. Hence, this identification induces an isomorphism between $T_{\rho}\mc T (S)$ and $\mc H^1(\rho)$. 

For a $\rho$-cocycle $\tau$ there exist 
two disjoint open convex sets of $\R^{2,1}$, maximal for the inclusion, respectively future- and past-complete, 
on which $\rho_\tau(\pi_1S)$ acts freely and properly discontinuously \cite{mess,mess+,bonsante,barbot-fillastre}. We denote them $\tilde{\Omega}^+(\rho,\tau)$ and $\tilde{\Omega}^-(\rho,\tau)$ respectively.
 If $\tau_v$ is a $\rho$-coboundary, then 
 $\tilde{\Omega}^\pm(\rho,\tau)$ and  $\tilde{\Omega}^\pm(\rho,\tau+\tau_v)$ differ by a translation. 
 
We will denote by $\Omega^\pm(\rho,\tau)$ the quotients of $\tilde{\Omega}^\pm(\rho,\tau)$ by $\rho_\tau(\pi_1 S)$ and will refer to them as to \emph{flat GHMC (2+1)-spacetimes}. The construction provides a natural marking isomorphism between $\pi_1\Omega^{\pm}(\rho, \tau)$ and $\pi_1S$. 
By a \emph{marked isometry} we will mean an isometry respecting the marking isomorphism. When two representations are conjugated by an element of ${\rm Isom}_0(\R^{2,1})$, the resulting flat GHMC spacetimes differ by a marked isometry. We are interested in spacetimes up to marked isometry. Hence, sometimes we define $\Omega^\pm(\rho,\tau)$ for $\rho \in \mc T(S)$ and $\tau \in T_\rho\mc T(S)\cong \mc H^1(\rho)$, meaning that we take representatives $\rho \in \mc T(S)$, $\tau \in \mc Z^1(\rho)$ and the quotient of $\tilde\Omega^\pm(\rho, \tau)$. The space $T\mc T(S)$ can be interpreted then as the space of flat GHMC spacetimes up to marked isometry. Equivalently, it can be considered as the space of flat GHMC Lorentzian structures on $S \times \R$ up to isometry isotopic to the identity.

\subsection{Convex polyhedral surfaces in flat spacetimes}

We now study totally convex subsets in  $\Omega^+(\rho, \tau)$. In this section a spacetime is fixed, so we will mostly skip the mention of $\rho$ and $\tau$.
By a \emph{totally convex subset} of $\Omega^+$ we mean a subset that contains every geodesic segment between every two points. For $K \subset \Omega^+$ we denote by $\clconv(K)$ the convex hull of $K$, i.e., the inclusion-minimal totally convex subset containing $K$. In our setting, however, convexity is better described by looking inside a choice of $\tilde\Omega^+$ in $\R^{2,1}$.
A set $C$ is totally convex if and only if $\tilde C$ is a convex subset of $\R^{2,1}$, so $\clconv(K)$ is the quotient of the convex hull of $\tilde K$.

For a convex set $C\subset \R^{2,1}$ we will say that a plane $\Pi$ is \emph{weakly supporting} to $C$ if $C$ belongs to a single closed halfspace bounded by $\Pi$. We will say that $\Pi$ is \emph{supporting} if additionally $\Pi$ has a common point with the closure of $C$. We will say that a surface in $\R^{2,1}$ is \emph{spacelike convex} if it is the boundary of a convex set and has only spacelike supporting planes. A spacelike convex surface in $\Omega^+$ is a surface such that its lift to $\tilde\Omega^+$ is spacelike convex.

We will say that $C \subset \Omega^+$ is \emph{future (complete)} if for any point $p\in C$ we have $I^+(p)\subset C$, where $I^+(p)$ is the set of the endpoints of all the future-directed timelike curves from $p$.
 A \emph{Cauchy surface} of $\Omega^+$ is a surface (not necessarily smooth) such that every inextendible causal curve intersects it exactly once.  Cauchy surfaces in $\Omega^+$ are all homeomorphic (to $S$), hence compact, see e.g. \cite{oneil}. (Note that in \cite{oneil} only timelike curves are used for the definition of Cauchy surfaces, which is weaker than our definition.) The embedding of the universal covering of a convex Cauchy surface to Minkowski space is the graph of a convex $1$-Lipschitz function over the whole horizontal plane \cite[Lemma 3.11]{bonsante}, a fact that we will often use implicitly.

The \emph{cosmological time} of $p \in \Omega^+$ is the supremum of Lorentzian lengths of the past-directed causal curves emanating from $p$ in $\Omega^+$. For a lift $\tilde p \in \tilde\Omega^+$ of $p$ this supremum is realized by a geodesic segment with an endpoint on the initial singularity of $\tilde\Omega^+$, which is basically the spacelike part of $\pt \tilde\Omega^+$. This produces a $C^1$-function on $\Omega^+$, the \emph{cosmological time} $\CT:\Omega^+ \rar \R_{>0}$. For $\alpha>0$ we will denote by $L_\alpha$ the level set of the cosmological time for the value $\alpha$. It is a $C^1$   spacelike convex Cauchy surface. We refer to \cite{bonsante} for details. 

We will need the following basic results.

\begin{lemma}[{\cite{bonsante}, Proof of Theorem 5.1, Steps 2 and 3}]
\label{convpoint}
Let $p \in \Omega^+$ and $C:=\clconv(p)$. Then $C$ is a future totally convex set with compact boundary.
\end{lemma}

\begin{lemma}
\label{cauchy}
Let $C \subset \Omega^+$ be a totally convex set such that the cosmological time is bounded away from zero on $C$. Then it is future and $\pt C$ is a spacelike convex Cauchy surface.
\end{lemma}

\begin{proof}
From Lemma~\ref{convpoint}, for every $p \in C$ we have $C \supseteq \clconv(p) \supset I^+(p)$. Hence, $C$ is a future set. Because of this, its boundary $\pt C$ is a closed achronal topological surface \cite[Cor. 14.27]{oneil}.
Lemma~\ref{convpoint} also says that $\pt\clconv(p)$ is compact, thereby the cosmological time is bounded from above on it. Thus $C$ contains all level surfaces of sufficiently large cosmological time in the interior. Since the cosmological time is also bounded from below on $C$, the complement to the closure of $C$ also contains a level surface of cosmological time. Hence, every inextendible causal curve intersects $\pt C$. The total convexity of $C$ implies that the intersection is unique. Indeed, suppose the converse that $p,q \in \pt C$ are in causal relation. Then the line through them meets $\inter(C)$ in a precompact set, and thus does not meet $L_\alpha$ for sufficiently large $\alpha$, which is a contradiction as all $L_\alpha$ are Cauchy surfaces. Hence, $\pt C$ is a Cauchy surface. Thereby, $\tilde C$ has only spacelike or lightlike supporting planes. Suppose that there is a lightlike supporting plane $\Pi$ at a point $p\in\partial \tilde C$. Then as $I^+(p)\subset \tilde C$, then the boundary of $C$ must contain the lightlike ray $\Pi\cap (\pt I^+(p))$, which contradicts the definition of Cauchy surface.
\end{proof}

For the next result, we will make use of the spherical model of $\R^{2,1}$: a geodesic map sending it injectively to an open half-sphere of the standard sphere $\S^3$. By $\pt_\infty \R^{2,1}$ we will denote its boundary in the spherical model, by $\pt^0_\infty \R^{2,1} \subset \pt_\infty \R^{2,1}$ we will denote the subset of endpoints of the future-directed lightlike rays, and by $\ol{\R^{2,1}}$ we will denote the compactification $\R^{2,1} \cup \pt_\infty \R^{2,1}$.

\begin{lemma}\label{lem:orbite infinie}
Let $\tau \in \mc Z^1(\rho)$, $p \in \tilde\Omega^+$, and $P$ be the $\rho_\tau$-orbit of $p$. Then the limit set of $P$ in $\ol {\R^{2,1}}$ is $\pt_\infty^0 \R^{2,1}$.
\end{lemma}

\begin{proof}
Let $L \subset \tilde\Omega^+$ be the level surface of the cosmological time that contains $P$. For $\gamma \in \pi_1S$, denote the supporting plane to $L$ at $\rho_{\tau}(\gamma) p$ by $\Pi_\gamma$. Denote by $\mathbb B^2 \subset \R^{2,1}$ the unit disk in the horizontal plane translated by the future unit vertical vector. We consider now $\mathbb B^2$ as a model for $\H^2$ and use it as the domain for the support function $s: \mathbb B^2 \rar \R$ of $L$. More precisely, $s(x)$ is the oriented Lorentzian distance from the origin to the support plane of $L$ with inward normal $(x,1)$, see e.g. \cite{BF} for details. By \cite[Proposition 3.7]{BF}, there exists a continuous function $r_{\rho, \tau}: \pt \mathbb B^2 \rar \R$ such that the support function of every $\rho_\tau$-invariant future convex set admits a continuous extension to $\pt \mathbb B^2$ by $r_{\rho, \tau}$. Denote $\mathbb B^2 \cup \pt \mathbb B^2$ by $\ol{\mathbb B^2}$, and consider now this as the domain of the continuous extension of $s$, which we still denote by $s$.

Let $\pt_\infty^c \R^{2,1}$ and $\pt_\infty^t \R^{2,1}$ be the subsets of $\pt_\infty \R^{2,1}$ corresponding to the future causal and future timelike rays respectively. Let $\iota: \ol{\mathbb B^2} \rar \pt_\infty^c \R^{2,1}$ be the homeomorphism induced by the projection from the origin. Abusing the notation, we will denote $s \circ \iota^{-1}$ still by $s$, i.e., now we consider $s$ defined over $\pt_\infty^c \R^{2,1}$. The group $\rho(\pi_1S)$ acts naturally on $\pt_\infty^c \R^{2,1}$. Moreover, since for any Fuchsian group isomorphic to the fundamental group of a closed surface, the limit set of its action on $\H^2$ is the whole $\pt_\infty \H^2$, the action of $\rho$ on $\pt_\infty^t\R^{2,1}$  associates the Gromov boundary $\pt_\infty \pi_1S$ with $\pt_\infty^0 \R^{2,1}$.

Let $q \in \pt_\infty^0 \R^{2,1}$. There exists a sequence $\gamma_i$ converging to $q$ as to an element of $\pt_\infty\pi_1S$. Then for any $x \in \pt_\infty^t \R^{2,1}$ the sequence $\rho(\gamma_i) x$ converges to $q$. Hence, the sequence of planes $\Pi_{\gamma_i}$ converges to the weakly supporting lightlike plane $\Pi$ to $L$ determined by $s(q)$. Note that since $L$ is a closed subset of $\R^{2,1}$ and $\rho_\tau(\pi_1S)$ acts properly discontinuously on it, the $\rho_ \tau$-orbit of $p$ cannot have limit points in $\R^{2,1}$. Suppose that a subsequence of $\rho_{\tau}(\gamma_i)p$ converges to $q' \in \pt_\infty \R^{2,1}$, $q' \neq q$. Then $q' \in \pt_\infty \Pi$, where $\pt_\infty \Pi$ is the limit set of $\Pi$ in $\pt_\infty{\R^{2,1}}$. Then $q'$ either corresponds to a spacelike ray, or to a past lightlike ray. Note, however, that the mentioned result of \cite{BF}  means that an affine coordinate system in $\R^{2,1}$ can be chosen so that $L$ belongs to the future cone of the origin. This clearly means that $q'$ must correspond to a future causal ray. Hence, $q'=q$, i.e. $\rho_{\tau}(\gamma_i)p$ converge to $q$ in $\ol{\R^{2,1}}$.

Suppose that $\gamma_i \in \pi_1S$ is a sequence such that $\rho_{\tau}(\gamma_i) p$ has a limit point $q \in \ol{\R^{2,1}}$. Up to a subsequence, there exists a limit point $q' \in \pt_\infty^0 \R^{2,1}$ for $\gamma_i$, where $\pt_\infty^0 \R^{2,1}$ is perceived as $\pt_\infty \pi_1S$. By the previous argument $\rho_{\tau}(\gamma_i)p$ must converge to $q'$, hence $q=q'$. This finishes the proof.
\end{proof}

\begin{defi}\label{def:polyhedral}We call a \emph{future polyhedral cone} in $\R^{2,1}$ the intersection of the closed future half-spaces of a finite number of spacelike planes meeting at a common point.

We say that a subset $C \subset \Omega^+$ is  \emph{convex polyhedral} if it is closed, totally convex, its boundary is locally modeled on future polyhedral cones and the cosmological time is bounded away from zero on it.

We say that a point $p \in \pt C$ is a \emph{face-point} if it has a local model on a future polyhedral cone with one plane, it is an \emph{edge-point} if it has a local model on a future polyhedral cone with two planes, and a \emph{vertex} otherwise. A \emph{face} of $C$ is the closure of a connected component of face-points and an \emph{edge} is the closure of a connected component of edge-points.
\end{defi}

\begin{remark}
Note that the condition that the cosmological time is bounded from below does not follow from the other conditions. For instance, consider $\tilde\Omega^+(\rho, \tau)$ such that its \emph{dual measured lamination} in the sense of Mess~\cite{mess} is rational. Then $\tilde\Omega^+(\rho, \tau)$ has isolated vertices. One may consider a spacelike supporting plane to $\tilde\Omega^+(\rho, \tau)$ at such a vertex and push it slightly in the direction of $\tilde\Omega^+(\rho, \tau)$. Next, one considers the intersection of the closed future sides of the $\rho_\tau$-orbit of such a plane. The quotient of the resulting $\rho_\tau$-invariant set is a desired example, provided that the push was sufficiently small.
\end{remark}

By Lemma~\ref{cauchy}, for a convex polyhedral subset $C$ its boundary $\partial C$ is a  Cauchy surface, hence compact and then it has a finite number of vertices. 

\begin{lemma}\label{lem:boudary polyhedral}
Every convex polyhedral subset $C \subset \Omega^+$ is the convex hull of its vertices. Conversely, if $V \subset \Omega^+$ is a finite set, then $\clconv(V)$ is a convex polyhedral subset.
\end{lemma}

\begin{proof}
Let $V$ be the set of vertices of $C$. Consider the closure $\ol C$ of $\tilde C \subset \tilde \Omega^+\subset \R^{2,1}$ in $\ol{\R^{2,1}}$. By Lemma~\ref{lem:orbite infinie}, $\ol C \cap \pt_\infty \R^{2,1}$ contains the set $\pt^c_\infty \R^{2,1}$ of the future causal directions. On the other hand, since we can choose the affine coordinate system so that $\Omega^+$ is contained in the future cone of the origin, $\ol C \cap \pt_\infty \R^{2,1}$ is contained in $\pt^c_\infty \R^{2,1}$. Thus, these sets coincide. It follows that $\ol C = \clconv(\tilde V \cup \pt^0_\infty \R^{2,1})$. By the Caratheodory Theorem, every point of $\pt \tilde C$ is a convex combination of at most three points from $\tilde V \cup \pt^0_\infty \R^{2,1}$ (and, clearly, at least one from $\tilde V$). If there is a point of $\pt^0_\infty \R^{2,1}$ participating in the convex combination, then there exists a lightlike ray in $\pt \tilde C$. This contradicts to that $\pt C$ is a spacelike surface. It is easy to deduce then that the interior points of $\tilde C$ are also convex combinations of $\tilde V$ and then obtain the first claim. 

For the second claim, denote $\clconv(V)$ by $C$. Note that because the level sets of the cosmological time are convex, the infimum of the cosmological time on $C$ is achieved at a point of $V$. Hence, due to Lemma~\ref{cauchy}, the boundary $\pt C$ is a spacelike convex Cauchy surface.

Consider the full preimage $\tilde V$ in $\tilde \Omega^+$. Let $\ol C$ be the closure of $\tilde C$ in $\ol{\R^{2,1}}$, so $\ol C=\ol{\clconv}(\tilde V)$. By Lemma~\ref{lem:orbite infinie}, $\ol C=\CH(\tilde V \cup \pt^0_\infty \R^{2,1})$. By the same argument as above, $\pt \tilde C \subset \clconv (\tilde V)$, i.e. $\tilde C$ is closed in $\R^{2,1}$, and $C$ is closed in $\Omega^+$. Note that there are finitely many edges of $\pt \tilde C$ emanating from a point of $\tilde V$. Indeed, otherwise once again there is a lightlike ray belonging to $\pt \tilde C$. It follows that $\tilde C$ is the intersection of a locally finite collection of spacelike planes, i.e., $C$ is convex polyhedral.
\end{proof}

We remark that the face decomposition of the boundary of a convex polyhedral subset $C$ is what we will call a \emph{cellulation}. We would like to consider it pulled back to $S$. To this purpose, let $\Sigma$ be a Cauchy surface in $\Omega^+$. We will need the following classical result.
\begin{cla}\label{claim:existence homeo}
 The marking isomorphism of the fundamental groups of $\pi_1S$ and $\pi_1\Omega^+$ induces a homeomorphism $S \rar \Sigma$ determined up to isotopy.   
\end{cla}
\begin{proof}
It is a general topological fact that the isomorphism of fundamental group gives a homotopy equivalence between the surfaces \cite[Theorem 1.B.8]{hatcher}. In our hyperbolic type case, any homotopy equivalence is homotopic to a homeomorphism \cite[Theorem 8.9]{primer}, and for compact surfaces, the homotopy is an isotopy \cite[Theorem 1.12]{primer}.
\end{proof}

This is particularly the case with the boundary of a convex polyhedral set $C$. Let $V$ be its set of vertices, by the homeomorphism above, abusing the notation, we identify it with a subset of $S$.

\begin{defi}
Let $V$ be a finite set in $S$. A \emph{cellulation} $\ms C$ of $(S, V)$ is a collection of simple disjoint arcs with endpoints in $V$ that cut $S$ into cells homeomorphic to the 2-disk. The points of $V$ are called \emph{vertices} of $\ms C$, the arcs are called \emph{edges} and the cells are called \emph{faces}. Two cellulations are \emph{equivalent} if there exists a homeomorphism $S \rar S$ fixing $V$ and isotopic to identity that sends one cellulation to another.
\end{defi}

To see that the faces of $C$ are homeomorphic to the 2-disk we note that the faces of $\tilde C$ are compact (since $\pt \tilde C$ is spacelike convex) convex Euclidean polygons. We frequently do not distinguish between a cellulation and its equivalence class. After identifying $\pt C$ with $S$, we will refer to the corresponding cellulation as to the \emph{face cellulation} of $C$.

\subsection{Spacetimes with marked points}
\label{sec:markedpoints}

Let $V \subset S$ be a finite set of cardinality $n$. The aim of the present section is to describe a topology on the following space.

\begin{defi}
A (future) \emph{flat GHMC  spacetime with marked points} is a triple $\p=(\rho, \tau, f)$, where $\rho \in \mc T(S)$, $\tau \in \mc H^1(\rho)$ and $f: V \rightarrow \Omega^+(\rho, \tau)$ is a map, the \emph{vertex marking map}.
The space of such marked spacetimes is denoted by $\P(S,V)$.
\end{defi}

By $\tilde V$ we denote the full preimage of $V$ in $\tilde S$ equipped with the natural $\pi_1S$-action. 
By $\widetilde{\P}(S, V)$ we denote the set of triples $(\rho,\tau, \tilde f)$, where $\rho$ is a Fuchsian representation, $\tau \in \mc Z^1(\rho)$ and $\tilde f: \tilde V \rightarrow \tilde\Omega^+(\rho, \tau)$ is an equivariant  map. This is an open subset of the space $T\mc R(S) \times (\R^{2,1})^V$. The group ${\rm Isom}_0(\R^{2,1})$ acts on $T\mc R(S)$ freely and properly by conjugation. By adding the usual action on $(\R^{2,1})^V$, we get an action on $T\mc R(S) \times (\R^{2,1})^V$, and $\widetilde{\P}(S, V)$ is invariant with respect to this action. Next, the group $\pi_1S$ acts on $T\mc R(S) \times (\R^{2,1})^V$ fiberwise, via $\rho_\tau$ on the fiber $\{\rho\}\times \{\tau\} \times (\R^{2,1})^V$. This action is free and properly discontinuous on the set $\{\rho\}\times\{\tau\} \times( \tilde\Omega^+(\rho, \tau))^V$, which is the intersection of $\widetilde{\P}(S, V)$ with $\{\rho\}\times\{\tau\} \times (\R^{2,1})^V$. These two actions commute and the quotient by them of $\widetilde{\P}(S, V)$ is exactly $\P(S, V)$, which is thus endowed with the structure of an analytic manifold of dimension $12\mathbf{g}-12+3n$, where $\mathbf{g}$ is the genus of $S$. We remark that we will need the analytic structure only in Section~\ref{sec:lift}.


We remark that  there is a natural projection map \begin{equation}\label{eq:def hat pi}\mu: \P(S, V) \rar T\mc T(S)\end{equation} forgetting the marked points. 

\begin{notation} By $\P_{c}(S, V)$ we denote the subset of $\P(S, V)$ corresponding to $f$ in a strictly convex position, i.e., no point of $f(V)$ is a convex combination of other points (note that it implies the injectivity of $f$). 
\end{notation}

\begin{lemma}
\label{open}
$\P_c(S, V)$ is open in $\P(S, V)$. In particular, $\P_c(S, V)$ is an analytic manifold of dimension $12\mathbf{g}-12+3n$.
\end{lemma}

\begin{proof}
Suppose the converse, then there exists $\p:=(\rho, \tau, f)\in \P_c(S, V)$ and a sequence $\p_i \in \P(S, V)\backslash \P_c(S, V)$ converging to $\p$. Lift them to a sequence $\tilde \p_i \in \widetilde{\P}(S, V)$ converging to $\tilde \p$. Denote the respective vertex marking maps by $\tilde f_i$, $\tilde f$. Up to passing to a subsequence, by the Carath\'eodory Theorem there exists $\tilde v \in \tilde V$ and a sequence of 4-tuples of points $p^1_i, p^2_i, p^3_i, p^4_i \in \tilde f_i(\tilde V\backslash \tilde v)$ such that $\tilde f_i(\tilde v)$ belongs to the convex hull of $p^1_i, p^2_i, p^3_i, p^4_i$. Up to passing to a subsequence, $p^1_i, p^2_i, p^3_i, p^4_i$ converge to a 4-tuple of points $p^1, p^2, p^3, p^4 \in \big(\tilde f(\tilde V\backslash \tilde v)\cup \pt^0_\infty \R^{2,1}\big)$ by Lemma~\ref{lem:orbite infinie}, containing $\tilde f(\tilde v)$ in their convex hull. It follows that $\tilde f(\tilde v)$ is contained in $\ol{\clconv}(\tilde f(\tilde V \backslash \tilde v))$ in $\ol{\R^{2,1}}$. On the other hand, $\tilde f(\tilde v) \in \R^{2,1}$, thus $\tilde f(\tilde v) \in \R^{2,1} \cap \ol{\clconv}(\tilde f(\tilde V \backslash v))$, where $\R^{2,1} \cap \ol{\clconv}(\tilde f(\tilde V \backslash v))=\clconv(\tilde f(\tilde V \backslash v))$. This contradicts to that $\p \in \P_c(S, V)$.
\end{proof}

\begin{notation} For any flat GHMC  spacetime with marked points $\p=(\rho,\tau,f)$ in $\P(S,V)$, we will denote by $\CH(\p)$ 
 the convex hull of $f(V)$ in $\Omega^+(\rho,\tau)$. \end{notation}

\subsection{The induced metric map}\label{sec: the realization map}

Let $d$ be a flat metric on $S$ with cone-points. We denote by $V(d)$ the set of cone-point of $d$. For a finite $V \subset S$,
by $\ol\M(S, V)$ denote the set of such metrics with $V(d)\subset V$. The metrics are considered  up to isometries fixing $V$ and isotopic to the identity. (Note that the isotopy does not have to fix $V$. In the next section we will have to deal with isotopies fixing $V$, in this case we will say \emph{isotopies relative to $V$}.) We denote by $\M(S, V)$ the  subset of $\ol\M(S, V)$ such that $V(d)=V$. Finally, by $\ol\M_-(S, V)$ and $\M_-(S, V)$ we denote the subsets of $\ol\M(S, V)$, $\M(S, V)$ respectively consisting of metrics with negative singular curvatures.


A metric $d$ of $\ol\M(S, V)$  can be triangulated so that the set of vertices of the triangulation is exactly $V$. Changing the lengths of the edges provides a chart of $\ol\M(S, V)$ valued in $\R^E$, where $E$ is the set of edges. If $\mathbf{g}$ is the genus of $S$ and $n=|V|$, then, as we are considering triangulations, $|E|=6\mathbf{g}-6+3n$. Changes of triangulation provide analytic transition maps, endowing  $\ol\M(S, V)$ with a structure of $(6\mathbf{g}-6+3n)$-dimensional analytic manifold. By looking at affine maps between two triangles, it is easy to see that the topology induced by this analytic structure is the one of Lipschitz convergence of metric spaces. The sets $\M(S, V)$, $\M_-(S, V)$ are open subsets of $\ol\M(S, V)$.

Take an element  $\p\in\P_c(S,V)$. The marking allows us to associate its boundary with $(S, V)$ up to a homeomorphism fixing $V$ and isotopic to identity (Claim~\ref{claim:existence homeo}; the isotopy does not have to fix $V$).  By Lemma~\ref{lem:boudary polyhedral}, the induced metric belongs to  $\M_-(S, V)$. This provides a map
$$\mc I: \P_c(S,V)\to \mc \M_-(S, V) $$
called the \emph{induced metric map}.

By $\ol{\P}_c(S, V) \subset \P(S, V)$ we denote the subset of injective configurations in convex position, i.e., when $f$ is injective and $f(V) \subset \pt\clconv(f(V))$. Note that it differs from the closure of $\P_c(S, V)$ in $\P(S, V)$ as the latter also include the configurations when some of the marked points collapsed. 
The map $\ol{\mc I}: \ol\P_c(S, V) \rar \ol\M_-(S, V)$ is defined in an obvious manner, and its restriction to $\P_c(S,V)$ is~$\mc I$. Pick $\p=(\rho,\tau,f) \in \ol{\P}_c(S, V)$, denote by $V(\p)\subset V$ its set of vertices, i.e., the inclusion-maximal set such that $f(V(\p))$ is in a strictly convex position and $\clconv(f(V(\p)))=\clconv(f(V))$. It is well-defined, since it can be defined alternatively as the set of $v \in V$ such that $f(v) \notin \clconv(f(V \backslash v))$.

If the face celluation of $\p \in \P_c(S, V)$ is a triangulation, then clearly $\mc I$ is analytic around $\p$. Now we analyze the regularity of $\mc I$ in a more general situation.


\begin{lemma}
\label{subdiv}
Let $\p=(\rho,\tau,f) \in \ol{\P}_c(S, V)$, and $\ms C$ be its face cellulation on $(S, V(\p))$. Then there exists a neighborhood $U$ of $\p$ in $\ol{\P}_c(S, V)$ such that for all $\p' \in U$ its face cellulation is a subdivision of $\ms C$ (with vertex set between $V(\p)$ and $V$).
\end{lemma}

\begin{proof}
Let $e$ be an edge of $\CH(\p)$ with vertices $v$ and $w$. We need to show that for all $\p'$ close enough to $\p$, $e$ is an edge of the face cellulation of $\CH(\p')$. Suppose the converse. Fix a lift $\tilde \p \in \widetilde{\P}(S, V)$. Then there exists a sequence $\tilde \p_i$ converging to $\tilde \p$ for which $e$ is not an edge of $\CH(\p_i)$,  where $\p_i$ is the projection of $\tilde \p_i$. Denote the respective vertex marking maps by $\tilde f_i$, $\tilde f$. Fix a lift $\tilde e$ of $e$ to $\tilde S$ with endpoints $\tilde v, \tilde w \in \tilde V$. For each $i$ there exists a triple of points $p^1_i, p^2_i, p^3_i \in \tilde f_i(\tilde V\backslash\{\tilde v_i,\tilde w_i\})$ such that the convex hull of $p^1_i, p^2_i, p^3_i$ intersects the geodesic segment between $\tilde f_i(\tilde v)$ and $\tilde f_i(\tilde w)$. Up to passing to a subsequence, each sequence $p^j_i$ converges to a point $p^j \in \big(\tilde f(\tilde V\backslash\{\tilde v, \tilde w\}) \cup \pt_\infty^0 \R^{2,1}\big)$ by Lemma~\ref{lem:orbite infinie}, and the convex hull of the points $p^j$ intersects the segment between $\tilde f(\tilde v)$, $\tilde f(\tilde w)$. This contradicts to that the latter is an edge of $\widetilde \CH(\p)$.
\end{proof}

Since there are only finitely many cellulations subdividing $\ms C$, from the definition of the topologies, it is evident that $\ol{\mc I}$ is continuous. We go further and show

\begin{lemma}
\label{diff}
The map $\mc I$ is $C^1$.
\end{lemma}

The proof is reminiscent to similar proofs in \cite{roman-dualhyp, roman-hyp}.

\begin{proof}
Pick $\p \in \P_c(S, V)$ with a lift $\tilde \p \in \widetilde{\P}(S, V)$. Let $\ms C$ be the boundary cellulation of $\CH(\p)$. By Lemma~\ref{subdiv}, there exists a small enough neighborhood $\tilde U$ of $\tilde \p$ such that it projects onto a neighborhood $U$ of $\p$ in $\P_c(S, V)$ and for every $\p' \in U$ the edges of $\ms C$ remain to be the edges of $\CH(\p')$. Let $\ms T$ be a triangulation subdividing $\ms C$. Denote $\mc I(\p)$ by $d$. If $e \in E(\ms T)$ is an edge of $\ms C$, then its length $l_e$ is an analytic function over $U$. We need to deal with those edges of $\ms T$ that are not edges of $\ms C$.

Choose such an edge $e$, and let $Q$ be a face of $\widetilde \CH(\p)$ containing a realization of $e$, $V_Q$ be the set of its vertices (which are considered here as points in $\R^{2,1}$). We orient $Q$ with respect to the orientation of $\widetilde \CH(\p)$. Pick a small enough perturbation of the position of the points in $V_Q$ such that all boundary edges of $Q$ remain to be edges of $\CH(V_Q)$, and all faces of $\CH(V_Q)$ remain spacelike. Let $U_Q \subset (\R^{2,1})^{V_Q}$ be a small neighborhood of the initial positions of the points in $V_Q$, where this holds. If $\CH(V_Q)$ has non-empty interior, $\pt \CH(V_Q)$ is divided by the cycle of boundary edges of $Q$ into two domains. One of them, chosen with the help of the orientation and called \emph{the lower part}, is considered as a deformation of $Q$. Let $l_{e, Q}$ be the length of $e$ in this realization, considered as a function over $V_Q$.

Provided that $\tilde U$ is sufficiently small, there is a smooth map $q: \tilde U \rar U_Q$ sending $\tilde \p' \in \tilde U$ to the corresponding positions of points of $V_Q$. The length function $l_e$ over $\tilde U$ coincides with $l_{e, Q} \circ q$. Hence, it is enough to show that $l_{e, Q}$ is a $C^1$-function of the positions of the points in $V_Q$. Abusing the notation, we denote $l_{e, Q}$ by $l_e$ until the end of the proof.

Let $\ms T_1, \ldots, \ms T_r$ be all the triangulations of $Q$. For every $i=1,\ldots, r$ and every configuration in $U_Q$, we consider the triangulated polyhedral surface $Q_i$ in $\R^{2,1}$ determined by the triangulation $\ms T_i$ and the positions of vertices. By decreasing $U_Q$ if necessary, we may assume that all face triangles of all these surfaces are spacelike,  and all boundary angles (of the surfaces in the intrinsic metric) are less than $\pi$. Then for every such surface we can determine the length of $e$ in this surface. This determines $r$ functions over $U_Q$, which we denote by $l_{e,1}, \ldots, l_{e,r}$. We note that since for every $i$ the lengths of all edges of $Q_i$ are analytic over $U_Q$, and $l_{e, i}$ is a analytic function of these lengths, we get that $l_{e,i}$ are analytic over $U_Q$.

For every configuration in $U_Q$, the lower part of $\CH(V_Q)$ coincides with one of $Q_i$. This provides a decomposition of $U_Q$ into cells $U_{Q, 1}, \ldots U_{Q, r}$, which are analytic manifolds with corners, and $l_e$ coincides with $l_{e, i}$ over $U_{Q, i}$. We need to check that the differentials of the corresponding $l_{e, i}$ coincide on the common points of $U_{Q, i}$. It is enough to prove this at the initial configuration.
\begin{cla}
\label{vertical}
The differentials of all $l_{e, i}$ at the initial configuration coincide.
\end{cla}

Let $\Pi$ be the oriented plane containing $Q$. Introduce a coordinate system in $\R^{2,1}$, where $\Pi$ is a coordinate plane, which we call \emph{horizontal}. The last coordinate is accordingly called \emph{vertical}. It is clear that at the initial configuration, all the horizontal derivatives of all $l_{e, i}$ coincide. We now show that the vertical derivatives coincide too, and, moreover, are equal to zero. 

Let $v_1, v_2$ be two points in $\R^{2,1}$ in spacelike position. If $v_1$ moves orthogonally to the segment connecting them, then it is easy to check that the derivative of the distance between $v_1$ and $v_2$ is zero. This implies that at the initial configuration, the vertical derivatives of the length of any edge of $Q_i$ are zero. Since for every $i$ the function $l_{e,i}$ is an analytic function of the edge-lengths of $Q_i$, its vertical derivatives are zero too. This finishes the proof.
\end{proof}

\section{Proof of Theorem~\ref{thmI'}}


\subsection{A topological interlude}
\label{interlude}

Our space $\P_c(S, V)$ is not simply connected. We are interested to give an interpretation to elements of its universal covering. To this purpose, we need a topological tool, Lemma~\ref{evaluation}, which we will prove in this section.

First, we describe the setting. Let $V \subset S$ be a finite subset of size $n$ and $S^{\ast V}$ be the space of injective maps $c: V \hookrightarrow S$,  i.e., the configuration space of $n$ points on $S$, where the points are marked by $V$. The inclusion map allows us to consider $V$ itself as a point of $S^{\ast V}$. Denote by $H_0(S)$ the group of self-homeomorphisms $S \rar S$ isotopic to the identity, denote by $H_0(S, V)$ the subgroup of self-homeomorphism that fix $V$ and are isotopic to the identity relative to $V$, and we denote by $\mc C(S, V)$ the space of left cosets $H_0(S)/H_0(S, V)$. We consider both homeomorphisms groups with the compact-open topology and the coset space with the induced topology. We have the evaluation map $H_0(S) \rar S^{\ast V}$ given by the evaluation of $h$ at $V$, which is clearly continuous. It factors through $H_0(S, V)$ as a map ${\rm ev}_{V}: \mc C(S, V) \rar S^{\ast V}.$ This map is a local homeomorphism, see e.g.~\cite[Chapter 4.1]{birman} (note that in \cite[Chapter 4.1]{birman} the homeomorphisms are not required to be isotopic to the identity, however, the argument does not change). Hence, $\mc C(S, V)$ is a manifold of dimension $2n$.

Every homeomorphism $h: S \rar S$ determines a homeomorphism $h_V: S^{\ast V} \rar S^{\ast V}$ by $c \mapsto h\circ c$, and if $h$ is isotopic to the identity, then so is $h_V$. Every self-homeomophism of $S^{\ast V}$ isotopic to the identity admits a well-defined lift to the universal cover via the homotopy-lifting property applied to an isotopy. So every such $h_V$ lifts to a homeomorphism $\tilde h_V: \widetilde{S^{\ast V}} \rar \widetilde{S^{\ast V}}$. Fix a lift $\hat V$ of $V$ to the universal covering $\widetilde{S^{\ast V}}$. Now we define a new evaluation map $H_0(S) \rar \widetilde{S^{\ast V}}$, given by the evaluation of $\tilde h_V$ at $\hat V$, which is clearly continuous. For $h \in H_0(S, V)$, the map $\tilde h_V$ fixes $\hat V$, hence the evaluation map factors through $H_0(S, V)$ as $$\wti{\rm ev}_{\tilde V}: \mc C(S, V) \rar \widetilde{S^{\ast V}}.$$

\begin{lemma}
\label{evaluation}
The evaluation map $\wti{\rm ev}_{\hat V}$ is a homeomorphism.
\end{lemma}

As $\wti{\rm ev}_{\hat V}$ is a continuous map between two manifolds of the same dimension, it is enough to prove that it is bijective. We will show it by induction on $n$. We need to establish few auxiliary facts. First, consider $v \in V$ and $W:=V\backslash\{v\}$. When $|V|>1$, we have the forgetful map $\theta_V: S^{\ast V} \rar S^{\ast W}$, so for $c: W \hookrightarrow S$ the fiber of $\theta_V$ over $c$ is naturally identified with  $S \backslash c(W)$. It is known that $\theta_V$ is a fiber bundle (with fiber homeomorphic to $S \backslash W$), see~\cite[Theorem 1.1]{FH}. We can then extend this bundle to the universal covering.

\begin{lemma}
\label{univcover}
The universal covering $\wti{S^{\ast V}}$ is homeomorphic to an open ball of dimension $2n$ and there exists a fiber bundle $\tilde\theta_V: \wti{S^{\ast V}} \rar \wti{S^{\ast W}}$ extending $\theta_V$.
\end{lemma}

\begin{proof}
We will also prove this by induction on $n$. For $n=1$, $S^{\ast V} \cong S$ and $\wti{S^{\ast V}} \cong \tilde S$, hence we have the base case for the first statement. Consider now the pull-back of the bundle $\theta_V$ to $\wti{S^{\ast W}}$. It has fiber homeomorphic to $S\backslash W$. If we know that $\wti{S^{\ast W}}$ is a topological ball, then the pull-back bundle is trivial, and we can extend it further to a trivial bundle over $\wti{S^{\ast W}}$ with fiber homeomorphic to $\wti{S\backslash W}$. The total space of the latter bundle is a topological ball of dimension $2n$ as it is a trivial bundle over a ball of dimension $2n-2$ with fiber homeomorphic to a 2-dimensional ball. Particularly, the total space is simply connected. It also covers $S^{\ast V}$ by construction, thus, this covering is the universal covering of $S^{\ast V}$. This gives the base case for the second statement and the induction steps for both statements.
\end{proof}

For the next two statements we need yet more notation. Let $W \subset S$ be a finite subset, possibly empty, and $h \in H_0(S, W)$ (which is just $H_0(S)$ if $W$ is empty). It then also lifts to a self-homeomorphism of $\wti{S\backslash W}$, which we denote by $\tilde h_{\backslash W}$. The construction in Lemma~\ref{univcover} shows that for a lift $\hat W$ of $W$ to $\wti{S^{\ast W}}$ the fiber of $\tilde\theta_V$ over $\hat W$ is naturally identified with $\wti{S\backslash W}$, in the sense that for every $h \in H_0(S, W)$ the restriction of $\tilde h_W$ to the fiber over $\hat W$ preserves the fiber and coincides with $\tilde h_{\backslash W}$ on $\wti{S\backslash W}$.

\begin{lemma}
\label{isotopy1}
Let $\tilde p, \tilde q \in \wti{S\backslash W}$ be two points. Then there exists $h \in H_0(S, W)$ such that $\tilde h_{\backslash W}$ sends $\tilde p$ to $\tilde q$.
\end{lemma}

\begin{proof}
Let $\tilde p, \tilde q$ project to points $p, q \in S\backslash W$. Consider an oriented path $\tilde \psi$ connecting $\tilde p$ and $\tilde q$ in $\wti{S\backslash W}$, let $\psi$ be its projection to an oriented path between $p$ and $q$ in ${S\backslash W}$. We make a homotopy of $\psi$ on $S \backslash W$ so that $\psi$ becomes the concatenation of finitely many simple closed curves $\psi_i$ based at $p$ and of a simple arc $\psi_a$ connecting $p$ and $q$. 

We now take ``point-pushing'' homeomorphisms along every $\psi_i$ and along $\psi_a$. For the loops $\psi_i$ such a map is obtained as follows. We take a small embedded strip along $\psi_i$ in ${S\backslash W}$ and consider a homeomorphism $h_i: S \rar S$ that is the identity outside the strip, that is also the identity on $\psi_i$, but that makes one full twist inside the strip in the direction determined by the orientation of $\psi_i$. Every $h_i$ is isotopic to the identity so that $W$ is fixed and $p$ moves along $\psi_i$ during the isotopy. For $\psi_a$ we consider a small embedded disk around $\psi_a$ in ${S\backslash W}$ and a homeomorphism $h_a: S \rar S$ that is the identity outside the disk and that sends $p$ to $q$. It is also isotopic to the identity by an isotopy fixing $W$ and moving $p$ along $\psi_a$. We denote by $h$ the composition of all $h_i$ and $h_a$. It is clear that it is isotopic to the identity by an isotopy fixing $W$ and is moving $p$ along $\psi$. Thereby $\tilde h_{\backslash W}$ sends $\tilde p$ to $\tilde q$.
\end{proof}

\begin{lemma}
\label{isotopy2}
Let $h \in H_0(S, W)$, let $v \in S \backslash W$ be a point and let $\tilde v \in \wti{S\backslash W}$ be a lift of $v$. Assume that the lift $\tilde h_{\backslash W}$ fixes $\tilde v$. Then $h \in H_0(S, V)$ for $V:=W\cup\{v\}$.
\end{lemma}

\begin{proof}
Let $h_t$, $t \in [0,1]$, be a path of homeomorphisms fixing $W$ such that $h_0$ is the identity and $h_1=h$. Since $\tilde h_0(\tilde v)=\tilde h_1(\tilde v)$, the loop $h_t(v)$, $t \in [0,1]$, is contractible in $S \backslash W$. Let $s: [0,1]^2 \rar S\backslash W$ be a homotopy sending $s(t, 0):=h_t(v)$ to the constant map $s(t,1):=v$ on $S\backslash W$ so that $s(0,\tau)=s(1,\tau)=v$ for all $\tau \in [0,1]$. Consider the space $S \times [0,1]$ and the map $H: S \times [0,1] \rightarrow S$, $H(p, t):=h_t(p)$. The map $(t, \tau) \mapsto (s(t, \tau), t)$ sends $[0,1]^2$ to $(S\backslash W) \times [0,1]$, and the image contains $\{v\} \times [0,1]$ and $H^{-1}(v)$. Turns out that the paths $\{v\} \times [0,1]$ and $H^{-1}(v)$ are homotopic in $S \times [0,1]$ by a homotopy respecting the levels $S \times \{t\}$ and fixing pointwise the set $W \times [0,1]$. 
Then there exists a homeomorphism $F: S  \times [0,1] \rar S\times [0,1]$ respecting the levels, fixing the levels $S \times \{0\}$ and $S \times \{1\}$ pointwise, fixing the set $W \times [0,1]$ pointwise, and sending $\{v\} \times [0,1]$ to $H^{-1}(v)$. The map $H \circ F$ constitutes the desired isotopy from $id$ to $h$ fixing $V$. 
\end{proof}

\begin{proof}[Proof of Lemma~\ref{evaluation}]
We prove the following two statements by induction: (1) for any $X \in \wti{S^{\ast V}}$ there exists $h \in H_0(S)$ such that $\tilde h_V(\hat V)=X$; and (2) for any $h \in H_0(S)$ such that $\tilde h_V$ fixes $\hat V$, we have $h \in H_0(S, V)$. It is easy to see that this shows that $\wti{\rm ev}_{\hat V}$ is bijective.

The base cases are trivially given by Lemmas~\ref{isotopy1} and~\ref{isotopy2}. For the induction step we employ the (trivial) fiber bundle $\tilde\theta:=\tilde\theta_V: \wti{S^{\ast V}} \rar \wti{S^{\ast W}}$ from Lemma~\ref{univcover}. 

(1) By the induction hypothesis, there exists $h' \in H_0(S)$ such that $\tilde h_W'\circ\tilde\theta(\hat V)=\tilde\theta(X)$. Next, by Lemma~\ref{isotopy1}, there exists $h'' \in H_0(S, h'(W))$ such that $\tilde h''_{\backslash h'(W)}$ sends $\tilde h'_V(\hat V)$ to $X$ where the fiber of $\tilde \theta$ containing $X$ is naturally identified with $\wti{S \backslash h'(W)}$. It follows that $h:=h''\circ h'$ fits.

(2) By the induction hypothesis, we have $h \in H_0(S, W)$. Since the fiber of $\tilde \theta$ containing $\hat V$ is naturally identified with $\wti{S \backslash W}$, Lemma~\ref{isotopy2} implies then that $h \in H_0(S, V)$.
\end{proof}
\begin{remark}
\label{decktr}
We need to identify the group of deck transformations 
of $\wti{S^{\ast V}}$, which is the \emph{pure braid group of $S$, i.e., the fundamental group of $S^{\ast V}$}. This is already easy to derive from the discussion above. Let $H_\star(S, V)$ be the group of homeomorphisms fixing $V$ and isotopic to identity. Then $H_0(S, V)$ is a normal subgroup of $H_\star(S, V)$. Consider the quotient group 
$H_\star(S, V)/H_0(S, V)$. It clearly acts on $\mc C(S, V)$ from the right by deck transformations, which provides a homomorphism 
\begin{equation}\label{eq:identitf}H_\star(S, V)/H_0(S, V) \rar \pi_1 S^{\ast V}~.\end{equation} It remains to notice that Lemma~\ref{isotopy1} clearly implies that this homomorphism is surjective, while Lemma~\ref{isotopy2} establishes its injectivity. Note that the fact that \eqref{eq:identitf} is an isomorphism can also be deduced from a theorem of J. Birman, see \cite[Theorem 4.2]{birman}.
\end{remark}

\begin{remark}
Lemma~\ref{evaluation} is probably already known, but we were unable to find a precise reference. We provided a proof for completeness. Another way to prove it would have been to start from  (the adaptation of) a classical result of J. Birman saying that
\begin{equation*}\label{eq: fb 1}H_\star(S,V)\rar H_0(S)  \xrightarrow[]{{\rm ev}} S^{\ast V}~ \end{equation*} 
is a fiber bundle (it is immediate to adapt the proof of Theorem 4.6 in \cite{primer}). Then Lemma~\ref{evaluation} would follow from general topological arguments.
\end{remark}
\subsection{The induced metric map on fibers}\label{sec:proof thm I'''}

We have a projection map $\pi: \P_c(S,V)\to \mc T(S)$, which is obtained from the projection map $\mu$ in \eqref{eq:def hat pi} composed with the natural projection from $T \mc T(S) $ onto $\mc T(S)$.  For $\rho \in \mc T(S)$ we will denote  $\pi^{-1}(\rho)$ by $\P_c(\rho, V)$. In other terms, we denote by $\P_c(\rho, V) \subset \P_c(S, V)$ the space of spacetimes with marked points and with linear holonomy given by $\rho$. Other related notation from Section~\ref{sec:markedpoints} restricts to this case straightforwardly. 



We denote the restriction of $\mc I$ to $\P_c(\rho, V)$ by $\mc I_\rho$. Our aim is to show  Theorem~\ref{thm:1}, which implies Theorem~\ref{thmI'}.

We have the restriction of $\mu$ to $\P(\rho, V)$, which we denote by $\mu_\rho$,
$$\mu_\rho:\P(\rho, V)\to T_\rho\mc T(S)= \mc H^1(\rho)~.$$
 We now introduce on $\P(\rho, V)$ a structure of fiber bundle of special type over $\mc H^1(\rho)$ with fiber $(S \times \R_{>0})^V$. We first suppose that $|V|=1$. This is a special case, and to distinguish it, we will denote $\P(\rho, V)$ rather by $\P(\rho, v)$, meaning that $V=\{v\}$, and denote $\mu_\rho$ by $\mu_{\rho, v}$.

\begin{lemma}
\label{fbundle}
There exists a fiber bundle atlas $\{(U, \phi_{U})\}$, $U \subset \mc H^1(\rho)$, for $\mu_{\rho, v}: \P(\rho, v) \rightarrow \mc H^1(\rho)$ with
$$\phi_{U}: \mu^{-1}_{\rho, v}(U) \rightarrow U \times S \times \R_{>0}$$
such that \\
(1) for every $\tau \in U$, $r \in \R_{>0},$ the set $\phi_{U}^{-1}(\tau, S, r)$ is a $\CT$-level surface in $\Omega^+(\rho, \tau)$;\\
(2) for every $\tau \in U$, $p \in S,$ the set $\phi_{U}^{-1}(\tau, p, \R_{>0})$ is a gradient line of $\CT$ parametrized by $\CT$.
\end{lemma}

\begin{proof}
Lift $\rho$ to an element of $\mc R(S)$. We denote by $\tilde \mu_{\rho, v}$ the natural projection $\tilde\mu_{\rho, v}: \tilde \P(\rho, v) \rightarrow \mc Z^1(\rho)$, which is a lift of $\mu_{\rho, v}$.
Let $\tilde U$ be an open bounded set in $Z^1(\rho)$. In \cite[Section 6]{bonsante} Bonsante showed that there exists a continuous map
$$\Phi_{\tilde U}: \tilde U \times \tilde S \times \R_{>0} \rar \R^{2,1}$$
such that for every $\tau \in \tilde U$ we have $\Phi_{\tilde U}(\tau, \tilde S, \R_{>0})=\tilde\Omega^+(\rho, \tau)$ and\\
(0) for every $\tau \in \tilde U$, the map $\Phi_{\tilde U}$ is $\pi_1 S$-equivariant, i.e., for every $\gamma \in \pi_1 S$, $p \in \tilde S$ and $r \in \R_{>0}$ we have $\Phi_{\tilde U}(\tau, \gamma p, r)=\rho_\tau(\gamma)\Phi_{\tilde U}(\tau, p, r)$,\\
(1) for every $\tau \in \tilde U$, $r \in \R_{>0}$ the set $\Phi_{\tilde U}(\tau, \tilde S, r)$ is a $\CT$-level surface in $\tilde\Omega^+(\rho, \tau)$;\\
(2) for every $\tau \in \tilde U$, $p \in \tilde S$ the set $\Phi_{\tilde U}(\tau, p, \R_{>0})$ is a gradient line of $\CT$ parametrized by $\CT$.

Now we consider the map $\zeta \times \Phi_{\tilde U}: \tilde U \times \tilde S \times \R_{>0} \rightarrow \tilde U \times \R^{2,1} $, where $\zeta$ is the projection onto the first factor. This is a homeomorphism onto the image, which is $\tilde \mu^{-1}_{\rho, v}(\tilde U)$. So the inverses of these maps constitute an equivariant fiber bundle atlas of the desired form for $\tilde\mu_{\rho, v}$. In turn, this gives the desired atlas for $\mu_{\rho, v}: \P(\rho, v) \rar \mc H^1(\rho)$.
\end{proof}

Our fiber bundle given by Lemma~\ref{fbundle} has the structure group $H_0(S)$, where we consider it acting on $S \times \R_{>0}$ by homeomorphisms on the first factor and trivially on the second factor. Since the space $\mc H^1(\rho)$ is contractible, this fiber bundle admits a trivialization
\begin{equation}
\label{trivialization}
\P(\rho, v) \cong \mc H^1(\rho) \times S \times \R_{>0}~,
\end{equation}
where for every $U \subset \mc H^1(\rho)$, the trivialization $U \times S \times \R_{>0}$ belongs to the maximal atlas for $\mu_{\rho, v}$ satisfying the conditions (1) and (2) from Lemma~\ref{fbundle}. For what follows, we fix such a trivialization. This endows every Cauchy surface $\Sigma \subset \Omega^+(\rho, \tau)$ with a homeomorphism $\Sigma \rar S$ given by the projection to $S$ in (\ref{trivialization}).


For $|V|=n>1$, we note that after an enumeration of points of $V$, the $n$-th power of the fiber bundle $\mu_{\rho, v}$ is naturally diffeomorphic to the space $\P(\rho, V)$, with the diffeomorphism preserving the projection to $\mc H^1(\rho)$. This endows $\mu_\rho: \P(\rho, V) \rar \mc H^1(\rho)$ with the desired fiber bundle structure with fiber $(S \times \R_{>0})^V$. The trivialization (\ref{trivialization}) then gives a trivialization
$$\P(\rho, V) \cong \mc H^1(\rho) \times S^V \times \R_{>0}^V~.$$

Consider a subset $\P^\ast(\rho, V)\subset \P(\rho, V)$ consisting of all configurations such that for every distinct $v, w\in V$, the points $f(v)$ and $f(w)$ do not belong to the same gradient line of $\CT$. This is clearly an open subset of $\P(\rho, V)$, and we have a trivialization 
\begin{equation}
\label{trivialization*}
\P^\ast(\rho, V) \cong \mc H^1(\rho) \times S^{\ast V} \times \R_{>0}^V~.
\end{equation}
We consider the universal cover $\P^\sharp(\rho, V)$ of $\P^\ast(\rho, V)$, which then has a trivialization
\[\P^\sharp(\rho, V) \cong \mc H^1(\rho) \times \wti{S^{\ast V}} \times \R_{>0}^V~.\]
Finally, we fix an evaluation homeomorphism from Section~\ref{interlude}, allowing us to identify $\wti{S^{\ast V}} \cong \mc C(S, V)$. 


It is clear that $\ol \P_c(\rho, V) \subset \P^\ast(\rho, V)$. We denote by $\ol \P^\sharp_c(\rho, V)$ its full preimage in the universal cover $\P^\sharp(\rho, V)$, and denote by $\P_c^\sharp(\rho, V)$ the full preimage of $\P_c(S, V)$. Let $\p^\sharp$ be an element of $\P^\sharp_c(\rho, V)$ and $\p$ be its projection to $\P_c(\rho, V)$. Basically, for us $\p^\sharp$ is just $\p$ equipped with one additional piece of data: a homeomorphism from $\pt\clconv(\p)$ to $(S, V)$ respecting the vertex marking map and defined up to isotopy fixing $V$. To see this, first, observe that trivialization (\ref{trivialization}) endows $\pt\clconv(\p)$ with a homeomorphism $\chi$ to $S$. Second, the choice of the evaluation map gives an element of $\mc C(S, V)$, i.e., a class in $H_0(S)/H_0(S, V)$. If $h \in H_0(S)$ is a representative of this class, then $h^{-1}\circ \chi$ is a desired homeomorphism from $\pt\clconv(\p)$ to $S$, and $h^{-1}\circ \chi \circ f$ is the identity on $V$.


Now we denote by $\M_-^\sharp(S, V)$ the respective space of metrics up to isometry isotopic to identity relative to $V$. 
The forgetful map $\M_-^\sharp(S, V) \rar \M_-(S, V)$ is a covering map, and we will soon see that it is the universal covering. It is straightforward from the definitions that the group of deck transformations can be identified with the group $\pi_1 S^{\ast V}$, which was also mentioned in Remark~\ref{decktr}. 

We can define the lift of the induced metric map
\[\mc I^\sharp_\rho: \P_c^\sharp(\rho, V) \rar \M_-^\sharp(S, V)~.\]

From Section~\ref{sec: the realization map} and Lemma~\ref{open}, we know that $\P_c^\sharp(\rho, V)$ and $\M_-^\sharp(S, V)$ are manifolds of the same dimension $6\mathbf{g}-6+3n$. We will need more information about their topologies, that are given by the two following lemmas.

\begin{lemma}
\label{lem: p rho connect}
$\P_c^\sharp(\rho, V)$ is connected.
\end{lemma}
\begin{proof}
Take $\p_0^\sharp, \p_1^\sharp \in \P_c^\sharp(\rho, V)$. First note that $\P^\sharp(\rho, V)$ is clearly connected. We connect $\p_0^\sharp$ and $\p_1^\sharp$ by a path $\p_t^\sharp$, $t \in [0,1]$, in $\P^\sharp(\rho, V)$. Our first goal is to deform $\p_t^\sharp$ to a path in $\ol{\P}_c^\sharp(\rho, V)$. Denote the projection of $\p_t^\sharp$ to $\P^\ast(\rho, V)$ by $\p_t=(\rho, \tau_t, f_t)$. 



For $v \in V$ we define the set 
\[O_v:=\{t \in [0,1]: f_t(v) \in \inter(\CH(\p_t))\}~.\]
Clearly, $O_v$ is an open subset of $(0,1)$, hence it consists of countably many open intervals. For every such interval $I$ with endpoints $t_1$ and $t_2$, we have $f_{t_1}(v) \in \pt \CH(\p_{t_1})$ and $f_{t_2}(v) \in \pt \CH(\p_{t_2})$. Now we decrease the cosmological time of $v$ to get $f_t(v) \in \pt \CH(\p_t)$. This will deform our path $\p_t$ to another continuous path, and we lift it to a deformation of the path $\p_t^\sharp$. We can do this for every $v \in V$ and every such interval $I$.  After all the modifications, $\p_t^\sharp$ is a path in $\ol{\P}_c^\sharp(\rho, V)$. When $f_t(v)$ is not a vertex of $\CH(\p_t)$, it remains again to decrease slightly the cosmological time of $f_t(v)$. After doing this for all $t$, the sets $f_t(V)$ get to strictly convex positions, thus $\p_t^\sharp$ becomes a path in $\P_c^\sharp(\rho, V)$.
\end{proof}

\begin{lemma}\label{lem:dim flat}
$ \M_-^\sharp(S,V)$ is connected and simply connected.
\end{lemma}
 \begin{proof}
We denote by $\mc T(S, V)$ the Teichm\"uller space with marked points, i.e., the space of conformal structures on $S$ up to conformal maps isotopic to identity relative to $V$. It is classical that $\mc T(S, V)$ is homeomorphic to an open ball of dimension $6\mathbf{g}-6+2n$. We denote by $X^V$ a subset of $\R_{<0}^V$ of tuples of negative numbers with sum $2-2\mb g.$ Consider the trivial bundle $\mc T(S, V) \times X^V \times \R_{>0}$. To $d \in \M_-^\sharp(S,V)$ we associate the following data: the underlying conformal structure as an element of $\mc T(S, V)$, the tuple of singular curvatures and the area. This gives a map $ \M_-^\sharp(S,V) \rar \mc T(S, V) \times X^V \times \R_{>0}$, which is clearly continuous. The work of Troyanov~\cite{troyanov} implies that it is injective and surjective, hence, from the invariance of the domain, it is a homeomorphism.

We remark that another similar proof but using discrete conformality instead of conformality follows from the work~\cite{gu-luo-sun-wu}.
 
%
 \end{proof}

The following result will be proven in Section~\ref{sec:iskhakov}. 
In a generic case when the face decomposition is a triangulation, it is Theorem \ref{thm:iskhakov} of Iskhakov.

\begin{prop}
\label{infrig}
The differential $d\mc I_\rho$ is non-degenerate.
\end{prop}

In Section~\ref{sec:compactness}, we will prove the following compactness result.

\begin{prop}
\label{compact}
Let $\p_i=(\rho, \tau_i, f_i)$ be a sequence in $\P_c(\rho, V)$ such that the induced metrics $d_i \in \mc \M_-(S, V)$ converge to a metric $d \in \mc \M_-(S, V)$. Then, up to a subsequence, $\p_i$ converge to $\p \in \P_c(\rho, V)$.
\end{prop}

%
%
%

Putting all these results together, we can now give the proof of Theorem~\ref{thm:1}.

\begin{proof}[Proof of Theorem~\ref{thm:1}]
The spaces $\P_c(\rho, V)$ and $\mc \M_-(S, V)$ are manifolds of the same dimension.
From  Lemma~\ref{diff}, Proposition~\ref{infrig} and the Inverse Function Theorem, $\mc I_\rho$ is a local $C^1$-diffeomorphism.
 Due to Lemma~\ref{compact}, $\mc I_\rho$ is proper. Hence the image of $\mc I_\rho$ is closed and open, so $\mc I_\rho$ is surjective, as, by Lemma~\ref{lem:dim flat}, the target is  connected. Then $\mc I_\rho$ is a covering map. Consider now a commutative diagram
 \begin{center}
\begin{tikzcd}
\P_c^\sharp(\rho, V) \arrow[r,"\mc I^\sharp_\rho"] \arrow[d]
& \mc \M_-^\sharp(S, V) \arrow[d] \\
\P_c(\rho, V) \arrow[r, "\mc I_\rho"]
& \mc \M_-(S, V)
\end{tikzcd}
\end{center}
We know that the vertical arrows are covering maps, and so is the lower horizontal arrow. Since all spaces are connected, it follows that $\mc I^\sharp_\rho$ is also a covering. 
By Lemma~\ref{lem: p rho connect}, the domain is connected, and by Lemma~\ref{lem:dim flat}, the target is simply connected, thus  $\mc I^\sharp_\rho$ is a homeomorphism. The groups of deck transformations of the vertical coverings are both naturally isomorphic to the group $\pi_1 S^{\ast V}$, see Remark~\ref{decktr}, and the map $\mc I^\sharp_\rho$ is clearly equivariant with respect to it. As it induces an isomorphism of the deck transformation groups, it follows that $\mc I_\rho$ is also a homeomorphism. 

\end{proof}

\subsection{Infinitesimal rigidity}\label{sec:iskhakov}

In this section we prove Proposition~\ref{infrig}.

Let $\p=(\rho,\tau,f) \in \P_c(\rho, V)$ and $\dot \p \in T_\p\P_c(\rho, V)$ be such that $d\mc I_\rho(\dot \p)=0$. Denote $\mc I(\p)$ by $d$ and denote a hyperbolic metric on $S$ with holonomy $\rho$ by $h$. Choose a triangulation $\ms T$ subdividing the face cellulation of $\p$. Pick a small enough neighborhood $U$ of $\p$, and for every $\p'=(\rho, \tau', f') \in U$ consider the polyhedral surface in $\Omega(\rho, \tau')$ with the combinatorics given by $\ms T$. Provided that $U$ is small enough, such a surface is an embedded Cauchy surface, though possibly non-convex. Its edge-lengths and oriented dihedral angles are analytic functions over $U$. By Claim~\ref{vertical} in the proof of Lemma~\ref{diff}, $\dot \p$ induces zero derivatives of all edge-lengths. We will prove that the induced variations of the (oriented) dihedral angles of $\ms T$ are zero. One can see that this implies that $\dot \p$ is zero.  

Let $\ms G$ be the trivalent dual graph to $\ms T$. The Gauss map of $\partial\CH(\p)$ realizes $\ms G$ as a geodesic graph on $(S, h)$, although some edges, which correspond to the dihedral angles equal to $\pi$, have zero length. From now and until the end of this section, we will work in $(S, h)$. We will change now our notation, and until the end of the section we will denote by $V$, $E$ and $F$ the sets of vertices, edges and faces of $\ms G$ rather than of $\ms T$. 
 Note that every connected component of the subgraph of the  edges of zero length is a tree. We consider all the edges oriented arbitrarily. By corner we will mean an angle of a face of $\ms G$, more exactly, we will use the word \emph{corner} to denote it as a combinatorial object, and the word \emph{angle} to denote its angle-measure. We denote by $Q$ the set of corners of $\ms G$. 

Consider a pull-back $\tilde {\ms G}$ of $\ms G$ to $\tilde S \cong \H^2 \subset \R^{2,1}$. We denote by $\tilde V, \tilde E$ and $\tilde Q$ its sets of vertices, edges and corners respectively.  For an edge $\tilde e \in \tilde E$ we will abuse the notation and will denote its unit oriented normal (with respect to the chosen orientation of the edge and the orientation of the surface), considered as a point in $\R^{2,1}$, also by $\tilde e$. Note that, because the angles between edges are known, points of $\R^{2,1}$ defined by edges of zero length are also well-defined. We will use a similar convention for the vertices $\tilde v \in \tilde V$, denoting by $\tilde v$ also the respective point in $\H^2 \subset \R^{2,1}$. For a corner $\tilde q$ we denote by $\alpha_{\tilde q}$ its angle. We have all $\alpha_{\tilde q} \in (0, \pi)$. Particularly, for any two adjacent edges, the (non-oriented) lines determined by them are distinct. For an edge $\tilde e$ we denote by $l_{\tilde e}$ its length. 

The deformation $\dot \p$ induces infinitesimal deformations $\dot{\tilde e}$ on all edges $\tilde e \in \tilde E$. Here we consider the latter as points in $\R^{2,1}$, and we consider $\dot{\tilde e}$ as vectors in $\R^{2,1}$. These deformations are equivariant with respect to $\rho$. Because $d\mc I_\rho(\dot \p)=0$, the induced deformations $\dot\alpha$ of all corner angles are zero. Our aim is to show that then the induced deformations of all lengths are zero too. The remainder of the present section is aimed to prove a slight generalization of Iskhakov's Theorem~\ref{thm:iskhakov}.

\begin{remark} 
In the case when $\ms G$ does not have edges of zero length, Theorem~\ref{thm:iskhakov} was basically proven by Iskhakov in \cite{iskhakov}. We need to make two remarks. First, we need to modify Iskhakov's proof to allow  edges of zero length. Second, unfortunately it seems that there is a mistake in Iskhakov's proof. Iskhakov claims that his result holds also when the face angles can be bigger than $\pi$. The first step of his proof is to eliminate enough edges to leave only one face, and then he gives a proof only for this case. Because we can do the elimination in such a way that any given edge can be kept, the one-face case implies the general case. In the bottom of p.56 in  \cite{iskhakov} Iskhakov writes that his coefficients $\lambda$ are positive numbers, which is crucial for him to finish the proof. However, this is wrong exactly when there are non-convex angles, so Iskhakov's proof does not seem to work in the non-convex case. This also makes the reduction to the one-face case invalid, as it is unclear how to guarantee that the resulting single face is convex, especially with the task to keep any given edge. However, it does not seem hard to fix Iskhakov's proof for convex initial cellulations. Because of these remarks, and because Iskhakov's proof remains formally unpublished, we decided to write an exposition of his proof. Our strategy is slightly different from Iskhakov's, as we do not make a reduction to the one-face case, but all the main ideas remain his.
\end{remark}

Recall that our edges are oriented. For a vertex $\tilde v \in \tilde V$ and an adjacent edge $\tilde e \in \tilde E$, we define the number $\tilde \sigma_{\tilde v\tilde e}$ equal to $1$ if $\tilde v$ is the source of $\tilde e$ and equal to $-1$ if $\tilde v$ is the sink of $\tilde e$. Hence, if $\tilde e$ is adjacent to $\tilde v$ and $\tilde w$, then $\tilde \sigma_{\tilde w\tilde e}=-\tilde \sigma_{\tilde v\tilde e}$. Every corner $\tilde q \in \tilde Q$, considered with clockwise orientation, can be denoted by an ordered pair of edges, which we denote $\tilde e_{\tilde q-}$ and $\tilde e_{\tilde q+}$.  We denote by $\tilde \lambda_{\tilde v\tilde e}$ the number $\tan(\frac{\alpha_{\tilde q_1}}{2})+\tan(\frac{\alpha_{\tilde q_2}}{2})$, where $\tilde q_1$ and $\tilde q_2$ are the two corners determined by the edge $\tilde e$ at the vertex $\tilde v$. 
Note that by convexity, the corner angles are less than $\pi$, which implies that $\tilde \lambda_{\tilde v\tilde e}$ are positive.

For every vertex, edge and corner of $\ms G$, we choose an arbitrary lift to $\tilde S$ and, abusing the notation, use these lifts to identify $V,$ $E$ and $Q$ as subsets of $\tilde V,$ $\tilde E$ and $\tilde Q$ respectively.
 We have the following system of equations:
\[(\mathcal S) \; \; \begin{cases}
\la \dot e, e\ra=0 &\forall e\in E \\
\la \tilde e_{q+}, \dot{\tilde e}_{q-}\ra+\la \dot{\tilde e}_{q+},\tilde e_{q-}\ra=0 &\forall q\in Q\\
\la \sum_{\tilde e \in \tilde E_v}\tilde\lambda_{v\tilde e}\tilde\sigma_{v\tilde e}\dot{\tilde e}, v\ra=0 &\forall v\in V
\end{cases} \]

Here $\langle \cdot,\cdot\rangle$ is the Minkowski bilinear form, and the first type of equations accounts for the fact that $e$ remains to be unit under the deformation. The second type is responsible for that the angles of the corners do not change. The third type comes from the fact that around a vertex $v$, the lines containing elements of $\tilde E_{v}$ (the set of edges of $\tilde {\ms G}$ adjacent to $v$) remain concurrent. Indeed, it is easy to compute that since the oriented normals are coplanar Euclidean unit vectors, we have $\sum_{\tilde e \in \tilde E_v}\tilde\lambda_{v\tilde e}\tilde\sigma_{v\tilde e}{\tilde e}=0$ (see, e.g., \cite[Lemma A.2]{iskhakov}). We want to prove the following

\begin{lemma}\label{lem:system S}
The zero vector is the only solution of $(\mathcal S)$.
\end{lemma}

 Note that the proof below does not assume that the graph is trivalent, moreover this generality in needed for the case of the presence of edges of zero length.
 
System $(\mathcal S)$  is a linear homogeneous system of $|E|+|Q|+|V|$ equations in variables $\dot e$ for $e \in E$.  Indeed, every $\dot{\tilde e}$ for $\tilde e \in \tilde E$ is obtained from some $\dot e$, where $e \in E$, by the application of an isometry of $\R^{2,1}$, given by $\rho$, hence $\dot{\tilde e}$ is a linear combination of coefficients of $\dot e$. As
every edge is adjacent to 4 corners and every corner is adjacent to 2 edges, we actually have $3|E|+|V|$ equations.

 Denote the matrix of the system $(\mathcal S)$ by $\mc R$ and consider the adjoint matrix $\mc R^*$. It is the matrix of a linear operator from $\R^{E\cup Q\cup V}$ to $\R^{3E}$. Consider an element of its kernel. We denote its components by $a_e,$ $e \in E$, $b_q,$ $q \in Q$, and $c_v,$ $v \in V$. Every edge $e$ determines a vector equation on these numbers.

\begin{figure}[h]
\begin{center}
\includegraphics[scale=0.3]{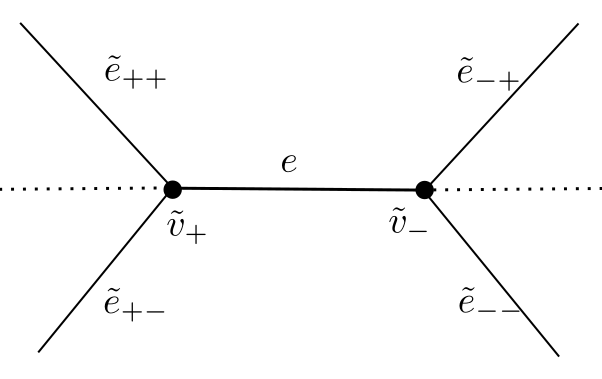}
\caption{\footnotesize	 A four-legged picture of the four corners around an edge $e$.}
\label{fig:fourleg}
\end{center}
\end{figure}

%
%

For every edge $e \in E$ we consider a ``four-legged picture'' (see Figure~\ref{fig:fourleg}). Denote the vertices of $e$ by $\tilde v_+$ and $\tilde v_-$. Note that we do not make an assumption on the orientation of $e$ with respect to the vertices. Denote the edges forming corners with $e$ at $\tilde v_+$ by $\tilde e_{++}$ and $\tilde e_{+-}$ clockwise, and forming corners at $\tilde v_-$ by $\tilde e_{-+}$, $\tilde e_{--}$ counterclockwise. For simplicity, we denote the numbers corresponding to this picture by $b_{++}$, $b_{+-}$, $b_{-+}$, $b_{--}$, $c_+$, $c_-$, $\tilde \lambda_+$, $\tilde \lambda_-$, $\tilde \sigma_+$, $\tilde \sigma_-$ respectively.
By noting that the equations in $(\mathcal S)$ are invariant for hyperbolic isometries, we get the following equation from $\mc R^*$
\begin{equation}
\label{r1}
a_e e+b_{++}\tilde e_{++}+b_{+-}\tilde e_{+-}+b_{-+}\tilde e_{-+}+b_{--}\tilde e_{--}+c_+\tilde\lambda_+\tilde\sigma_{+}\tilde v_++c_-\tilde\lambda_-\tilde\sigma_{-}\tilde v_-=0~.
\end{equation}

Denote also $\tilde\sigma_{++}:=\tilde \sigma_{\tilde v_+\tilde e_{++}}$, $\tilde \sigma_{+-}:=\tilde \sigma_{\tilde v_+\tilde e_{+-}}$, and denote the angles of the respective corners by $\alpha_{++}$ and $\alpha_{+-}$. Note that
\begin{equation*}
\begin{split}
& \la \tilde v_-, \tilde\sigma_{++}\tilde e_{++}\ra=-\sin(\alpha_{++})\sinh(\ell_e)~,\\
& \la \tilde v_-, \tilde\sigma_{+-}\tilde e_{+-}\ra=\sin(\alpha_{+-})\sinh(\ell_e)~,
\end{split}
\end{equation*}
where $\ell_e$ is the hyperbolic length of $e$.
For example, for the first line, we consider an orthogonal projection of $\tilde{v}_-$ onto the line directed by $\tilde{e}_{++}$ in the hyperbolic plane, and use formula for hyperbolic right-angled triangle, and then use the relation between the Minkowski scalar product and the oriented distance from a point to a geodesic in the hyperbolic plane, see also \cite{iskhakov}.

Thus, when we take the Minkowski product of  Equation~\eqref{r1} and $-\tilde\sigma_+\tilde v_-$, it turns to
\begin{equation}
\label{r2}
\sinh(\ell_e)\left(b_{++}\tilde\sigma_{++}\tilde\sigma_+\sin(\alpha_{++})-b_{+-}\tilde\sigma_{+-}\tilde\sigma_+\sin(\alpha_{+-})\right)+c_+\tilde\lambda_+\cosh(\ell_e)-c_-\tilde\lambda_-=0~.
\end{equation}

At this point, we need to treat separately the presence of edges of zero length. Actually, the case without edges of zero length is included in the case with edges of zero length, but we provide a proof for  completeness,  because 
it is more straightforward, and also because 
 this is the argument the closest to \cite{iskhakov}.

\vspace{\baselineskip}

\begin{proof}[Proof of Lemma~\ref{lem:system S} when all edges have positive length.]  In this case, Equation~\eqref{r2}
can be transformed to
\begin{equation}\label{eq:r2b}b_{++}\tilde\sigma_{++}\tilde\sigma_+\sin(\alpha_{++})-b_{+-}\tilde\sigma_{+-}\tilde\sigma_+\sin(\alpha_{+-})+\frac{c_+\tilde\lambda_+\cosh(\ell_e)-c_-\tilde\lambda_-}{\sinh(\ell_e)}=0~.\end{equation}

This expression depends only on the image of $\tilde v_+$ in $V$ and on $e$. Now for $v \in V$ we denote by $E_v$ the set of edges of $\ms G$ adjacent to $v$. We need to add one technical peculiarity: if $e \in E$ is a loop in $\ms G$, then we consider two copies of it in $E_v$. We then mark these copies by the orientation of $e$ at $v$. For $e \in E_v$
define \begin{equation}\label{eq:notation san tilde}\lambda_{ve}=\tilde \lambda_{\tilde v\tilde e}~,\;\sigma_{ve}=\tilde\sigma_{\tilde v\tilde e}\end{equation}  for adjacent lifts $\tilde v$ of $v$ and $\tilde e$ of $e$.
 As here $e \in E_v$, it can be one of the two copies of a loop, and thus a loop contributes two times. We denote by $E_{vw}$ the set of edges adjacent, as edges of $\ms G$, to both $v, w\in V$. Fixing $v \in V$ and summing Equations \eqref{eq:r2b} over all edges incident to $v$, we see that the summands containing $b$ get eliminated and we obtain a system of equations only on the summands containing $c$:
\begin{equation}\label{eq: matrice X} c_v\sum_{e \in E_v}\lambda_{ve}\frac{\cosh(\ell_e)}{\sinh(\ell_e)}-\sum_{w \in V}c_w \sum_{e \in E_{vw}}  \lambda_{we}\frac{1}{\sinh(\ell_e)}=0,~~~ \forall v \in V~.\end{equation}

The matrix of this system is a $(|V|\times |V|)$-matrix, which we denote by $\mc X$. It has positive coefficients on the diagonal, non-positive coefficients in the other entries, and has a positive sum of the entries in each column. Indeed, for the column corresponding to a vertex $v$ the sum of the coefficients is
\begin{equation}
\label{r3}
\sum_{e \in E_{v}}\lambda_{ve}\frac{\cosh(\ell_e)-1}{\sinh(\ell_e)}>0~.
\end{equation}
Then $\mc X$ is an instance of what is called a strictly diagonally dominant matrix. It is easy to check that it has a non-zero determinant, see e.g. \cite[Appendix B]{iskhakov}. Thereby, all the components $c_v$ are zero. 

This implies that Equation~\eqref{r2} turns to
\begin{equation}
\label{r4}
b_{++}\tilde\sigma_{++}\sin(\alpha_{++})=b_{+-}\tilde\sigma_{+-}\sin(\alpha_{+-})~.
\end{equation}
For a given vertex $v \in V$ and any corner $q$ incident to $v$ in $\ms G$, Equation~\eqref{r4}
 says that the expression $b_q\sigma_{q+}\sigma_{q-}\sin(\alpha_q)$ is independent on $q$, where the meaning of the notation $\sigma_{q\pm}$ is obvious. Hence, there exists a real number $t_v$ such that $b_q=\frac{t_v}{\sigma_{q+}\sigma_{q-}\sin(\alpha_q)}$. The components $a_e$ are expressed in terms of $t_v$ from Equation~(\ref{r1}). Thus, the nullity of $\mc R^*$ is at most $|V|$, but it is also at least $|V|$, hence it is equal to $|V|$. Then the rank of $\mc R^*$, as well as the rank of $\mc R$, is $3|E|$, thereby the nullity of $\mc R$ is zero.
 \end{proof}
 
 \vspace{\baselineskip}

\begin{proof}[Proof of Lemma~\ref{lem:system S} in the general case.]
 We denote by $E_0 \subset E$ the subset of edges of zero length.  First, note that if $e \in E_0$, then Equation~(\ref{r2}) changes to
\begin{equation}\label{eq: null edge 5}c_+\tilde\lambda_+-c_-\tilde\lambda_-=0~.\end{equation}
We still sum it together with the other equations for each vertex as for \eqref{eq: matrice X}, and obtain the matrix $\mc X$. Namely, due to  \eqref{eq: null edge 5}, using the notation from \eqref{eq:notation san tilde}, for every $v \in V$ the equation defining $\mc X$ is 
\begin{multline}\label{eq: matrice X2} c_v\left(\sum_{e \in E_v\backslash E_0}\lambda_{ve}\frac{\cosh(\ell_e)}{\sinh(\ell_e)}+\sum_{e\in E_v\cap E_0}\lambda_{ve}\right)-\\-\sum_{w \in V}c_w \left(\sum_{e \in E_{vw}\backslash E_0}  \lambda_{we}\frac{1}{\sinh(\ell_e)}+ \sum_{e\in E_{vw}\cap E_0}\lambda_{we}\right)=0 ~.\end{multline}
The  sum of the coefficients in a column 
is now equal to
\begin{equation}\label{eq: sommediag2}s_v:=\sum_{e \in E_{v}\backslash E_0}\lambda_{ve}\frac{\cosh(\ell_e)-1}{\sinh(\ell_e)}~,\end{equation}
which is zero when $v$ is adjacent only to  edges of zero length. Thus we cannot conclude that $\mc X$ has non-zero determinant exactly as in the previous case.
However, $\mc X$ is still non-degenerate. Indeed, consider a non-zero vector $c\in \R^V$, and let $c_v$ be such that $|c_v|$ is maximal among the coefficients of $c$. Without loss of generality, let us suppose that $c_v$ is positive. Let $x_{vw}$ be the coefficients of $\mc X^T$. There are two cases.
\begin{itemize}
\item If there exists an edge $e$ adjacent to $v$ and with positive length, then $s_v>0$. On the other hand,
$\left( \mc X^T c\right)_v\geq c_v s_v>0$, i.e.,
 $c$ is not in the kernel of $\mc X^T$. 
\item Suppose now that all edges adjacent to $v$ have zero length, then $s_v=0$ and $x_{vv}=-\sum_{w\not= v} x_{vw}$. Suppose that $c$ is in the kernel of $\mc X^T$, i.e., $x_{vv}c_v=-\sum_{w\not= v}x_{vw}c_w$. This is impossible if $c_v>c_w$ when $x_{vw}\not=0$. Hence for $w\in V$ adjacent to $v$, we have $|c_v|=|c_w|$. As the graph of $\ms G$ is connected, and as there are edges of positive length, we are back to the previous case.
\end{itemize}
In turn, the kernel of $\mc X^T$ is trivial, thus so is the kernel of $\mc X$.

Now we have to study the coefficients $b_q$. A difficulty is that we do not have Equations~\eqref{r4} for  edges of zero length. 

Consider the restricted matrix $\mc R_r$ obtained from $\mc R$ by deleting all the equations of the vertex-type. Define $\bar E:=E\backslash E_0$. We denote by $\bar V$ the set of visible vertices, i.e., the vertices of the quotient of $\ms G$ obtained by contracting all  edges of zero length. Similarly, denote by $\bar Q$ the set of visible corners. Consider the matrix $\bar{\mc R}_r$ coming from the system of equations
\[\begin{cases}
\la \dot e, e\ra=0 &\forall e\in \bar E \\
\la \tilde e_{q+}, \dot{\tilde e}_{q-}\ra+\la \dot{\tilde e}_{q+},\tilde e_{q-}\ra=0 &\forall q\in \bar Q
\end{cases}~. \]

By the argument from the proof of the non-zero edge length case, \[\dim(\ker(\bar{\mc R}_r)\cap \R^{\bar E})\leq |\bar V|~.\] 
 Due to Lemma~\ref{zerocorn}, all rows of $\bar{\mc R}_r$ are linear combinations of the rows of $\mc R_r$, see Figure~\ref{fig:mergecorner}.  Hence, also \[\dim(\ker(\mc R_r)\cap \R^{\bar E})\leq |\bar V|~.\]

\begin{figure}[h]
\begin{center}
 \includegraphics[scale=0.3]{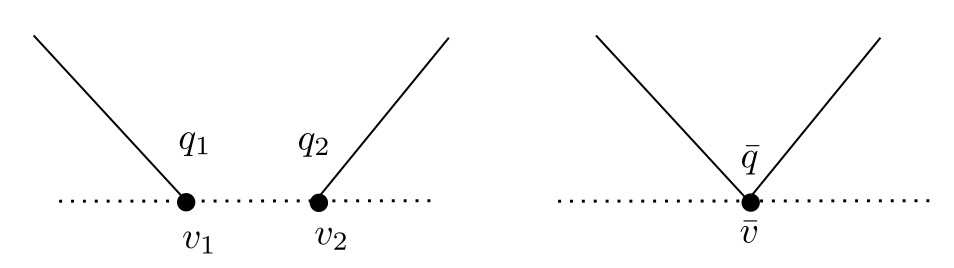}
\caption{\footnotesize	 On the left picture are two corners defining rows in $\mc R_r$, the doted edge between the two vertices having zero length. On the right is the corresponding corner defining a row in $\bar{\mc R}_r$.  Lemma~\ref{zerocorn} says that those rows are linearly dependent.}
\label{fig:mergecorner}
\end{center}
\end{figure}

 We claim that $\dim(\ker(\mc R_r)) \leq |\bar V|+|E_0|$. Indeed, order all edges in $E_0$ as $e_1, \ldots, e_k$ such that every $e_i$ is adjacent to some $e \in \bar E_{i-1}:=\bar E\cup\{e_1,\ldots, e_{i-1}\}$, and pick a corner $q_i$ adjacent to $e_i$ and to such $e$. Consider the matrix $\bar{\mc R}_{r, 1}$ obtained from $\bar{\mc R}_r$ by adding the equations
\begin{equation*}\begin{split}&\la \dot e_1, e_1\ra=0~, \\
& \la \tilde e_{q_1+}, \dot{\tilde e}_{q_1-}\ra+\la \dot{\tilde e}_{q_1+},\tilde e_{q_1-}\ra=0~.\end{split}\end{equation*}
Clearly, the corresponding two rows are linearly independent with each other and with all the previous rows: the entries corresponding to $\dot e_1$ constitute two linearly independent vectors, while for the previous rows those are zero-vectors. Hence, we have \[\dim(\ker(\bar{\mc R}_{r, 1})\cap \R^{\bar E_1})\leq |\bar V|+1~.\] By continuing similarly, we obtain matrices $\bar{\mc R}_{r, i}$, $i=1,\ldots, k$, and in the end we get $\dim(\ker(\bar{\mc R}_{r, k}))\leq |\bar V|+|E_0|~.$ But all the rows of $\bar{\mc R}_{r, k}$ are linear combinations of the rows of $\mc R_r$. Hence, \[\dim(\ker(\mc R_{r}))\leq |\bar V|+|E_0|= |V|~.\] It follows that the nullity of $\mc R$ is zero and the proof of Lemma~\ref{lem:system S} is finished.
\end{proof}

During the proof of Lemma~\ref{lem:system S}, we used the following fact.

\begin{lemma}
\label{zerocorn}
Let $e_1$, $e_2$, $e_3$ be three pairwise distinct hyperbolic lines intersecting at a single point and $\dot e_1$, $\dot e_2$, $\dot e_3$ be their infinitesimal deformations inducing zero-deformations on the angle between $e_1$ and $e_2$ and on the angle between $e_2$ and $e_3$. Then the induced deformation on the angle between $e_1$ and $e_3$ is also zero.
\end{lemma}

Note that here the infinitesimal deformations are not assumed to preserve the concurrence of the lines.

\begin{proof}
We orient the lines arbitrarily and use the same notation $e_i$ for their unit oriented normals in $\R^{2,1}$. Then the deformations $\dot e_i$ also can be considered as vectors in $\R^{2,1}$.
Since the lines intersect at a point, for some $\lambda_1, \lambda_2, \lambda_3 \neq 0$ we have
\[\lambda_1e_1+\lambda_2e_2+\lambda_3e_3=0~.\]
Then
\begin{equation*} 
\begin{split}
\la \dot e_1, e_3\ra=\la \dot e_1, -\frac{\lambda_2}{\lambda_3}e_2\ra & =\frac{\lambda_2}{\lambda_3}\la \dot e_2, e_1 \ra=\frac{\lambda_2}{\lambda_3}\la \dot e_2, -\frac{\lambda_3}{\lambda_1} e_3\ra= \\ &=\frac{\lambda_2}{\lambda_1}\la \dot e_3, e_2 \ra=\frac{\lambda_2}{\lambda_1}\la \dot e_3, -\frac{\lambda_1}{\lambda_2} e_1 \ra=-\la \dot e_3, e_1\ra~.
\end{split}
\end{equation*}
\end{proof}

\subsection{Compactness}\label{sec:compactness}

The aim of this section is to prove Proposition~\ref{compact}.  We first prove the subconvergence of the spacetimes in Section~\ref{sec:cv cocycles}, then the subconvergence of the marked points in Section~\ref{sec:sub points}.

\subsubsection{Convergence of ambient spacetimes}
\label{sec:cv cocycles}

The aim of this subsection is to prove Lemma~\ref{cohomconverge}, which gives the subconvergence of the classes of cocycles when induced distances converge.

\begin{lemma}
\label{horiz}
Let $\Pi_0, \Pi_1, \Pi_2 \subset \R^{2,1}$ be three spacelike planes intersecting transversely at a common point; $\psi$ be the intersection line of $\Pi_1$ and $\Pi_2$; $C$ be the future-convex wedge formed by $\Pi_1$ and $\Pi_2$; and $\chi$ be the intersection curve of $\pt C$ and $\Pi_0$. Denote by $\alpha_+$ the intrinsic angle in $\pt C$ of $\chi$ at its intersection point with $\psi$ in the future side from $\Pi_0$. Then $\alpha_+ \geq \pi$.
\end{lemma}

\begin{proof}
By applying a Lorentzian isometry, we can assume that the intersection point of all three planes is the origin and that the wedge $C$ is symmetric with respect to the $x_0$-axis. This particularly means that the line $\psi$ is horizontal. Denote the two rays that constitute $\chi$ by $\chi_1\subset \Pi_1$ and $\chi_2 \subset \Pi_2$. Denote by $\chi'_1 \subset \Pi_2$ the ray obtained from $\chi_1$ by the symmetry in $\R^{2,1}$ with respect to the $x_0$-axis. Note that by symmetry, $\chi_1$ and $\chi_1'$ constitute an intrinsic geodesic in $\pt C$, see Figure~\ref{fig:lemmas}. On the other hand, one can see that $\chi_1'$ belongs to the future side of all spacelike planes containing the ray $\chi_1$, particularly of $\Pi_0$. Therefore, the intrinsic angle between $\chi_1$ and $\chi_2$ in $\pt C$ that is in the future of $\Pi_0$ contains $\chi_1'$, and hence is $\geq \pi$.
\end{proof}

\begin{lemma}
\label{vertic}
Let $\Pi_1, \Pi_2 \subset \R^{2,1}$ be two non-parallel spacelike planes; $\psi$ be the intersection line of $\Pi_1$ and $\Pi_2$; $C$ be the future-convex wedge formed by $\Pi_1$ and $\Pi_2$; $\Pi_0$ be a timelike plane intersecting $\psi$ transversely; and $\chi$ be the intersection curve of $\pt C$ and $\Pi_0$. Assume that the origin $o$ is the intersection point of all three planes and that a normal vector $n$ to $\Pi_0$ belongs to $C$. Denote by $\alpha_+$ the intrinsic angle in $\pt C$ of $\chi$ at $o$ in the opposite side to $n$ with respect to $\Pi_0$. Then $\alpha_+ \geq \pi$.
\end{lemma}

\begin{proof}
We consider the configuration and the notation as in the proof above, see Figure~\ref{fig:lemmas}. Denote the horizontal plane by $\Pi$ and the future-directed unit vertical vector by $n_v$. For any timelike plane, if it has a normal vector in the future of $\Pi$, then it separates this vector from $n_v$. Hence, in our situation $n$ and $n_v$ cannot belong to the same open halfspace with respect to $\Pi_0$. But $n_v$ is in the convex hull of $\chi_1$ and $\chi'_1$. Thus, since $\Pi_0$ contains $\chi_1$, we get that $n$ and $\chi'_1$ cannot belong to the same open halfspace with respect to $\Pi_0$. Hence, the intrinsic angle in $\pt C$ of $\chi$ at $o$ in the opposite side to $n$ with respect to $\Pi_0$ contains $\chi'_1$. Then $\alpha_+ \geq \pi$.
\end{proof}

\begin{figure}[h]
\begin{center}
\includegraphics[scale=0.25]{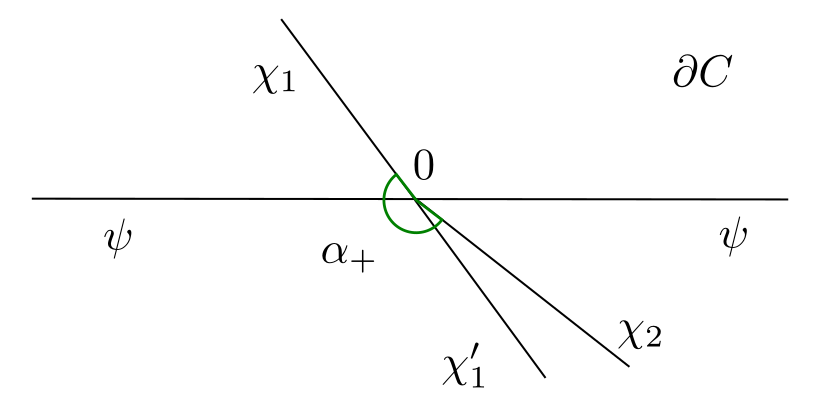}
\caption{ \footnotesize	For the proofs of Lemma~\ref{horiz} and~\ref{vertic}, the positions of rays in the intrinsic geometry of $\pt C$, which is isometric to the Euclidean plane. The idea of both proofs is to show that the ray $\chi'_1$ belongs to the correct side with respect to the plane $\Pi_0$, which is spanned by the rays $\chi_1$ and $\chi_2$ in space.}
\label{fig:lemmas}
\end{center}
\end{figure}

The proofs of the next two lemmas use the ideas of Bonsante from his proof of Proposition 7.6 in \cite{bonsante}.

\begin{lemma}
\label{nondecr}
Let $C \subset \R^{2,1}$ be a spacelike future convex Cauchy polyhedral set. Fix an affine system of coordinates such that the origin $o$ is in $\pt C$ and the horizontal plane $\Pi$ through the origin is a supporting plane to $C$. Let $\phi: \Pi \rar \R$ be the convex function defining $\pt C$. Let $p \in \pt C$ and $\psi: [0, T] \rar \pt C$ be an intrinsic geodesic segment connecting $o$ with $p$. Let $\psi_h: [0, T] \rar \Pi$ be the orthogonal projection of $\psi$ to $\Pi$. Then $\phi\circ\psi_h$ and $\|\psi_h\|$ are non-decreasing functions.
\end{lemma}
In this section, $\|\cdot\|$ is the  Minkowski norm, which gives the Euclidean norm over $\Pi$.
\begin{proof}
The intrinsic metric on $\pt C$ is CAT(0), which implies that a geodesic segment between any two points is unique. Note also that $\psi$ is a piecewise-linear map. For $a>0$ define $D_a:=\{x=(x_0, x_1, x_2) \in \pt C: x_0 \leq a\}$. Lemma~\ref{horiz} implies that $D_a$ is convex in the intrinsic geometry of $\pt C$. Indeed, it implies directly that at every kink point of $\pt D_a$ that is the intersection of a horizontal plane with an edge of $\pt C$, the intrinsic angle is $\leq \pi$. But for the kink points that come from the vertices of $\pt C$ the same follows by continuity. From the convexity of $D_a$, it is easy to see that $\phi\circ\psi_h$ is non-decreasing. Indeed, suppose the converse. Then some $t \in [0, T]$ belongs to a segment where $\phi\circ\psi_h$ is decreasing. Then the intrinsic geodesic from $o$ to $\psi(t)$ does not completely belong to $D_{\phi\circ\psi_h(t)}$. This contradicts to the convexity of $D_{\phi\circ\psi_h(t)}$. 

For the second claim, pick $q \in \Pi$ and define the set $D_q:=\{x \in \pt C: \la x,q \ra \leq \|q\|\}$. We claim that it is also convex. Indeed, consider the timelike plane \[\Pi_q:=\{x \in \R^{2,1}: \la x,q \ra = \|q\|\}~.\] Because $C$ is future-convex and $\Pi$ is a supporting plane to $C$ at $o$, for every $p \in \pt D_q$ the ray from $p$ orthogonal to $\Pi_q$, pointing to the halfspace containing $o$, does not point to the exterior of $C$. Thereby, for any two supporting planes $\Pi_1$ and $\Pi_2$ to $C$ that form an edge of $C$ intersecting $\pt D_q$, Lemma~\ref{vertic} shows that the intrinsic angle at the corresponding kink point of $\pt D_q$ is $\leq \pi$. The same holds by continuity for all kink points of $\pt D_q$. It follows that $D_q$ is convex.

Suppose now that $\|\psi_h\|$ is somewhere decreasing. Then some $t \in [0, T]$ belongs to a segment where $\|\psi_h\|$ is decreasing. Then the geodesic from $o$ to $\psi(t)$ does not completely belong to $D_{\psi_h(t)}$, which contradicts to the convexity of the latter.
\end{proof}

\begin{lemma}
\label{cohomconverge}
Let $\p_i=(\rho,\tau_i,f_i)$ be a sequence in $\P_c(\rho, V)$ such that the induced metrics $\mc I(\p_i)=d_i \in \mc \M_-(S, V)$ converge to a metric $d \in \mc \M_-(S, V)$. Then, up to a subsequence, the cohomology classes $\tau_i \in \mc H^1(\rho)$ converge.
\end{lemma}

\begin{proof}
Fix $v_0 \in \tilde V$ and choose a sequence representative of the classes of cocycles, which  we will still denote by  $\tau_i$, such that for the corresponding maps $\tilde f_i: \tilde V \rar \tilde \Omega^+(\rho, \tau_i)$ we have $\tilde f_i(v_0)=o$, where $o$ is the origin, and the horizontal plane $\Pi$ through $o$ is a supporting plane to all $\widetilde \CH(\p_i)$. Up to passing to a subsequence, we may assume that $\tau_i$ converge to a map $\ol\tau: \pi_1S \rar \ol{\R^{2,1}}$. Pick a system of generators $\Xi$ for $\pi_1S$. If for all $\gamma \in \Xi$ we have $\ol \tau(\gamma)\in  \R^{2,1}$, then $\ol{\tau}$ is a cocycle, determining the desired $\tau \in \mc H^1(\rho)$. 

Suppose that for some $\gamma \in \Xi$ we have $\ol{\tau}(\gamma)\in \pt_\infty \R^{2,1}$. The rest of the proof is devoted to finding a contradiction to this. 

Let $K \subset \R^{2,1}$ be the Euclidean unit ball centered at $o$. Consider a sequence of positive numbers $\lambda_i$ such that for all $\gamma \in \Xi$ we have $\lambda_i\tau_i(\gamma)\in K$ and for some $\gamma_i \in \Xi$ we have $\lambda_i\tau_i(\gamma_i)\in \pt K$. Then the sequence $\lambda_i$ converges to zero and, up to passing to a subsequence, $\lambda_i\tau_i$ converge to a cocycle $\tau_\lambda \in \mc Z^1(\rho)$ such that for some $\gamma_0 \in \Xi$ we have $\tau_\lambda(\gamma_0) \in \pt K$, particularly $\tau_\lambda(\gamma_0) \neq 0$.

We denote by $\ol C_i$ the closures of $ C_i:=\widetilde \CH (\lambda_i \p_i)$ in $\ol{\R^{2,1}}$. Up to passing to a subsequence, the sets $\ol C_i$ converge in the Hausdorff sense to a closed convex set $\ol C \subset \ol{\R^{2,1}}$ (here we mean the Hausdorff convergence with respect to the topology of $\ol{\R^{2,1}}$), and the maps $\lambda_i \tilde f_i$ converge to a map $\tilde f: \tilde V \rar \pt \ol C$. Because for all $i$ we have $o \in \pt C_i$ and $\pt^0_\infty \R^{2,1} \subset \pt \ol C_i$, the set $C:=\ol C \cap \R^{2,1}$ is non-empty, and is a future convex set. Note that $\Pi$ is a supporting plane to $C$, which implies that for all $\gamma \in \pi_1S$, the vector $\tau_\lambda(\gamma)$ is spacelike or zero. Indeed, if for some $\gamma \in \pi_1S$ this is wrong, then $\tau_\lambda(\gamma^{-1})$ does not belong to the future side of $\Pi$, as for all $i$ we have $\tau_i(\gamma^{-1})=-\rho(\gamma)\tau_i(\gamma)$. Because $C$ is future convex, $\pt C$ is a graph over $\Pi$. Because $\lambda_i$ converge to zero, the metrics $\lambda_id_i$ on $\tilde S$ also converge to zero. 

We denote by $\phi_i, \phi: \Pi \rar \R$ the convex functions defining $\pt C_i$, $\pt C$.
Let $\psi_i$ be the geodesic segment in $\pt C_i$ connecting $o$ with $\lambda_i \tilde f_i(\gamma_0 v_0)$. By definition, this sequence of points converges to the point $\tilde f (\gamma_0v_0)=\tau_\lambda(\gamma_0)\in \partial K$.
Let $\psi_{h, i}$ be the orthogonal projection of $\psi_i$ onto $\Pi$, and let $\ell_i$ be its length. 
Note that since $\tilde f(\gamma_0v_0)$ is in $\pt K$ and is spacelike, $\ell_i$ are bounded away from zero. 
We remark that either the sequence $\ell_i$ is bounded, or there is a subsequence that goes to infinity. We will see below that any of these cases leads to the wanted contradiction.

Suppose first that $\ell_i$ are bounded.
 Then, up to passing to a subsequence, $\ell_i$ converge to a positive number $\ell$.
Abusing notation, we will denote by $\psi_{h, i}$ and $\psi_{ i}$ the parametrization proportional to arc-length over $[0,\ell]$ of the respective paths, with $0$ sent to $o$. By Arzel\`a--Ascoli,
 $\psi_{h, i}$ converge to a path 
  $\psi_h: [0,\ell] \rar \Pi$, with length $\ell_0\leq \ell$. We denote by $\psi: [0, \ell] \rar \pt C$ its lift to $\pt C$, which is also a Lipschitz path. 
Since the metrics $\lambda_id_i$ converge to zero, the lengths of the paths $\psi_i$ converge to zero. We have
\[\lim_{i \rar +\infty} \int_0^{\ell}\|\psi_i'(t)\|dt=0~.\]
Therefore, $\lim_{i \rar \infty}\|\psi_i'(t)\|=0$ for almost all $t \in [0, \ell]$. 
By Lemma~\ref{nondecr}, the numbers $(\phi_i\circ\psi_{h,i})'(t)$ are non-negative where defined (which is almost everywhere). As by
definition, $\|\psi_{h,i}'(t)\|=\frac{\ell_i}{\ell}$ almost everywhere, and as $\|\psi_i'(t)\|^2=\|\psi_{h,i}'(t)\|^2-((\phi_i\circ\psi_{h,i})'(t))^2$  almost everywhere, it follows that  almost everywhere, $\lim_{i \rar \infty}(\phi_i\circ\psi_{h,i})'(t)=1$.
By the Dominated Convergence Theorem, $\phi_i \circ \psi_{h,i}(t)$ converges then to $t$. But  it also converges to $\phi \circ \psi_h(t)$,  so the latter is equal to $t$.
On the other hand, \[\|\psi_h(t)\| \leq \len(\psi_h|_{[0, t]})\leq \liminf_i\len (\psi_{h,i}|_{[0, t]}) =\lim_i t \frac{\ell_i}{\ell}=t~.\]
We see that for every $t$ the Minkowski vector $\psi(t)$ is causal. Particularly so is $\psi(\ell)=\tilde f(\gamma_0 v_0)=\tau_\lambda(\gamma_0)$. This is a contradiction as $\tau_\lambda(\gamma_0)$ must be spacelike.

Suppose  now that up to passing to a subsequence, $\ell_i\to \infty$. Note however that the endpoint of $\psi_i$ is in $K$,  and by Lemma~\ref{nondecr}, for an arc-length parametrization, $\|\psi_{h,i}\|$ are non-decreasing for every $i$. Hence, the images of $\psi_{h,i}$ belong to the compact set $K \cap \Pi$. Then, up to a subsequence again, the maps $\psi_{h, i}$ converge uniformly on compact sets to a Lipschitz map $\psi_h: [0, +\infty) \rar K \cap \Pi$. More precisely, let us parametrize all the $\psi_{h, i}$ by arc-lengths.  Up to a finite number of elements, we assume that all lengths are greater than 2.
Using Arzel\`a--Ascoli we choose a subsequence of $\psi_{h,i}$ such that  $\psi_{h,i}|_{[0, 2]}$ converges  to a 1-Lipschitz path $\psi_{h}|_{[0, 2]}$.
Starting with this subsequence, we reproduce the same procedure, with $3$ instead of $2$, and  the limiting path  $\psi_{h}|_{[0, 3]}$ coincides with $\psi_{h}|_{[0, 2]}$ over $[0,2]$ by the uniqueness of the limit.
 We repeat  this countably many times to define $\psi_{h}$ over $[0, +\infty)$. 

 Denote its lift to $\pt C$ by $\psi: [0, +\infty) \rar \pt C$. Similarly to the previous case, we get $\phi\circ\psi_h(t)=t$. Since $\psi_h(t)$ belongs to $K$, we have $\|\psi_h(t)\|\leq 1$ for all $t$. Therefore, for some $t$ the Minkowski vector $\psi(t) \in \pt C$ is timelike. This contradicts to the fact that $C$ is a future convex set.

At the end of the day, such a sequence $(\ell_i)_i$ cannot exist, and  we have arrived at the wanted contradiction.
\end{proof}

\subsubsection{Convergence of marked points}
\label{sec:sub points}
\begin{lemma}
\label{minmax}
For every $\alpha>0$ and every compact set $U \subset \mc H^1(\rho)$ there exists a constant $\beta \geq \alpha$ such that for every convex Cauchy surface $\Sigma\subset \Omega^+(\rho, \tau)$, $\tau \in U$, if the infimum of the cosmological time  over $\Sigma$ is at most $\alpha$, then its supremum is at most $\beta$. 
\end{lemma}

\begin{proof}
In the case when $U$ is reduced to a single point, this lemma is proved in \cite[Corollary 3.19]{BF}. Now let us consider the general case, and suppose that the result is false. Hence, there exist $\alpha>0$ and a sequence $(\tau_i)_i$ of cocycles converging to a cocycle $\tau$ and a sequence of convex Cauchy surfaces $(\Sigma_i)_i$ with the infimum of cosmological time bounded above by $\alpha$ and the supremum of the cosmological time going to $\infty$. We have that $\tilde{\Omega}^+(\rho,\tau_i)$ converge to $\tilde{\Omega}^+(\rho,\tau)$, as well as the respective level sets of the cosmological time, see   \cite[Lemma 3.27]{barbot-fillastre} or Propositions~6.2 and 6.4 in \cite{bonsante}. (It is worth noting the relation between support functions and extremal values of the cosmological time, cf. \cite[Lemma 3.10]{BF}.) This implies that for any sequence of points $\tilde p_i \in \tilde \Omega^+(\rho, \tau_i)$ converging to $\tilde p \in \tilde \Omega^+(\rho, \tau)$ we have $\CT_i(\tilde p_i)$ converge to $\CT(\tilde p)$.

Let $p_i\in \Sigma_i\cap L_\alpha$. 
As the convex hull $P_i$ of $p_i$ is in the future of $\Sigma_i$, the cosmological time of $\pt P_i$ also goes to infinity.
 As  $L_\alpha$ is compact, $(p_i)_i$ subconverges to a point $p$ in $L_\alpha$. From Lemma~\ref{subdiv}, for all sufficiently large $i$ the boundary cellulation of $\clconv(p_i)$ is a subdivision of the boundary cellulation of $\clconv(p)$. From this it is easy to see that $\pt\widetilde\clconv(p_i)$ converges to $\pt \widetilde\clconv(p)$. This and the remark above on the behavior of the cosmological time imply that $\CT$ is unbounded from above on $\pt\clconv(p)$. But by Lemma~\ref{lem:boudary polyhedral}, the latter is a convex Cauchy surface, particularly compact, which shows the contradiction.
\end{proof}

Consider a level surface of the cosmological time. It is a $C^1$-submanifold, hence the induced metric is induced by a Riemannian structure. Its universal covering can be put in the future cone of the origin of Minkowski space, so that it is locally a graph of a function over the hyperboloid. We may locally approximate this graph by smooth strictly convex ones, whose induced metric has negative sectional curvature. Then, following the proof of Proposition~3.12 in \cite{filslu}, the induced distances locally uniformly converge, and hence the local CAT(0) property is preserved. We note that by a finer intrinsic investigation one can deduce that the induced metric is actually $C^{1,1}$ with curvature defined almost everywhere, see~\cite{KP}. We will apply the following lemma to the level surfaces.


\begin{lemma}
\label{curve1}
For every $\e, A>0$ and every CAT(0)-metric $d$ on $S$ of area at most $A$, for every $p \in S$ and every complete geodesic $\psi$ through $p$ in $(S, d)$ there exists a closed homotopically non-trivial Lipschitz curve $\chi$ through $p$ in $(S, d)$ consisting from three subsequent arcs $\chi_1$, $\chi_2$ and $\chi_3$ such that the arcs $\chi_1$ and $\chi_3$ are subarcs of $\psi$ and we have $\len(\chi_1) \leq A/\e$, $\len(\chi_2)\leq \e,$ $\len(\chi_3)\leq A/\e$.
\end{lemma}

We refer to~\cite{BH} for a discussion of the notion of area in general CAT(0)-spaces, but remark again that we will apply Theorem~\ref{curve1} only to metrics coming from $C^{1,1}$-Riemannian structures.
It is possible that some arcs are empty, particularly $\psi$ might be a closed geodesic of length at most $A/\e$, in which case it may be taken as $\chi$. It is possible that the arcs $\chi_1$ and $\chi_3$ overlap as subarcs of $\psi$.

\begin{proof}
Let $\tilde P$ be the full preimage of $p$ in $\tilde S$ and $\tilde p \in \tilde P$ be some lift. The lift $\tilde d$ of $d$ is a (globally) CAT(0)-metric. Let $\tilde \psi$ be the lift of $\psi$ through $\tilde p$ and $R(\tilde p, l, \e)$ be the set of points $r$ at distance at most $\e/2$ from $\tilde \psi$ (on both sides) such that the closest point to $r$ on $\tilde\psi$ is at distance at most $l/2$ from $\tilde p$.
There is a natural surjective map from $R(\tilde p, l, \e)$ to an $(\e\times l)$-Euclidean rectangle $E$: we send the corresponding segment of $\tilde\psi$ to the segment $I$ of $E$, connecting the midpoints of the sides of length $\e$, and then extend it to a map from $R(\tilde p, l, \e)$ to $E$ so that the nearest-point projections to $\tilde\psi$ and  to $I$ commute. The Fermi Lemma, \cite[Lemma 3.1]{alexander-bishop}, states that this map is 1-Lipschitz, hence 
 $\area(R(\tilde p, l, \e))\geq \e l$. 

  We denote by $l_0$ the minimal number such that $R(\tilde p, l_0, \e)$ intersects $R(\tilde p', l_0, \e)$ for some $\tilde p' \in \tilde P\setminus\{\tilde p\}$. Due to the area bound, $l_0 \leq A/\e$.

Since the sets $R(\tilde p, l_0, \e)$ and $R(\tilde p', l_0, \e)$  intersect, there exists a curve $\tilde \chi$ connecting $\tilde p$ with $\tilde p'$, consisting of three subsequent arcs $\tilde \chi_1$, $\tilde \chi_2$ and $\tilde \chi_3$ such that $\tilde\chi_1 \subset \tilde \psi$, $\len(\tilde \chi_1)\leq l_0$; $\len(\tilde\chi_2)\leq \e$; $\tilde \chi_3 \subset \tilde \psi'$, where $\tilde \psi'$ is the lift of $\psi$ passing through $\tilde p'$, $\len(\tilde\chi_3)\leq l_0$. The projection of $\tilde \chi$ to $S$ is the desired curve $\chi$.
\end{proof}

For a spacetime $\Omega^+(\rho, \tau)$ the flow $\xi^t$ is minus the gradient flow of the cosmological time.

\begin{lemma}[{\cite[Proposition 6.1]{BBZ}}]
\label{bbz1}
Let $\chi: [a,b] \rar \Omega^+(\rho, \tau)$ be a spacelike Lipschitz curve contained in the past of the level set $L_1$ of the cosmological time, and $\chi'$ be its projection to $L_1$ along the flow $\xi^t$. Then $\len(\chi)\leq \len(\chi')$.
\end{lemma}

The following lemma uses a decomposition of $\Omega^+(\rho, \tau)$ into \emph{thin blocks} and \emph{solid blocks}. Roughly speaking, the intersection of the level sets $L_a$ of the cosmological time with the thin blocks recovers the  geodesic lamination $\mc L$, which is obtained as the image of the Gauss map of the spacelike part of $\partial\tilde\Omega^+(\rho, \tau)$.  The connected components of $L_a\setminus \mc L$ are the intersections of $L_a$ with the solid blocks. As we only need to use the following lemma, we do not need to define further those blocks, and we refer to \cite{BBZ} for more details.

\begin{lemma}[{\cite[Proposition 6.2]{BBZ}}]
\label{bbz2}
Let $\chi: [a,b] \rar \Omega^+(\rho, \tau)$ be a spacelike Lipschitz curve and $\chi'$ be its projection to $L_1$ along the flow $\xi^t$. Assume that $\chi$ is contained in a single block of the canonical decomposition. Then
\[\exp(-\len(\chi'))\leq\frac{\CT(\chi(a))}{\CT(\chi(b))}\leq\exp(\len(\chi'))~,\]
\[\len(\chi)\leq \len(\chi')\exp(\len(\chi'))\CT(\chi(a))~.\]
\end{lemma}

The proof of the next Lemma is based on the ideas of Barbot--Beguin--Zeghib in their proof of Theorem 3.5 in \cite{BBZ}.

\begin{lemma}
\label{curve2}
For every $\e>0$ and every compact set $U \subset \mc H^1(\rho)$ there exists a constant $\alpha>0$ such that for every convex Cauchy surface $\Sigma\subset \Omega^+(\rho, \tau)$, $\tau \in U$, if the systole of $\Sigma$ is greater than $\e$, then the infimum of the cosmological time $\CT_{\rho,\tau}$ over $\Sigma$ is at least $\alpha$. 
\end{lemma}

\begin{proof}
Due to Lemma~\ref{minmax}, it is enough to consider only those $\Sigma$, for which the supremum of the cosmological time is bounded from above by some $\beta>0$. By applying scaling, we can assume that $\beta=1$.  Support functions of $\tilde L_1$  converge uniformly when cocycles converge \cite[Lemma 4.11]{barbot-fillastre}. This implies convergence of the areas of $L_1$-surfaces \cite[Lemma 2.12]{BF}. Hence, since $U$ is compact, there exists $A>0$ depending only on $U$ such that for every $\tau \in U$, we have $\area(L_1(\tau)) \leq A$, where $L_1(\tau)$ is the $L_1$-surface in $\Omega^+(\rho, \tau)$. Fix now some $\tau \in U$.

Consider a Cauchy surface $\Sigma \subset \Omega^+(\rho, \tau)$ in the past of $L_1$ of systole $\geq \e$. 
Pick $p \in \Sigma$ first in a thin block of $\Omega^+(\rho, \tau)$. The intersection of this block with $L_1$ is a simple geodesic $\psi'$. 
Let $p'$ be the projection of $p$ to $L_1$ along the flow $\xi^t$.  We apply Lemma~\ref{curve1} to $p'$, $\psi'$ and $\e/2$ and get a curve $\chi' \subset L_1$, consisting of three arcs $\chi_1', \chi_2'$ and $\chi_3'$ with $\len(\chi_1'), \len(\chi_3') \leq \frac{2A}{\e}$ and $\len(\chi'_2)\leq \frac{\e}{2}.$

Let $\chi$ be the $\xi^t$-projection of $\chi'$ to $\Sigma$. We transfer the decomposition of $\chi'$ into the arcs to the corresponding decomposition of $\chi$. Due to Lemma~\ref{bbz1}, we have $\len(\chi_2)\leq\len(\chi'_2)\leq \frac{\e}{2}$. Thus, as by assumption the systole of $\Sigma$ is greater than $\e$, one of the arcs $\chi_1$ or $\chi_3$ has length $\geq \frac{\e}{2}$. Applying Lemma~\ref{bbz2} to this arc, we obtain
\[\CT(p) \geq \frac{\e^2}{4A\exp(\frac{2A}{\e})}~.\]

Now take $p \in \Sigma$ in a solid block of $\Omega^+(\rho, \tau)$, and let $p'$ its $\xi^t$-projection onto $L_1$. 
Set $\delta:=\diam(S, h)$. There exists a point $q' \in L_1$ in a thin block of $\Omega^+(\rho, \tau)$ and a Lipschitz arc $\chi \subset L_1$ of length $\leq \delta$ connecting $p'$ and $q'$ and belonging to the solid block of $p'$ except the endpoint at $q'$.
Let $q$ be the projection of $q'$ along $\xi^t$. Then Lemma~\ref{bbz2} gives us a bound
\[\CT(p)\geq \CT(q)\exp(-\delta)~.\]

Putting it all together, it follows that for every $p \in \Sigma$ we have
\[\CT(p)\geq \exp(-\delta) \frac{\e^2}{4A\exp(\frac{2A}{\e})}~.\]
\end{proof}

Finally, we can give the proof of Proposition \ref{compact}.

\begin{proof}[Proof of Proposition \ref{compact}]
Lemma~\ref{cohomconverge} shows that $\tau_i$ subconverge to some $\tau \in \mc H^1(\rho)$. Let $\CT_i$ be the cosmological time of $\Omega^+(\rho, \tau_i)$.
Suppose that for all $v \in V$ the cosmological time $\CT_i(f_i(v))$ is unbounded from above.
From \cite[Lemma 7.4]{bonsante}, the area of the level sets of the cosmological time is an increasing function of the cosmological time. As the area of convex Cauchy surfaces is non-decreasing with respect to the inclusion \cite[Lemma 3.24]{BF},
the areas of $\pt \CH(\p_i)$ grow to infinity, which is a contradiction. Hence there is $v \in V$ such that $\CT_i(f_i(v))$ is uniformly bounded from above. Lemma~\ref{minmax} implies that the supremum of $\CT_i$ over $\pt \CH(\p_i)$ is uniformly bounded from above, particularly for all $v \in V$, $\CT_i(f_i(v))$ is uniformly bounded from above. Lemma~\ref{curve2} shows that it is also uniformly bounded from below, by continuity of the systole, see e.g. \cite[Lemma 2.20]{roman-hyp} (in the reference hyperbolic cone-metrics are considered instead of flat ones, but the proof of the statement remains exactly the same).
It follows that up to a subsequence, $\p_i$ converge to $\p \in\P(\rho, V)$. We need to observe that the vertices do not collapse. For every distinct $v, w \in V$, all segments between $f_i(v)$ and $f_i(w)$ are spacelike. By applying the reversed triangle inequality for the Minkowski plane, we see that the intrinsic distance between $f_i(v)$ and $f_i(w)$ in $\pt\clconv(\p)$ is bounded from above by the infimum of lengths of these segments. Hence, the vertices do not collapse, and we obtain $\p \in\ol{\P}_c(\rho, V)$. Since $d \in \mc \M_-(S, V)$, we have $\p \in \P_c(\rho, V)$.
\end{proof}

\section{Proof of Theorem~\ref{thmII'}}

\subsection{The Teichm\"uller tangent vector field associated to a flat metric}\label{sec:lift}

Recall from Lemma~\ref{open} that $\P_c(S, V)$ is an analytic manifold of dimension $12 \mathbf{g}-12+3n$. 
Recall from \eqref{eq:def hat pi} the fibration $\mu: \P(S, V) \rar T\mc T(S)$. We will also denote by $\mu$ its restriction to $\P_c(S, V)$, which is also a fibration.
The composition    with   the natural projection  $p:T \mc T(S)\rar \mc T(S)$ gives a fibration $\pi: \P_c(S, V) \rar \mc T(S)$. We have that  $\pi^{-1}(\rho)=\P_c(\rho,V)$ is a manifold of dimension $6 \mathbf{g}-6+3n$. 
By Theorem~\ref{thm:1}, the restriction of $\mc I$  to $\pi^{-1}(\rho)$ is a $C^1$ diffeomorphism, and in turn, by Lemma~\ref{diff}, the map $\mc I: \P_c(S, V)\to \mc \M_-(S, V)$ is a $C^1$ submersion.

It also follows that $\mc I^{-1}(d)$ is  a $C^1$ submanifold of $\P_c(S, V)$ of dimension $6\mathbf{g}-6$. From Theorem~\ref{thm:1},   the intersection $\mc I^{-1}(d)\cap \pi^{-1}(\rho)$ is a point, and is transverse.
In turn, $\pi|_{\mc I^{-1}(d)}$ is a $C^1$ diffeomorphism.

 The restriction of  $\mu :\P_c(S, V)\to T\mc T(S) $ to $\mc I^{-1}(d)$ is a  $C^1$ injective immersion.  Moreover,  $\mc I^{-1}(d)$  is the image of the $C^1$ section $(\pi|_{\mc I^{-1}(d)})^{-1}\circ p$, hence $\mu|_{\mc I^{-1}(d)}$ is proper, and in turn an embedding.

\begin{notation}\label{not: m X}
We denote by $\p_{d}(\rho)=\mc I^{-1}(d)\cap \pi^{-1}(\rho)$ and by $X_d(\rho)=\mu(\p_{d}(\rho))$ the resulting tangent vector of Teichm\"uller space at the point $\rho$.
\end{notation}

We consider $\p_d$ as a map $\mc T(S) \rar \P_c(S, V)$.

\begin{lemma}\label{lem:Xd C1}
The vector field $X_d$ over Teichm\"uller space is $C^1$.
\end{lemma}
\begin{proof}
That is immediate from its definition, as $\p_d$ is also $C^1$.
\end{proof}


We also need to check a stronger regularity property for $\p_d$.

\begin{notation} Let $\ms C$ be a cellulation of $(S, V)$, denote by $\P_c(S, \ms C)$ the space of marked spacetimes $\p$ such that the  face cellulation of $\clconv(\p)$ is exactly $\ms C$. Denote by $\ol{\P}_c(S, \mc C) \subset \P_c(S, V)$ the subspace of marked spacetimes $\p$ such that $\mc C$ subdivides the face cellulation.
\end{notation}

 We note that $\P_c(S, \ms C)$, $\ol{\P}_c(S, \ms C)$ are subsets  of an analytic variety $\mc A(\ms C) \subset \P_c(S, V)$ determined by the coplanarity conditions for the quadruples of points belonging to the same face. These subsets are open and closed respectively with respect to the usual topology. (Here by an analytic variety we mean the intersection of finitely many analytic submanifolds. Note that we cannot guarantee that the intersection is transverse, and hence that $\mc A(\ms C)$ is an analytic submanifold.)

\begin{lemma}\label{lem:lift}
Let $\rho_t$, $t \in (-\e,\e)$, be an analytic curve in $\mc T(S)$. Then there exists a cellulation $\ms C$ of $(S, V)$, which is a subdivision of the face cellulation  $\ms C_0$ of the boundary of $\CH (\p_d(\rho_0))$, and $\e_0$, $0<\e_0<\e$, such that for all $t \in (0, \e_0]$ we have $\p_d(\rho_t) \in \P_c(S, \ms C)$. Moreover, $\p_d(\rho_t)$ is analytic over $[0, \e_0]$.
\end{lemma}

\begin{proof}
Denote $\p_d(\rho_t)$ by $\p_t$. 
Let $\ms T_1, \ldots, \ms T_r$ be all the triangulations refining the boundary cellulation $\ms C_0$ of $\CH(\p_0)$ and let $U$ be a neighborhood of $\p_0$ in $\P_c(S, V)$. For every $i=1, \ldots, r$, an element $\p=(\rho, \tau, f) \in U$ and a triangulation $\ms T_i$ determine a polyhedral surface $\Sigma_i$ in $\Omega^+(\rho, \tau)$, which, provided that $U$ is small enough, is a spacelike Cauchy surface. We also assume that $U$ is small enough so that in any such surface $\Sigma_i$ any edge of $\ms C_0$ has nonzero exterior dihedral angle, so the combinatorics of $\Sigma_i$ is a subdivision of $\ms C_0$. For a celluation $\ms C$ subdividing $\ms C_0$, we denote by $\P_{i}(S, \ms C)$ the set of such triples $(\rho, \tau, f)\in U$ that the resulting $\Sigma_i$ has combinatorics exactly $\ms C$ (which may be empty).
We denote by $\mc I_i: U \rar \mc \M_-(S, V)$ the map sending $\p \in U$ to the induced metric on $\Sigma_i$ (the induced metrics are flat with negative singular curvatures provided that $U$ is sufficiently small). The maps $\mc I_i$ are analytic. It follows from Claim~\ref{vertical} in the proof of Lemma~\ref{diff} that the differentials of all $\mc I_i$ at $\p_0$ coincide with the differential of $\mc I$. In turn,  Lemma~\ref{infrig} implies that $\mc I_i^{-1}(d)$ are $(6g-6)$-dimensional analytic submanifolds around $\p_0$ transverse to $\pi$. We denote by $\p_{i, t}$ $\pi^{-1}(\rho_t)\cap \mc I_i^{-1}(d)$, which is then an analytic curve, provided that $\e$ is sufficiently small. By possibly decreasing $\e$, we assume that it is small enough so that $\p_{i, t}$ and $\p_t$ belong to $U$ for all $t \in (-\e, \e)$ and for all $i$.

We first claim that for every $i$ there exists a cellulation $\ms C_i$ of $(S, V)$ (which is a subdivision of $\ms C_0$) and $\e_i$, $0<\e_i<\e$, such that for all $t \in (0, \e_i]$ we have $\p_{i,t} \in \P_{i}(S, \ms C_i)$. (We remark that a cellulation $\ms C_i$ satisfying this condition is clearly unique.) Indeed, fix $i$, then because of the choice of $U$ there exists $\ms C_i'$ refining $\ms C_0$ and a sequence of positive numbers $t_j'$ converging to zero such that $\p_{i, t_j'} \in \P_{i}(S, \ms C_i') \subset \mc A(\ms C_i')$. 
Since $\p_{i,t}$ is an analytic curve, it belongs to $\mc A(\ms C_i')$.
The set $\P_{i}(S, \ms C_i')$ is an open subset of $\mc A(\ms C_i')$, and $\p_{i, 0}=\p_0$ belongs to the closure of $\P_{i}(S, \ms C_i')$. If $\p_{i, t}$ does not belong to $\P_{i}(S, \ms C_i')$ for all small enough positive $t$, then for every $t_j'$ consider the maximal open segment in $[0, \e]$ containing $t_j'$ such that $\p_{i,t}$ is in $\P_{i}(S, \ms C_i')$ over this segment (because $\P_{i}(S, \ms C_i')$ is open in $\mc A(\ms C_i')$, such a segment exists). Let $t_j''$ be the leftmost endpoint of this segment. Then $t_j''$ is a non-increasing sequence converging to zero, and up to passing to a subsequence there exists a cellulation $\ms C''_i$ such that $\p_{i, t_j''} \in \P_{i}(S, \ms C''_i)$. Moreover, by construction of $\{t''_j\}$, the cellulation $\ms C_i'$ is a strict subdivision of the cellulation $\ms C''_i$. Then the curve $\p_{i,t}$ belongs to the analytic variety $\mc A(\ms C''_i)$. If  $\p_{i, t}$ does not belong to $\P_{i}(S, \ms C''_i)$ for all small enough positive $t$, we can continue this process with a strictly larger cellulation. Since we cannot pass to a strictly larger cellulation infinitely many times, eventually we stop. Thus, we get a cellulation $\ms C_i$ such that $\p_{i, t}$ belongs to $\P_{i}(S, \ms C_i)$ for all small enough positive $t$. (We remark that by uniqueness, our resulting cellulation $\ms C_i$ actually coincides with our starting cellulation $\ms C_i'$.)

We note now that for each $t$ there is $i$ such that $\p_t=\p_{i,t}$, by Lemma~\ref{subdiv}. Since any two analytic curves either coincide, or intersect in a discrete set of points, we see that $[-\e/2, \e/2]$ can be decomposed into finitely many closed segments, over which $\p_t=\p_{i,t}$ for some fixed $i$ for this segment. Thereby, the second statement of the lemma holds, and the first statement for the curve $\p_t$ follows from the statement above for the curves $\p_{i,t}$.
\end{proof}

\subsection{The total length function}

\subsubsection{Definition}\label{sec def total leght}

We saw in the preceding section how to associate to any flat metric $d$ on $S$ with negative singular curvatures a vector field $X_d$ over $\mc T(S)$. We now associate to $d$ a real-valued function over $\mc T(S)$, the \emph{total length}, which can be defined in two ways. 

First, for
$\rho\in \mc T(S)$, if $\ms C$ is the face cellulation of $\CH(\p_{d}(\rho))$, with $\p_{d}(\rho)=\mc I^{-1}(d)\cap \pi^{-1}(\rho)$, then
$$\L_d(\rho):= \sum_{e\in E(\ms C)}\theta_e l_e~,$$
where $\theta_e$ is the dihedral angle at the edge $e$ of $\CH(\p_{d}(\rho))$, and $l_e$ is its $d$-length.

On the other hand, for $\rho\in \mc T(S)$, the Gauss image of $\partial\CH(\p_{d}(\rho))$ defines a balanced cellulation $(\ms G, w)$ over $(S,h)$, where $h$ is the hyperbolic metric on $S$ with holonomy $\rho$ and the weights are the dihedral angles, and we have
$$\L_d(\rho):=\sum_{e\in E(\ms G)}w_e \ell_e~,$$ 
where $\ell_e$ is the $h$-length of the edge. In fact, if $e^*$ is dual to $e$, $l_e=w_{e^*}$ and $\theta_e=\ell_{e^*}$.

We first prove in Section~\ref{sec:smoothness} that $\L_d$ is $C^1$. This result is based on the Schl\"afli Formula. In Section~\ref{sec:proper} we show that $\L_d$ is proper. After preliminary work in Section~\ref{sec:pairing}, in Section~\ref{sec:directional derivatives} we compute the first derivatives of $\L_d$. By investigating its second derivatives, in Section~\ref{sec:WP derivative} we show that it is strictly convex.

\subsubsection{Smoothness}\label{sec:smoothness}

Before discussing our function, we have to make a short sidestep to discuss a useful tool, the Schl\"afli formula in the Minkowski space. We will need it only in a very restricted case: for deformations of tetrahedra close to spacelike degenerate ones. Thereby, we will make our discussion as short as possible, and we refer to \cite{souam,rivin-schlenker2} for a more advanced exposition.

Let $T \subset \R^{2,1}$ be a tetrahedron with all spacelike faces. Particularly, all its edges are spacelike, but are divided into two types, if the exterior normals to the adjacent faces belong (1) to the same component of the unit pseudo-sphere, or (2) to the opposite components. By a (non-oriented) \emph{dihedral angle} $\theta_e$ of an edge $e$ we mean (1) the non-oriented Lorentzian angle between the exterior normals to the adjacent faces, or (2) the non-oriented Lorentzian angle between the exterior normal of one face and the opposite to the exterior normal of the other face, depending respectively if $e$ is of type (1) or (2).

Suppose that the vertices of $T$ undergo an infinitesimal deformations, inducing infinitesimal deformations $\dot \theta_e$ on all the dihedral angles. Then the Schl\"afli formula is the following result:

\begin{equation}
\label{schlafli}
\sum_{e \in E(T)}\dot\theta_e l_e=0~.
\end{equation}

This formula is proven in a much greater generality in \cite[Theorem 2]{souam}. Note, however, that Souam excludes from the consideration the degenerate case, which is what we actually need: when two adjacent faces of an oriented polyhedron have coinciding normals. However, the formula easily extends to our case by continuity. Souam makes his exclusion due to the impossibility to define coherently his \emph{signed dihedral angle} in this situation. However, when the branches of values of dihedral angles are given to every edge during a deformation, the variational formula anyway holds. We note that other proofs can be extracted from the articles \cite{victor}, \cite{rivin-schlenker2}, though the desired result is also not explicitly formulated there.

Let  $\Sigma$ be a spacelike polyhedral Cauchy surface in $\Omega^+(\rho,\tau)$, not necessarily convex. For an edge $e$ of $\Sigma$, its \emph{oriented dihedral angle} is the non-oriented Lorentzian angle between the future normals of the two adjacent sides, taken with plus sign if the edge is future-convex, and with minus sign otherwise. Then we define \emph{the total mean curvature} of $\Sigma$ as
\[M(\Sigma)=\sum_{e \in E(\Sigma)}\theta_e l_e~,\]
where the sum is over all edges of $\Sigma$. We now use the Schl\"affli formula to show

\begin{lemma}
\label{diff2}
$\L_d$ is $C^1$.
\end{lemma}

\begin{proof}
Pick $\p=(\rho,\tau,f) \in \P_c(S, V)$, let $U$ be its neighborhood and $\ms C$ be its face cellulation.
Let $\ms T_1, \ldots, \ms T_r$ be all subtriangulations of $\ms C$. For every $i=1, \ldots, r$, an element $\p'=(\rho',\tau',f') \in U$ and a triangulation $\ms T_i$ determine a polyhedral surface $\Sigma_i$ in $\Omega^+(\rho', \tau')$, which, provided that $U$ is small enough, are spacelike Cauchy surfaces. Define $M_i:=M(\Sigma_i)$, which is then an analytic function over $U$. For $\p' \in (U\cap \ol{\P}_c(S, \ms T_i))$ the surface $\pt \CH(\p')$ coincides with $\Sigma_i$. Thus, over $\ol{\P}_c(S, \ms T_i)\cap \mc I^{-1}(d)$, the function $\L_d\circ \pi$ coincides with $M_i$. Hence, it is enough to show that at the initial configuration $\p$ the differentials of $M_i$ coincide.

We assume that $\ms T_1$ and $\ms T_2$ differ by a flip of one diagonal in a quadrilateral $Q$, and will show the coincidence of the differentials of $M_1$ and $M_2$. Since any two triangulations under our consideration can be connected by a sequence of such flips, this will finish the proof.

Consider the quadrilateral $Q$ as a degenerate spacelike tetrahedron $T$. Let $\dot \p$ be a tangent vector at $\p$ at the initial position. By means of local development of $T$ in $\R^{2,1}$, $\dot \p$ induces a variation on the dihedral edges of $T$. By \eqref{schlafli}, we have
\[\sum_{e \in E(T)}\dot\theta_e l_e=0~.\]

On the other hand, as shown in Claim~\ref{vertical} in the proof of Lemma~\ref{diff}, the differentials of all edge-lengths for two surfaces coincide in the initial position. It is easy to see then, that \[\dot M_1-\dot M_2=\pm \sum_{e \in E(T)}\dot\theta_e l_e=0~.\] This finishes the proof.
\end{proof}

\subsubsection{Properness}\label{sec:proper}

We prove properness of $\L_d$ in Lemma~\ref{proper2}.

\begin{lemma}
For every $\e, c_1, c_2>0$ there exists $c_3>0$ such that if a Euclidean triangle $T$ has $\area(T)\leq c_1$, all side-lengths $\geq c_2$ and one side-length $\geq c_3$, then it has an angle $\geq \pi-\e$.
\end{lemma}

\begin{proof}
Let $T$ be a triangle that satisfy our conditions, with side-lengths $l_1\leq l_2\leq l_3$ and with respective angles $\alpha_1$, $\alpha_2$ and $\alpha_3$. We have
\[\sin(\alpha_2)= \frac{2\area(T)}{l_1l_3}\leq \frac{2c_1}{c_2c_3}~.\]
Since $\alpha_1 \leq \alpha_2$ by the Sine Law, we see that for sufficiently large $c_3$ we get $\alpha_1+\alpha_2 \leq \e$, as desired.
\end{proof}

\begin{cor}
\label{t1}
For every $\e>0$ there exists $c=c(\e,d)>0$ such that if a geodesic triangulation $\ms T$ of $(S, V, d)$ has an edge of length $\geq c$, then it has an angle $\geq \pi-\e$.
\end{cor}

\begin{lemma}
\label{t2}
For every $c_1, c_2>0$ there exists $\e$, $0<\e<\pi$, such that if a hyperbolic polygon $P$ has $\area(P)\geq c_1$ and $\diam(P) \leq c_2$, then all its angles are $\geq \e$.
\end{lemma}

\begin{proof}
Note that $P$ belongs to a circular sector of radius $c_2$ centered at any its vertex, with sides of the sector containing the sides of the polygon adjacent to the vertex. If an angle of the vertex goes to zero, then the area of such sector monotonically decreases to zero. This implies the desired conclusion.
\end{proof}

We now introduce few more notions. Let $d \in \mc \M_-(S, V)$. We will say that a geodesic cellulation $\ms C$ of $(S, V, d)$ is \emph{strict} if all the faces are (isometric to) strictly convex Euclidean polygons. By strictly convex we mean that the angle of each polygon at every point of $V$ is strictly less than $\pi$. Let $\theta: E(\ms C) \rar \R_{\geq 0}$ be a weight function. We will say that $(\ms C, \theta)$ is \emph{$H$-admissible}, if $\ms C$ is strict, all weights are positive, and for every vertex there exists a hyperbolic polygon with edges determined by $\theta$ and with angles complementary to the respective face angles of $\ms C$. This polygon is called the \emph{dual polygon} to $v \in V$. We will say that $(\ms C, \theta)$ is \emph{weakly $H$-admissible} if it is a subdivision of an $H$-admissible $(\ms C', \theta')$ (without adding new vertices and with new edges having zero weight). The \emph{total length} of the weight function is $\sum_{e \in E(\mc C)} \theta_el_e$, where $l_e$ is the length of $e$ in $d$.

 By the dualization construction, every $H$-admissible weighted geodesic cellulation determines a hyperbolic metric $h$. When $\ms C$ is given, we will be calling a function $\theta$ (weakly) $H$-admissible, if it makes a (weakly) $H$-admissible pair with $\ms C$.

\begin{lemma}
\label{finitetriang}
For every $c>0$ there are finitely many triangulations $\ms T$ of $(S, V, d)$ supporting a weakly $H$-admissible weight function of total length $\leq c$.
\end{lemma}
\begin{proof}
Consider a triangulation $\ms T$ with a weakly $H$-admissible weight function $\theta$ such that $L(\ms T, \theta) \leq c$, where $L(\ms T, \theta)$ is the total length.
Let $c_1$ be the minimal distance between two points of $V$ in $d$. Then for every $e \in E(\ms T)$ we have $\theta_e \leq c/c_1$. Because the number of edges of $\ms T$ is fixed, the degree of each vertex is bounded from above, hence for every $v \in V$ the diameter of the dual polygon is bounded from above. Note that its area is minus the curvature of $v$. It follows from Lemma~\ref{t2} that there exists $\e$, $0<\e<\pi$, such that all angles in $\ms T$ are $\leq \pi-\e$. Corollary~\ref{t1} implies that then all the lengths of edges of $\ms T$ are $\leq c_2$ for some $c_2>0$. Since the lengths of geodesics between points of $V$ form a discrete set, there are finitely many such triangulations.
\end{proof}

\begin{lemma}
\label{proper2}
$\L_d$ is proper.
\end{lemma}
\begin{proof}
Let $\rho_i \in \mc T(S)$ be a sequence such that $\L_d(\rho_i)$ converges to $c \in \R_{>0}$. Due to Lemma~\ref{finitetriang}, up to a subsequence, we can assume that all the corresponding cellulations $\ms C_i$ admit a subdivision to the same triangulation $\ms T$. Up to passing to a subsequence, all weights on $\ms T$ converge to a weight function $\theta: E(\ms T) \rar [0, c]$. We need to check that $\theta$ is weakly $H$-admissible. Note that the space of hyperbolic polygons with marked sides and fixed angles is closed (we allow the degenerations of edges to zero, but the result is still a convex hyperbolic polygon as its area is fixed). Hence, by continuity, for every $v \in V$, the weights on the edges incident to $v$ determine a dual hyperbolic polygon. It remains only to check that the decomposition $\ms C$ obtained from $\ms T$ by deleting all edges of zero length, is a strict cellulation (particularly, that all faces are simply connected). Since all the dual polygons are convex hyperbolic polygons, there are at least three edges in $\ms C$ from every vertex, and all face angles in $\ms C$ are $<\pi$. Suppose that a face of $\ms C$ is not simply connected. That would mean that there is a non-simply connected compact flat surface with strictly convex boundary, which cannot happen due to the Gauss--Bonnet formula. 
Hence, $\ms C$ is a strict cellulation and $(\ms C, \omega)$ determines $\rho \in \mc T(S)$, which is the limit of $\rho_i$.  
\end{proof}

\subsection{Pairing of Codazzi tensors and balanced cellulations}\label{sec:pairing}

The aim of the present Section is to prove Formula \eqref{eq:la formule mixte}, which will be used to relate the differential of the total length function to the Weil--Petersson metric.
 
\subsubsection{Background on Codazzi tensors}

Let $b$ be a (smooth) symmetric $(0,2)$ tensor over the hyperbolic surface $(S,h)$. 
Recall that if $\nabla$ is the Levi-Civita connection of $h$, then

 \begin{equation}\label{eq:def cod 2}
 \nabla_X b (Y,Z)= X.b(Y,Z) - b(\nabla_XY,Z) - b(\nabla_XZ,Y)~.
 \end{equation}
Since $b$ is symmetric, we have
$$\nabla_Xb(Y,Z)=\nabla_Xb(Z,Y)~.$$
We will say that a symmetric $(0, 2)$ tensor $b$ is \emph{Codazzi} if
$$\nabla_Xb(Y,Z)=\nabla_Yb(X,Z)~. $$
 Recall that the divergence $\div  b$ of $b$ is the $1$-form defined as 
$$\div(b)(X)=\operatorname{tr}_h(\nabla_. b(X,\cdot))~. $$ 
 
\begin{lemma}\label{lem: cod et div}
If $b$ is a symmetric traceless $(0, 2)$ tensor over $(S,h)$, then Codazzi is equivalent to 
 divergence-free. 
\end{lemma} 
\begin{proof} 
 The almost-complex  structure $J$ defined by $h$ in conformal coordinates at a point is given by $\begin{pmatrix}0 & -1 \\ 1 & 0 \end{pmatrix}$. So it is immediate to check that 
  $bJ$ and $Jb$ are symmetric and traceless and  that  
\begin{equation}\label{eq:jbj}
bJ=-Jb~.
\end{equation} 
In particular,
$$b(X,JY)=(bJ)(X,Y)=(bJ)(Y,X)=b(Y,JX)~, $$
and also
\begin{equation}\label{eq:jbj2}
b(JX,JY)=b(Y,JJX)=-b(X,Y)~.
\end{equation}
It is easy to see that $J$ is parallel for $\nabla$, for example by expressing $\nabla$ in conformal coordinates. In turn, $J$ and $\nabla$ commutes:
$$0=(\nabla_X J)(Y)=\nabla_X(JY)-J\nabla_XY~.$$

From the equations above and from \eqref{eq:def cod 2} we easily check that
 $$\nabla_X (bJ)(Y,Z)=\nabla_Xb(Y,JZ)~.$$

We first prove that if $b$ is Codazzi, then  $bJ$ divergence-free. Indeed, if $e_i$ is an orthonormal frame for $h$, we compute for example
 $$\div(bJ)(e_1)=\nabla_{e_1}b(e_1,e_2)-\nabla_{e_2}b(e_1,e_1)$$
 which is equal to zero because $b$ is symmetric and Codazzi. The same holds for $e_2$. 
 
We also have that if $b$ Codazzi, then  $bJ$ is Codazzi. Indeed,
$$\nabla_X(bJ)(Y,Z)=\nabla_Xb(Y,JZ)=\nabla_Yb(X,JZ)=\nabla_Y (bJ)(X,Z)~. $$
 
 Now, if $b$ is Codazzi, $bJ$ is Codazzi, hence $bJJ=-b$ is divergence free.

In the same way, it is immediate to check that if $b$ is divergence free, then $bJ$ is divergence free and Codazzi. 
 \end{proof}
 
 For future reference, let us note the following consequence of the Divergence Theorem.

\begin{lemma}\label{lemma divergence}
Let $f, \varphi$ be  smooth functions on a hyperbolic surface with boundary $(M,h)$ and suppose that $b$ is symmetric and divergence free. Then
 $$ \int_M \tr_h(\varphi b\nabla^2f) = -\int_M  b(\nabla f,\nabla \varphi) + \int_{\partial M}  \varphi b(\nabla f,N)   $$
where $N$ is the unit outward normal of $\partial M$. 
\end{lemma}
The integrations are implicitly done with respect to the volumes forms of $M$ and $\partial M$ respectively. Also, for  
 a $(0,2)$ symmetric tensor $b$, we will denote by $b^\sharp$ the self-adjoint $(1,1)$ tensor such that 
$b(X,Y)=h(b^\sharp(X),Y)$, and hence for two such tensors $b_1$ and $b_2$, $\tr_h(b_1b_2)$ means $\tr(b_1^\sharp b_2^\sharp)$.
\begin{proof} 
Recalling that $\nabla_Xb^\sharp=\left(\nabla_X b\right)^\sharp$ (see e.g. \cite[2.66]{ghl}), we have

$$\div(b)(X)=\operatorname{tr}(\nabla_{.} b^\sharp(X))=\operatorname{tr}(\nabla_. (b^\sharp(X)))- \operatorname{tr}(b^\sharp(\nabla_. X))=\div (b^\sharp(X))-\operatorname{tr}(b^\sharp(\nabla_. X))~.$$ 
Thereby,
$$\div(b)(\nabla f)=\div (b^\sharp(\nabla f))-\operatorname{tr}(b^\sharp(\nabla_. \nabla f))
=$$ $$= \div (b^\sharp(\nabla f))-\operatorname{tr}(b^\sharp (\nabla^2 f)^\sharp)=
\div (b^\sharp(\nabla f))-\operatorname{tr}_h(b \nabla^2 f)~.$$
Hence if $b$ is divergence free, multiplying by $\varphi$ and reordering
$$\operatorname{tr}_h(\varphi b \nabla^2 f)=
\varphi \div (b^\sharp(\nabla f))=\div (\varphi b^\sharp(\nabla f))- b^\sharp(\nabla f).\varphi~.$$
Integrating and using the Divergence Theorem leads to the result.
\end{proof}

\subsubsection{Traceless Codazzi tensor associated to a balanced cellulation}

Let us consider $\p=\p_{d}(h)\in \P_c(S,V)$, and let
$s_\p$ be the support function of $P=\widetilde\CH(\p)$. We already used support functions in the proof of Lemma~\ref{lem:orbite infinie}, but here we assume that the domain of a support function is the hyperboloid $\H^2$. In more details, for $x \in \H^2$,
\[s_\p(x):=\sup \{\langle p, x \rangle: p \in P\}~.\]
It is known to be well-defined over $\H^2$, see~\cite{bonsante}. For more details on support functions in this setting, we refer to~\cite{FV, BF}.
%
The function $s_\p$ in general is not $\rho(\pi_1S)$ invariant, but as for $x \in \H^2$ there exists a point $p$ on $P$ such that $s_\p(x)=\langle p,x\rangle$, then $$s_\p(\rho(\gamma)(x))=\langle \rho_\tau(\gamma)(p),\rho(\gamma)(x)\rangle~,$$ hence we have
\begin{equation}\label{eq:equiariance}
s_\p(\rho(\gamma)(x))=s_\p(x)+\langle \rho^{-1}(\gamma)(\tau(\gamma)),x\rangle~.
\end{equation}

If $(\ms G, w)$ is the balanced cellulation of $(S,h)$ defined by  $\p$,
over any face $F$ of the lift of $\ms G$ on $\H^2$, $s_\p$ is the restriction  of the linear map
 $x\mapsto \langle x,v\rangle$, where $v$ is the vertex of $P$ sent by the Gauss map to $F$.

\begin{prop}\label{prop: codazzi cocycle}
 There exists a unique smooth function $s_{d}$ on $\mathbb{H}^2$ such that it satisfies \eqref{eq:equiariance} and $\nabla^2 s_d-s_d \tilde h$ is a traceless symmetric Codazzi tensor which is $\rho$-invariant, where $\tilde h$ is the metric over $\H^2$.
\end{prop}

 We denote by $b_{d}$ the projection of this tensor onto the hyperbolic surface $(S,h)$. We refer to \cite{FV,BS,BF,OlikerSimon} for details.

\begin{remark}
To simplify computations, it is worth to note (see e.g. \cite{FV}) that if $F$ is the one-homogeneous extension to the future cone of the origin of a smooth function  $f$ over the hyperboloid,  then at $x\in \mathbb{H}^2$,
\begin{equation}\label{eq:gradient}
\grad F(x)= \nabla f(x) - f(x)x~;
\end{equation} 
\begin{equation}\label{eq:hessien}\operatorname{Hess} F(x)= (\nabla^2-\tilde h)(x)f~.\end{equation}
\end{remark}

\begin{lemma}\label{lem fuch noyau}
Let $b$ be a traceless divergence-free $(0, 2)$ tensor  on  $(S,h)$. Let $\{O_i,i\in I\}$ be a finite covering of $S$ by discs and let $s$ be a $C^2$ function on $\H^2$ satisfying \eqref{eq:equiariance}. For any $i$, 
let us denote by $s_i$ the function over $S$ with support in $O_i$ defined by the restriction of $s$ to an arbitrary lift of $O_i$.
Let $\{\varphi_i, i\in I\}$ be a partition of unity subordinated to $\{(O_i)_{i\in I}\}$. 
Then
$$\int_S \tr_h(bb_{d})  =-\sum_i \int_S  b(\nabla s_{i},\nabla \varphi_i)~. $$
\end{lemma}
Note that by \eqref{eq:equiariance} two different choices of lifts of $O_i$ will result by adding to $s_i$ the restriction to $\H^2$ of a linear map.
\begin{proof}
By  \eqref{eq:equiariance}, the function $s-s_d$ is $\rho$-invariant, hence it defines a function $f$ over $S$. Let $s_{0,i}=s_i-f$. 
Then
$$\int_S \tr_h(bb_d)
=\sum_{i\in I} \int_S \varphi_i \tr_h(bb_d) 
=\sum_i \int_S \tr_h(\varphi_i b(\nabla^2 -h)s_{0,i})~.$$ 
As $b$ is traceless, we obtain
$$\int_S \tr_h(bb_d) = \sum_i \int_S \tr_h(\varphi_i b\nabla^2 s_{0,i})~.  $$

Lemma~\ref{lemma divergence} gives (because the support of $\varphi_i$ is in $O_i$, there is no boundary term)
$$\int_S \tr_h(bb_d) = -\sum_i \int_S  b(\nabla s_{0,i},\nabla \varphi_i)~. $$
Finally,
$$\sum_i \int_S  b(\nabla s_{0,i},\nabla\varphi_i)=\sum_i \int_S  b(\nabla s_{i},\nabla\varphi_i)~,$$
because
$$\sum_i \int_S  b(\nabla f,\nabla\varphi_i) 
= \int_S  b(\nabla f,\nabla 1)=0 ~.$$
\end{proof}

\begin{lemma}\label{lem trace mixte}
Let $b$ be a traceless symmetric Codazzi $(0,2)$ tensor  on $(S,h)$, and let $(\ms G, w)$ be the balanced cellulation of $(S,h)$ given by $\p$. Then 
\begin{equation}\label{eq:la formule mixte}\int_S \tr_h(bb_d) = -\sum_{e\in E(\ms G)}w_e \int_e b(U_e,U_e)~,\end{equation}
where $U_e$ is a unit tangent vector of $e$. 
\end{lemma}
\begin{proof}
In the proof we will denote $s_\p$ by $s$. 
 Let $\{\varphi_i, i\in I\}$ be a partition of unity subordinated to an open cover  $\{O_i, i\in I\}$ of $S$ such that 
 \begin{itemize}[nolistsep]
\item on each $O_i$ we have well-defined functions $s_i$ and $s_{0,i}$,
\item $I=I_f\sqcup I_e \sqcup I_v$ where
 if $i\in I_f$, $O_i$ is contained in a face of $\ms G$,
if $i\in I_e$, $O_i$ meets exactly one edge of $\ms G$,
 if $i\in I_v$, $O_i$ meets exactly one vertex of $\ms G$ and no edges but the edges emanating from this vertex.
\end{itemize} 

 Lemma~\ref{lem fuch noyau} was proved for $C^2$ functions satisfying \eqref{eq:equiariance}. However, as $s$ is a support function, it can be approximated by $C^2$ support functions satisfying \eqref{eq:equiariance} \cite[Appendix A]{BF}, and on $S$ minus the edges and vertices the gradients also converge by convexity. Actually the proof of  Lemma~\ref{lem fuch noyau} could be directly implemented in the present proof without using any approximation, but we have preferred to state a separate statement.

By Lemma~\ref{lem: cod et div}, $b$ is divergence-free. 
Denoting $A_i=\int_S  b(\nabla s_{i},\nabla \varphi_i)$, from
 Lemma~\ref{lem fuch noyau},
\begin{equation}\label{eq:def A_i}\int_S \tr_h(bb_d)= -\sum_i A_i~. \end{equation}

Let us compute $A_i$ with respect to the different sets of indices $i$.
\begin{enumerate}
\item Let $i\in I_f$. Then on $O_i$ the function $s_i$ is  $C^\infty$, so by 
Lemma~\ref{lemma divergence},
 $$A_i= -\int_{O_i} \tr_h(\varphi_i b\nabla^2s_i)$$
(there is no boundary term as the support of $\varphi_i$ is included in $O_i$). As  $b$ is traceless, 
$$A_i= -\int_{O_i} \tr_h(\varphi_i b(\nabla^2s_i-s_ih))$$
and as  $s_i$ is  the restriction  of a linear map, by \eqref{eq:hessien}
$\nabla^2s_i-s_ih=0$. Hence $A_i=0$.
\item Let $i\in I_e$. We have an edge $e$ which separates 
$O_i$: 
$$O_i=O_i^- \sqcup O_i^+ \sqcup (O_i\cap e)~. $$
Let $N_i^-=-N_i^+$ be  the outward unit normal of $\pt O_i^\pm$ (the choice of $+$ or $-$ being arbitrary). Let $s_i^\pm$  be the extensions of $s_i|_{O_i^\pm}$ to a neighborhood of $O_i^\pm$, coming from the restriction of linear maps.
By Lemma~\ref{lemma divergence},
\begin{equation*}
\begin{split}
\ A_i= &-\int_{O^+_i} \tr_h(\varphi_i b(\nabla^2s_i^+))- \int_{O^-_i} \tr_h(\varphi_i b(\nabla^2s_i^-)) \\ & + \int_{\partial O_i^+} \varphi_i b(\nabla s_i^+,N_i^+) 
+  \int_{\partial O_i^-} \varphi_i b(\nabla s_i^-,N_i^-)~.
\end{split}
\end{equation*}

By the same reason as in the $i\in I_f$ case,   $\int_{O^\pm_i} \tr_h(\varphi_i b(\nabla^2s_i^\pm))=0$,
so
$$A_i= \int_{\partial O_i^+} \varphi_i b(\nabla s_i^+,N_i^+)+  \int_{\partial O_i^-} \varphi_i b(\nabla s_i^-,N_i^-)~.$$
Now  $\varphi_i$ has compact support in $O_i$, hence 
$\partial O_i^\pm$ reduces to $e\cap O_i$. It will be convenient to write them as integrals on $e$:
$$A_i= \int_{e} \varphi_i b(\nabla s_i^+,N_i^+)+  \int_{e} \varphi_i b(\nabla s_i^-,N_i^-)~,$$
and clearly  on $e\cap O_i$, $N_i^-= -N_i^+$. So 
 $$A_i= \int_{e}\varphi_i  b(\nabla s_i^+-\nabla s_i^-,N_i^+)~.$$
 
It is easy to compute that  over $e$,
\begin{equation}\label{eq:diff grad}\nabla s_i^+-\nabla s_i^-=-w_e N_i^+~.\end{equation}

Indeed, there are vertices $v_\pm$ of $\widetilde\CH(\p)$ on an edge $e^*$, whose Gauss image is a lift of $e$,   such that $s_i(x)=\langle x,v_\pm\rangle$ (here we abuse notation identifying $x$ as a point in $O_i$ and as a point in $\H^2$). By definition, $w_e$ is the length of $e$, and $v_+-v_-=w_eN_i^-$ (here we also abuse the notation and consider $N_i^-$ as the unit vector in $\R^{2,1}$ normal to the plane defining the edge $e$). From  \eqref{eq:gradient}, and because for all $x\in e$, $\langle x,N_i^+\rangle=0$, we obtain \eqref{eq:diff grad}. So
 we are left with 
$$A_i=-w_e\int_{e} \varphi_i b(N_i^+,N_i^+)~.$$
\item Let $i\in I_v$. Around a vertex $v$, number  the emanating edges in a clockwork fashion $e_1,\ldots,e_n$, and also in a cyclic manner, $e_{n+1}=e_1$, etc. Let us denote by $O_{i,k}$ the part of $O_i$ contained between the edges $e_k,e_{k+1}$. Now we can proceed exactly as in point 2. The only difference is now that the boundary of $O_{i,k}$ on which $\varphi_i$ is not zero is the union of two geodesic segments, so we need to use the version of the Divergence Theorem for manifolds with corners. But for such manifolds, only the codimension one boundary component enters Stokes theorem, so we may apply the same formula, see e.g. \cite{lee}. In turn,
$$A_i= -\sum_{k=1}^n w_{e_k} \int_{e_k} \varphi_{k}b(N_i^+,N_i^+)~. $$

\end{enumerate}

Returning to \eqref{eq:def A_i}, to compute 
$\int_S \tr_h(bb_d)$ we need to 
sum all the $A_i$. This will lead to the result by definition of a partition of unity. Indeed, let $e$ an edge of the cellulation. Suppose that for example $e=(e\cap O_1)\cup (e\cap O_2) \cup (e\cap O_3)$ i.e. $O_1$ and $O_2$ will each contain a (different) vertex of the cellulation, while $O_2$ don't. So
$$\int_{e} \varphi_{1}b(N_e,N_e)+\int_{e} \varphi_{2}b(N_e,N_e)+\int_{e} \varphi_{3}b(N_e,N_e)=\int_{e} b(N_e,N_e)~,$$
where $N_e$ is any unit normal to $e$.

Hence from \eqref{eq:def A_i} we have  $\tr_h(bb_d) = \sum_{e}w_e \int_e b(N_e,N_e)$
and the result follows because as $U_e=JN_e$,  from  \eqref{eq:jbj2} we get
$b(N_e,N_e)=-b(U_e,U_e)$.
\end{proof}

\subsubsection{Weil--Petersson metric}

From now on we consider $\mc T(S)$ as a subspace of the space of hyperbolic metrics on $S$, transverse to the action by the group of isotopies. 
 For this point of view, the tangent space of $\mc T(S)$ at $h$ is the space of symmetric traceless Codazzi $(0,2)$ tensors over $(S,h)$, see e.g. \cite{tromba,yamada,BS} (note Lemma~\ref{lem: cod et div}). In the following, we will consider implicitly paths $h(t)$ in $\mc T(S)$ such that $h'(0)$ is a non-zero symmetric traceless Codazzi tensor.

By a folklore result, known as \emph{Hopf Lemma} (see e.g. \cite[Lemma 3.1]{KT}), $b$ is a symmetric traceless Codazzi tensor if and only if 
\begin{equation}\label{def:Qb}
Q(X)= b(X,X)+i Jb(X, X)~ \end{equation}
is a holomorphic quadratic differential over the Riemann surface underlying $(S,h)$. 

In a local chart, a holomorphic quadratic differential is expressed as $Q=q(z)dz^2$. If the metric is expressed as $h=\chi(z)dzd\bar z$, then for two holomorphic quadratic differentials $Q_1$ and $Q_2$ the function $\frac{q_1\bar q_2}{\chi}$ is independent on the choice of coordinates and is denoted by $\frac{Q_1\bar Q_2}{h^2}$.
The \emph{Weil--Petersson cometric} is defined as
$$\langle Q_1,Q_2\rangle_{\textsc{wp}}=\mathrm{Re}\int_S \frac{Q_1\bar Q_2}{h^2}~.$$
With \eqref{def:Qb}, one compute easily (see \cite{tromba}) that for two symmetric transverse Codazzi tensors $b_1,b_2$ over $(S,h)$,
\begin{equation}\label{eq:WP1}\langle X(b_1),X(b_2)\rangle_{\textsc{wp}}=\frac{1}{2}\int_S \tr_h(b_1b_2)~,\end{equation}
where $X(b)$ is the tangent vector of Teichm\"uller space associated to the Codazzi tensor $b$, as a first order deformation of the hyperbolic metric.

Now recall the  tangent vector field $X_d$ over $\mc T(S)$ defined by the flat metric $d$ (Notation~\ref{not: m X}). The vector $X_d(h)$ defines a balanced cellulation $(\ms G, w)$, which in turn defines a function $s_\p$ and hence a symmetric traceless Codazzi tensor $b_d$, as explained in the preceding section. 
Here, the symmetric traceless Codazzi tensor $b_d$ is not identified with a first-order deformation of hyperbolic metrics, as in \eqref{eq:WP1}, 
but rather with an element of
$\mc H^1(\rho)$. Hence there are two ways to identify $b_d$ with a tangent vector of Teichm\"uller space, but they differ only by the almost-complex structure $\mc J$ of $\mc T(S)$ \cite[Theorem B]{BS}:
\begin{equation}\label{eq:J acts}
\mc J X(b_d)=X_d(h)~.
\end{equation}

This almost complex structure for holomorphic quadratic differential is the left multiplication by the almost complex structure $J$ of $(S,h)$,  
 and then $\mc J$ acts on symmetric traceless Codazzi tensors as $b\mapsto bJ$ (recall \eqref{eq:jbj}).
 At the end of the day, using \eqref{eq:J acts} to implicitly identify $X_d(h)$ with a tangent vector of Teichm\"uller space as a space of metrics,  
 \eqref{eq:WP1} and \eqref{eq:la formule mixte}, we get that if $\omega$ is the Weil--Petersson symplectic form, then
 \begin{equation}\label{eq:WP et integrale}
 \begin{split}
\omega(X(b),X_d(h))&=\omega(X(b),\mc J X(b_d))=\langle X(b),X(b_d)\rangle_{\textsc{wp}}=\\ &=\frac{1}{2}\int_S \tr_h(bb_d)=-\frac{1}{2}\sum_{e}w_e \int_e b(U_e,U_e)~.
\end{split} \end{equation}

For future reference, let us consider a symmetric traceless Codazzi tensor $b$, which is a first-order deformation of a path of hyperbolic metrics $h(t)$ with $h(0)=h$ along a 
Weil--Petersson geodesic. From \cite[Theorem 3.8]{yamada},
we have 
\begin{equation}\label{eq:derivee seconde le lon WP}h''(0)=\left(\frac{1}{4}-\frac{1}{2}(\Delta_h -2)^{-1}\right)\|b\|^2 h~, \end{equation}
with $\|b\|^2=\tr_h(bb)$ and $\Delta_h$ the Laplacian of $h$. 
This formula is also proved in \cite[(3.4)]{wolf}. In that reference $h'(0)$ 
corresponds to 2 times our Codazzi tensor $b$, and with this difference taken into account, the both formulas are the same. Note that the Weil--Petersson metrics they both consider differ by a factor $2$, but that does not change the parametrization of the Weil--Petersson  geodesics with a given tangent vector.

\subsection{First derivatives}\label{sec:directional derivatives}

Let $h(t)$, $t\in(-\e,\e)$, be an analytical path of hyperbolic metrics in $\mc T(S)$, sufficiently small so that the conclusion of Lemma~\ref{lem:lift} occurs.
For each $t>0$, let $P(t)=\partial \widetilde\CH(\p_d(h))$.
 The quotient  of the image by the Gauss map of $\partial P(t)$ gives a  weighted graph $(\ms G_t, w_t)$  over $S$. More precisely, there exists a graph $G$ and a piecewise smooth mapping $f(t)$ from $G$ to $\ms G_t$.

From Lemma~\ref{lem:lift}, it follows that the dependence on $t>0$ of the vertices of $P(t)$ in Minkowski space is analytic. As the combinatorics of $P(t)$ is fixed, and as the Gauss map can be written as a cross product of vertices, for a vertex $v\in G$ the map $t\mapsto f(t)(v)$ is smooth (note that the graph is sent to the smooth surface $S$, as the hyperbolic metric changes over $S$, we cannot keep the conformal structure over $S$). 

By construction, for each $t>0$, $(\ms G_t, w_t)$ is a  balanced cellulation over $(S,h(t))$, whose dual metric is (isometric to) $d$. In turn, the edge lengths of $P(t)$ are constant, so the weights $w_t$ does not depend on $t$, and we will write them $w$.

Let us insist that for $t>0$, for all edges $e$ of $G$, the image of $f_e(t)$ does not collapse to a point. Also for $t=0$, if a connected subgraph of $G$ collapses to a point, then it is a tree. 
The same conclusion occurs for $t<0$, but the underlying graph could be different from $G$.
We will call \emph{admissible} such a path 
$(\ms G_t, w,h(t))$.

The fact that the dual metric is prescribed allows to define unit tangent vector also for edges reduced to a point, as explained in the lemma below.

\begin{lemma}\label{lem:continuite Ut}
Let $(\ms G_t, w,h(t))$ be an admissible path.
 For $t>0$ let $U_{e,v}(t)$ be the $h(t)$-unit tangent vector  at the vertex $f_v(t)$  of  $f_e(t)$. Then $U_{e,v}(t)$ converges to a $h(0)$-unit tangent vector $U_{e,v}(0)$ at $f_v(0)$.
\end{lemma}

\begin{proof}
Let $G' \subset G$ be a connected subgraph collapsing to a point at $t=0$ and $v$ be its leaf. Then $v$ has an edge $e$ that does not collapse to a point. Hence we have the well-defined tangent vector $U_{e,v}(0)$. The dual Euclidean metric determines all angles between the tangent vectors to the edges at $v$ as the supplementary angles between the dual edges in $d$, see Figure~\ref{fig:path_cellulation}. Hence we can reconstruct the tangent vectors for any edge of $G'$ at $v$. For an edge $e'$ of $G'$ adjacent to $v$ and $v'$ the tangent vector $U_{e',v'}(0)$ is determined just as $-U_{e', v}(0)$. Then again we can reconstruct all tangent vectors at $v'$ with the help of the angles of $d$. Proceeding this way we reconstruct all tangent vectors at the edges of $G'$. It is easy to see that the result is independent from the sequence, in which we are doing the reconstruction, and every obtained tangent vector is the limit of the respective tangent vectors as $t \rar 0+$.
%
\end{proof}

In particular, for $t\geq 0$ the following balance condition holds: for any vertex $v$ of $G$,
\begin{equation}\label{eq:balanced 1 t}
\sum_{E\ni e\ni v} w_e U_{e,v}(t)=0~.
\end{equation}
Moreover, let $G_0$ be the graph obtained from $G$ by contracting the components of collapsed edges at $t=0$, so $\ms G_0$ is its geodesic realization in $(S, h(0))$. We denote by $E_0$ the set of its edges (we may refer to them as to \emph{visible edges} and to the vertices of $G_0$ as to \emph{visible vertices}). Then at $t=0$ for a vertex $v$ of $G_0$ we still have the balance condition
\begin{equation*}
\sum_{E_0\ni e\ni v} w_e U_{e,v}(0)=0~,
\end{equation*}
as for a collapsed edge $e \in E$ with vertices $v$ and $v'$ we have $U_{e,v}(0)=-U_{e, v'}(0)$.
 
Of course, a similar result holds for $t\leq 0$, but the vectors $U_{e,v}(0)$, defined by continuity, are not necessarily the same for $t>0$ or $t<0$, see 
Figure~\ref{fig:path_cellulation}.

\begin{figure}[h!]
\begin{center}
\includegraphics[scale=0.3]{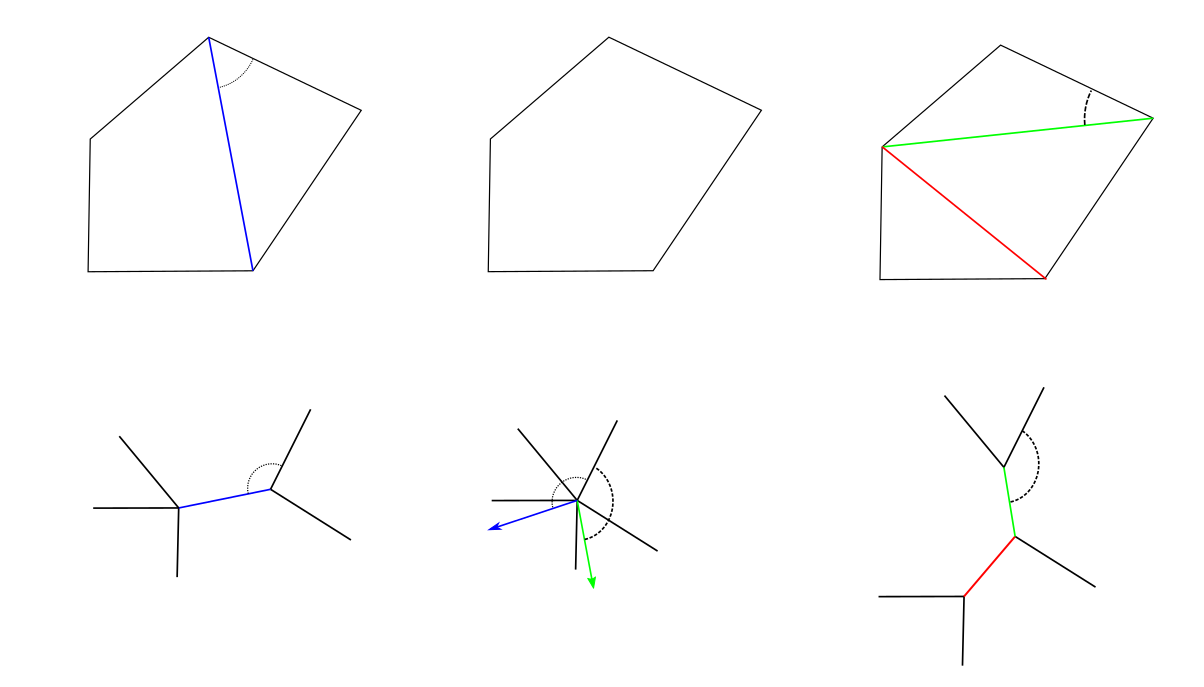}
\caption{\footnotesize	 To Lemma~\ref{lem:continuite Ut}: vectors $U_{e,v}(0)$ for $t>0$ and $t<0$ are different.}
\label{fig:path_cellulation}\end{center}
\end{figure}

For any edge $e\in E(G)$ and $t\in [0,\e)$, we consider a parameterization of the $h(t)$-geodesic $f_e(t)$ over $[0,1]$ proportionally to the arc length, and we will denote $f_e(t)(s)$ by $f_e(t,s)$, so $f_e$ is now a continuous function from $[0,\e)\times [0,1]$ to $S$. In the case of collapsed edges, at $t=0$ the map $f(0, s)$ sends $[0,1]$ to a point. By considering that the displacement of lifts of points in the hyperbolic plane is smooth, it follows that $f$ is smooth over  $(0,\e)\times [0,1]$.
We denote by $V_e$ the image by the tangent map of $f_e$ of $\frac{\partial}{\partial t}$ and by $T_e$ the image by the tangent map of $f_e$ of $\frac{\partial}{\partial s}$. We will denote by $U_e(t,s)$ the $h(t)$-unit tangent vector of $f_e(t)$ at $f_e(t,s)$. Comparing with notation of Lemma~\ref{lem:continuite Ut}, if $f_e(t,0)=f_v(t)$, then $U_{e,v}(t)=U_e(t,0)$.
In particular, if $f_e(t)$ is not reduced to a point (which does not occur for $t>0$),
\begin{equation}\label{eq:T et U}T_e(t,s)=  \ell_e(t)U_e(t,s)~, \end{equation}
where  $\ell_e(t)$ is the $h(t)$-length of the edge $f_e(t)$ (which may be zero).

 Note that in general,     
 if $\nabla^t$ is the Levi-Civita connection of $h(t)$,  $\nabla^t_{V}T$
is meaningless. However, there is a notion of connection defined by the pull-back over the variation of $f$, which allows formally the same computations as if $V$ and $T$ were defined in a neighborhood, so we will abuse the notation and write $\nabla^t_V T$, see e.g. Section 3.B.1 in \cite{ghl} or Section 6.1 in \cite{jost}. In particular, as $V$ and $T$ are images of coordinates fields,
\begin{equation}\label{eq:crochet 0 TU}\nabla^t_VT= \nabla^t_T V~. \end{equation}

\begin{remark}\label{rem:somme sur V}
A remark that we will often use is that at a vertex $v$ the vector $V$ does not depend on the edges starting at $v$, conversely to $U$ and $T$. In turn,
\eqref{eq:balanced 1 t}  implies, for any $t_1$ and $t_2$, 
\begin{equation}\label{eq:wsum zero derivative}\sum_{e\ni v} w_e h(t_2)(U_e(t_1,0),V(t_1,0))=0~.\end{equation} 
 \end{remark}

For an admissible path $(\ms G_t, w,h)$, we will look at  the total length of the balanced cellulation at $t$:

\begin{equation}\label{eq:Lt 1}L(t):=\L_d(h(t))=\sum_{e\in E} w_e \ell_e(t)~, \end{equation}
which is of course continuous over $(-\e,\e)$. (In the present section, $E$ is the set of edges of $G$.)

 \begin{lemma}\label{lem:L'} Along an admissible path 
\begin{equation}\label{eq:L'}L'(t)=\frac{1}{2}\sum_{e\in E} w_e\ell_e \int_0^1 h'(t)(U_e(t,s),U_e(t,s))ds~. \end{equation}
\end{lemma}

We will often skip the points $(t,s)$ from $U$ and $V$ when the context makes them clear. 
\begin{proof}
We recall that $L$ is $C^1$ (because $\L_d$ is $C^1$). We prove the formulas for $t\not=0$, for which there are no collapsed edges, and notice that the formula extends continuously to $t=0$. (It is actually in general true that a continuous function over $(-\e,\e)$ with continuous derivative over $(-\e,\e)\setminus \{0\}$ that extends continuously at $0$ is $C^1$.)

Let $t\not=0$. Let us denote by  $\L_{h(u)}(f(\alpha))$ the total length for the metric $h(u)$ of the the graph at time $\alpha$, so that $f(\alpha)$ is geodesic for $h(\alpha)$ and  $L(t)=\L_{h(t)}(f(t))$.
We have
$$L'(t)=\frac{d}{du}\L_{h(u)}(f(t))|_{u=t}+\frac{d}{d\alpha}\L_{h(t)}(f(\alpha))|_{\alpha=t}~. $$

The point is that 
$\frac{d}{d\alpha}\L_{h(t)}(f(\alpha))|_{\alpha=t} $ is a weighted sum of variations of geodesics, which is equal for each $e$ (see e.g.  \cite[(6.1.4)]{jost}), to $\frac{1}{\ell_e(t)}h(t)(T,V)|_{s=0}^{s=1}=h(t)(U,V)|_{s=0}^{s=1}$, whose weighted sum is equal to zero by Remark~\ref{rem:somme sur V}. Hence,
\begin{equation}\label{eq:equation geodesiques}\frac{d}{d\alpha}\L_{h(t)}(f(\alpha))|_{\alpha=t} =0~.\end{equation}

Now
\begin{equation*}
\begin{split}
L'(t)&=\frac{d}{du}\L_{h(u)}(f(t))|_{u=t}=\frac{1}{2}\sum_{e\in E} w_e\int_0^1 \frac{h'(t)(T,T)}{h(t)(T,T)^{1/2}}ds \\ &=\frac{1}{2}\sum_{e\in E} w_e\int_0^1 h'(t)(U,U)
h(t)(T,T)^{1/2}ds =\frac{1}{2}\sum_{e\in E} w_e\ell_e\int_0^1h'(t)(U,U)ds\end{split}\end{equation*}
and this last expression converges to the same expression when
$t\to 0^\pm$.
\end{proof}

Recall the tangent vector field $X_d$ over $\mc T(S)$ defined by the flat metric $d$ (Notation~\ref{not: m X}).

\begin{prop}\label{prop:WP gradient}
The vector field $X_d$ is minus the Weil--Petersson symplectic gradient of $\L_d$.
\end{prop}
\begin{proof}
Let $b$ be a symmetric traceless Codazzi tensor over $(S,h)$, which is tangent to the analytical path $h(t)$ at $h(0)=h$ in $\mc T(S)$. From \eqref{eq:L'},
if $f_e(0)$ is a point, then the corresponding integral does not contribute to the sum. In particular,
\begin{equation*}\label{eq:L'(0)}\L_d(h(t))'|_{t=0}=\frac{1}{2}\sum_{e\in E_0} w_e \int_{e}b(U_e,U_e)~. \end{equation*}

As  $\L_d$ is $C^1$ by Lemma~\ref{diff2}, we have
$d\L_d(h)(b)=\frac{1}{2}\sum_{e\in E_0} w_e \int_{e}b$
and the result follows from \eqref{eq:WP et integrale}.
\end{proof}

\begin{cor}\label{cor:C2}
$\L_d$ is $C^2$.
\end{cor}
\begin{proof}
This is because $X_d$ is $C^1$ (Lemma~\ref{lem:Xd C1}).
\end{proof}

\begin{remark}
Recall that the study of the Weil--Petersson symplectic gradient of the length of a simple closed geodesic is classical and goes to works~\cite{wolpert82, wolpert83, wolpert87} of Wolpert. In particular, Wolpert established a nice formula for the Weil--Petersson symplectic form evaluated on two such vectors. An analogue of this for balanced celluations $(\ms G, w)$ and 
$(\ms G', w')$ over a hyperbolic surface $(S,h)$ was computed in~\cite{FS}.
Namely, if $d$ and $d'$ are the respective dual flat metrics, then
$$w(X_{d}(h),X_{d'}(h))=\frac{1}{2}\sum_{p\in e\cap e'} w_ew_e' \cos \theta_{ee'}~, $$
where $\theta_{ee'} $
is the angle of intersection
between $e$ and $e'$ according to the orientation of $S$.
\end{remark}

\subsection{Second derivative along Weil--Petersson geodesic}\label{sec:WP derivative}

We continue to differentiate the length function $\L_d$ along an admissible path, starting from the first derivative obtained in Lemma~\ref{lem:L'}.

\begin{lemma}\label{lem:L''} Along an admissible path 
$$L''(0)=X-Y$$
with
\begin{equation*}
\begin{split}
& X:=\sum_{e\in E_0} w_e\ell_e \int_0^1 \frac{1}{2}h''(0)(U,U)-\frac{1}{4}h'(0)(U,U)^2ds~, \\
& Y:= \sum_{e\in E_0} w_e\ell_e\int_0^1 \|\nabla_{U} V^\bot\|^2 +\|V^\bot\|^2ds~. 
\end{split}
\end{equation*}
Here  $\|\cdot\|$ is the $h(0)$ norm and $V^\bot$  is the component of $V$ orthogonal to $U$.
\end{lemma}
\begin{proof} With the same argument as in the proof of Lemma~\ref{lem:L'},
 for $t\not=0$,
$$L''(t)=\frac{d^2}{du^2}\L_{h(u)}(f(t))|_{u=t} + \frac{d}{d\alpha}\left(\frac{d}{du} \L_{h(u)}(f(\alpha))|_{u=t}\right)|_{\alpha=t}~. $$

It is easy to see that 
$$\frac{d^2}{du^2}\L_{h(u)}(f(t))|_{u=t} =\sum_{e\in E} w_e\ell_e \int_0^1 \frac{1}{2}h''(t)(U(t),U(t))-\frac{1}{4}h'(t)(U(t),U(t))^2ds~, $$
which extends continuously when $t\to 0^\pm$.

The derivative of \eqref{eq:equation geodesiques} with respect to $t$
says that 
$$\frac{d}{du}\left(\frac{d}{d\alpha} \L_{h(u)}(f(\alpha))|_{\alpha=t}\right)|_{u=t}=-\frac{d^2}{d\alpha^2}\L_{h(t)}(f(\alpha))|_{\alpha=t}~. $$

From \cite[(6.1.7)]{jost},
$$\frac{d^2}{d\alpha^2}\L_{h(t)}(f(\alpha))|_{\alpha=t}=\sum_e \frac{w_e}{\ell_e}\left(\int_0^1 \|\nabla_T V^\bot\|^2 - h(R(T,V^\bot)V^\bot,T)ds+ h(\nabla_VV,T)|^{s=1}_{s=0}\right)~, $$
where $R$ is the Riemann curvature tensor of $h=h(0)$.
 From Remark~\ref{rem:somme sur V},
$$\sum_e \frac{w_e}{\ell_e}h(\nabla_VV,T)|^{s=1}_{s=0}=0~.$$ Also, as the sectional curvature of $h$ is constant equal to $-1$, we obtain
$$\frac{d^2}{d\alpha^2}\L_{h(t)}(f(\alpha))|_{\alpha=t}=\sum_e w_e \ell_e\int_0^1 \|\nabla_U V^\bot\|^2 + \|V^\bot\|^2ds~, $$
that leads to the result, up to justify that we can change the order between $\frac{d}{du}$ and $\frac{d}{d\alpha}$. Indeed, one may write explicitly $\frac{d^2}{d\alpha du} \L_{h(u)}(f(\alpha))$ and see that all the second partial derivatives of $(u,t)\mapsto \L_{h(u)}(f(t))$ are continuous.
\end{proof}

The following result is an adaptation of the corresponding one in \cite{wolf} (where the second variation of length of closed geodesics is studied) and of some arguments of \cite{KT} (where the second variation of energy of weighted graphs is studied).

\begin{prop}\label{prop:derivve seconde}
If the admissible path follows a Weil--Petersson geodesic up to second order, 
then $L''(0)>0.$
\end{prop}
\begin{proof}
We denote the respective integrals over an edge $e$ in Lemma~\ref{lem:L''}
by $X_e$ and $Y_e$, so
\begin{equation}
\label{eq:L''}
L''(0)=\sum_{e \in E_0} \omega_e\ell_e(X_e - Y_e)~.
\end{equation}

Let $b$ be the traceless symmetric Codazzi tensor that corresponds to  $h'(0)$. Pick an edge $e\in E_0$. Let $U^\perp$ be the rotation by $\pi/2$ of $U$. We denote by $\alpha: [0,1] \rar \R$ and by $\beta: [0,1] \rar \R$ the functions
\begin{equation*}
\alpha:=\frac{1}{2}b(U,U),~~~~~
\beta:=\frac{1}{2}b(U,U^\bot)~.
\end{equation*}
(In more details, e.g., $\alpha(s)$, $s \in [0,1]$, is defined as $\frac{1}{2}b_{f_e(s)}(U_{f_e(s)}, U_{f_e(s)})$.)

Let us first bound $X_e$. We have
$$X_e=\int_0^1 \frac{1}{2} h''(0)(U,U)-\frac{1}{4}h'(0)(U,U)^2ds~.$$
As $b$ is symmetric and traceless, \begin{equation}\label{eq:norme de b}\|b\|^2=\tr_h(bb)=2(b(U,U)^2+b(U, U^\bot)^2)=8\alpha^2+8\beta^2~.
\end{equation} 
As we are  following up to second order a Weil--Petersson geodesic starting from $h$ and directed by the Codazzi tensor $b$, \eqref{eq:derivee seconde le lon WP} applies, and we get
\begin{equation}
\label{eq:A_e1}
X_e=\int_0^1 -\frac{1}{4}(\Delta-2)^{-1}\|b\|^2+\beta^2  ds~.
\end{equation}

We note that \cite[Lemma 5.1]{wolf}, \cite[Formula (3.13)]{KT}
\begin{equation}
\label{t6}
-\frac{1}{4}(\Delta-2)^{-1}\|b\|^2 \geq \frac{1}{24}\|b\|^2~.
\end{equation}
So we substitute it in~(\ref{eq:A_e1}) and get
\begin{equation}
\label{A_e2}
X_e \geq \int_0^1 \frac{1}{3}\alpha^2+\frac{4}{3}\beta^2 ds~.
\end{equation}

Now we deal with $Y_e$.
Take the deformation vector field $V$ and decompose it into the tangential and orthogonal parts, $V^\top$ and $V^\bot$, to the edge $e$. We denote by $\zeta: [0,1] \rar \R$ and $\eta: [0,1] \rar \R$ the oriented lengths of $V^\top$ and $V^\bot$ respectively, pulled back to $[0,1]$. In other terms, 
\begin{equation*}
\zeta:=h(V, U), ~~~~~
\eta:=h(V, U^\perp)~.
\end{equation*}

In the following, we will denote by $\dot{\phantom{x}}$ the derivative with respect to the parameter $s$ of the respective functions.
As $\zeta$ is the oriented length of $V^\top$,
$\dot{\zeta}$ is the oriented length of $\ell_e\nabla_U V^\top$ and because $\nabla_U U=0$, $\nabla_U V^\top$ is collinear with $U$, 
hence
\begin{equation}\label{eq:zeta dot}
\dot{\zeta}=\ell_eh(\nabla_UV,U)=\ell_e h(\nabla_UV^\top,U)~,
\end{equation}
\begin{equation}\label{eq zeta zeta'}
\zeta\dot{\zeta}=\ell_e h(V^\top, \nabla_U V^\top)~.
\end{equation} 

Similarly, $\dot{\eta}$ is the oriented length of $\ell_e\nabla_U V^\bot$. Then $\nabla_U V^\bot$ is orthogonal to $U$, i.e., collinear to $U^\bot$, 
hence
\begin{equation}\label{eq:alpha dot}
\dot{\eta}=\ell_eh(\nabla_UV,U^\bot)=\ell_e h(\nabla_UV^\bot,U^\bot)~,
\end{equation}
\begin{equation}\label{eq alpha lpha'}
\eta\dot{\eta}=\ell_e h(V^\perp, \nabla_U V^\perp)=\ell_e h(V, \nabla_U V)-\ell_e h(V^\top, \nabla_U V^\top)=\ell_e h(V, \nabla_U V)-\zeta\dot\zeta~.
\end{equation} 

The key points are connections between $\alpha$ and $\zeta$ and between $\beta$ and $\eta$. We start from the first pair. Note that  by \eqref{eq:crochet 0 TU} and \eqref{eq:T et U}, along an edge,
\begin{equation}
\nabla_U V=\nabla_V U +\frac{\ell'}{\ell}U~.
\end{equation}
From this and $h_t(U_t, U_t)=1$, we get by differentiating at $t=0$,
\[0=b(U, U)+2h(\nabla_V U, U)=b(U, U)+2h(\nabla_U V, U)-2\frac{\ell'}{\ell}~.\]
From this,
\begin{equation}
\label{t3}
\dot{\zeta}=\ell h(\nabla_U V, U)=-\frac{\ell}{2}b(U, U)+{\ell'}=-\ell\alpha+{\ell'}~.
\end{equation}

For the pair $\eta$ and $\beta$ we have the following differential equation  
\begin{equation}\label{eq:la equation}\ddot{\eta}-\ell^{2}\eta=-\ell\dot{\beta}~.\end{equation}

This formula is proved in
  Section 4.1 in \cite{wolf} and also in  Lemma 3.12 in \cite{KT}. Note that in  \cite{wolf}, what is denoted by $V$ is the normal component of the deformation, as there Wolf is interested in the case of closed geodesic, thus neglecting the impact of the tangent component. We, however, have to care about the impact of $V^\top$ on the boundary terms. Also, the geodesics are parameterized by arc-length in \cite{wolf}, while they are parameterized proportionally to arc-length in our case. In contrast, the parametrization in \cite{KT} is the same as ours.

It is easy to check (as in \cite{wolf}) that there exists a primitive $\theta$ of $\eta$ satisfying
\begin{equation}
\label{t2b}
\ddot{\theta}-\ell^2\theta=-\ell\beta~,
\end{equation}
and, as a primitive, $\dot{\theta}=\eta$.

Using \eqref{eq:alpha dot},  the term $Y_e$ in (\ref{eq:L''}) can be written as
\begin{equation}\label{eq:Be bis}Y_e=\int_0^1 (\ell^{-1}\dot{\eta})^2  + \eta^2 ds~. \end{equation}

We now partially follow computations from \cite{wolf} keeping a special track on the boundary terms (which do not appear in~\cite{wolf} as only the case of a closed geodesic is considered there): 
\begin{equation}
\label{B_e}
\begin{split}
Y_e&=\int_0^1 (\ell^{-1}\dot{\eta})^2  + \eta^2 ds \\
&\stackrel{\mathrm{IBP}}{=} \ell^{-2}\eta\dot{\eta}|^1_0+\int^1_0(\eta^2-\ell^{-2}\eta\ddot{\eta})ds \\ 
&\stackrel{\eqref{eq:la equation}}{=}\ell^{-2}\eta\dot{\eta}|_0^1+\ell^{-1}\int_0^1\eta\dot{\beta}ds\\
&\stackrel{\mathrm{IBP}}{=}\ell^{-2}\eta\dot{\eta}|^1_0+\ell^{-1}\eta\beta|^1_0-\ell^{-1}\int^1_0\dot{\eta}\beta ds~.
\end{split}
\end{equation}
We denote $\ell^{-1}\int^1_0\dot{\eta}\beta ds$ by $Y^*_e$. We first see what happens with the boundary term $\ell^{-2}\eta\dot{\eta}|^1_0+\ell^{-1}\eta\beta|^1_0$ in the total sum, and then bound $Y^*_e$.
 
We substitute \eqref{eq alpha lpha'} instead of $\eta\dot{\eta}|^1_0$. Then, in the total sum  the term $\ell_e h(V, \nabla_U V)$ gets eliminated by the balance condition, so we have
\begin{equation}
\label{flip}
\sum_{e \in E_0} \omega_e\ell^{-1}_e\eta\dot{\eta}|^1_0=-\sum_{e \in E_0} \omega_e\ell^{-1}_e\zeta\dot{\zeta}|^1_0~.
\end{equation}

Using integration by parts and (\ref{t3}), we get
\begin{equation}
\label{t8}
\zeta\dot{\zeta}|^1_0=\int_0^1(\dot{\zeta}^2+\ddot{\zeta}\zeta)ds=\int_0^1(\dot{\zeta}^2-\ell\dot{\alpha}\zeta)ds
=-\ell\alpha\zeta|^1_0+\int^1_0\dot{\zeta}(\dot{\zeta}+\ell\alpha)ds~.
\end{equation}

By definition,
\begin{multline}
\label{elimin}
\eta\beta+\alpha\zeta=\frac{1}{2}(h(V, U^{\perp})b(U, U^\perp)+h(V, U)b(U,U))=\\=\frac{1}{2}b(U, h(V, U^{\perp})U^\perp+h(V, U)U)=\frac{1}{2}b(U, V)~.
\end{multline}
When we take its weighted sum, it disappears due to the balance condition.

From \eqref{t3},
\begin{equation}
\label{t9}
\dot{\zeta}(\dot{\zeta}+\ell\alpha)=\left({\ell'}\right)^2-\alpha{\ell}{\ell'}=\left({\ell'}-\frac{{\ell}\alpha}{2} \right)^2-\frac{{\ell}^2\alpha^2}{4}\geq -\frac{{\ell}^2\alpha^2}{4}~. 
\end{equation}

Using (\ref{B_e}), (\ref{flip}), (\ref{t8}), (\ref{elimin}) and (\ref{t9}) we then get 
\begin{equation*}
-Y=-\sum_{e\in E_0} w_e \ell_e Y_e \geq \sum_{e\in E_0} w_e \ell_e \left(Y^*_e-\int_0^1 \frac{\alpha^2}{4}ds  \right)~.
\end{equation*}

Introducing the notations $Y^*:= \sum_e w_e \ell_e Y^*_e$, $A:=\sum_e w_e \ell_e \int_0^1 \alpha^2$, the inequality above is expressed as

\begin{equation}
\label{B}
-Y\geq Y^* - \frac{1}{4}A~.
\end{equation}

Now we write
$$Y^*_e=\ell^{-1}\int^1_0\dot{\eta}\beta ds= \int_0^1\ell^{-1}\left(\left(\sqrt {\frac{\ell^{-1}}{4}}\dot\eta+\sqrt {\ell}\beta\right)^2-\frac{\ell^{-1}}{4}\dot\eta^2-\ell  \beta^2\right)ds~.$$
By adding to the integrand $\pm \eta^2/4$, with \eqref{eq:Be bis} we have that
\begin{equation}
\label{2way}
Y^*_e  = - \frac{Y_e}{4} + \int_0^1 \frac{\eta^2}{4}+\ell^{-1}\left(\sqrt {\frac{\ell^{-1}}{4}}\dot\eta+\sqrt {\ell }\beta\right)^2ds -\int_0^1\beta^2 ds~. 
\end{equation}
We denote $B:=\sum_{e \in E_0} w_e \ell_e \int_0^1 \beta^2$,
\begin{equation}
\label{remainder}
C:=\sum_{e \in E_0}w_e\ell_e \int_0^1\frac{\eta^2}{4}+\ell^{-1}\left(\sqrt {\frac{\ell^{-1}}{4}}\dot\eta+\sqrt {\ell }\beta\right)^2ds~.
\end{equation}
Taking the weighted sum of (\ref{2way}), we get
\begin{equation}
\label{eq:B* ter}
Y^*= -\frac{1}{4}Y-B+C~.\end{equation}

Using this with \eqref{B}, we obtain
$$-Y\geq -\frac{1}{4}Y - \frac{1}{4}A-B+C~,$$
i.e.
$$-Y\geq -\frac{1}{3}A-\frac{4}{3}B+\frac{4}{3}C ~. $$

With our current notation, \eqref{A_e2} gives
$$X\geq \frac{1}{3}A+\frac{4}{3}B~, $$ hence 
 $$L''(0)=X-Y\geq \frac{4}{3}C~.$$
Looking at (\ref{remainder}), we see that $L''(0)$ is non-negative. Moreover, $L''(0)=0$ implies $C=0$, hence on every edge $\eta=0$. Thereby, also $\dot\eta=0$ on every edge, and then $C=0$
also implies that $\beta=0$ on every edge, so $B=0$. From the definition of $Y_e^*$ and $Y^*$, we then get that $Y^*=0$. Thus (\ref{B}) gives us
$$-Y \geq -\frac{1}{4}A$$
and we have
$$L''(0)=X-Y\geq \frac{1}{12}A~.$$
It follows that $L''(0)=0$ also implies that $\alpha=0$ at every edge. Then $b$ is zero along every edge. But $b$ is the real part of a holomorphic quadratic differential, see \eqref{def:Qb}. By holomorphy, it follows that $b=0$ over $S$, which is excluded by our choice of the path.
\end{proof}

The Weil--Petersson metric is not complete over $\mc T(S)$, but it has an exhaustion by compact  convex sets \cite{wolpert87}, so we may speak about Weil--Petersson  convex funtions.

\begin{cor}\label{cor Ld cvxe}
The Weil--Petersson Hessian of $\L_d$ is positive definite and $\L_d$ is strictly convex.
\end{cor}
\begin{proof}
Although the Weil--Petersson metric is not complete over $\mc T(S)$, the exponential map is a homeomorphism from its domain \cite{wolpert87}, hence we may consider a Weil--Petersson geodesic tangent to a traceless symmetric Codazzi tensor $b$. We then consider an admissible path of $\mc T(S)$ that coincides with this Weil--Petersson geodesic up to second order. Recall that $\L_d$ is $C^2$ by Corollary~\ref{cor:C2}. Hence from  Proposition~\ref{prop:derivve seconde}, the Hessian of $\L_d$ is positive definite. As any pair of points in $\mc T(S)$ can be joined by a Weil--Petersson geodesic \cite{wolpert87}, it follows that $\L_d$ is strictly convex.
\end{proof}

\subsection{Proof of Theorem~\ref{thmII'}}\label{sec: proof II}

Having the results of the preceding sections in hand,  one may mimic the argument in \cite{bonahon} to conclude the proof of Theorem~\ref{thmII'}.

Let $d_1$ and $d_2$ be two flat metrics with negative singular curvatures over $S$. The function $L_{d_1}+L_{d_2}$ is $C^1$ and proper (Lemma~\ref{diff2} and Lemma~\ref{proper2}), hence it has a critical point $h\in \mc T(S)$. As $X_{d_i}$ is minus the Weil--Petersson symplectic gradient of $L_{d_i}$ (Proposition~\ref{prop:WP gradient}), it follows that $X_{d_1}(h)+X_{d_2}(h)=0$. As $L_{d_1}+L_{d_2}$ is strictly convex (Corollary~\ref{cor Ld cvxe}), this critical point is unique. As the Hessian of $L_{d_1}+L_{d_2}$ is positive definite (Corollary~\ref{cor Ld cvxe}), the intersection is transverse.

\small

\bibliographystyle{alpha}
\bibliography{flat-ghmc-with-polyhedral-boundary}
\bigskip{\footnotesize\par
  \textsc{IMAG, Universit\'e de Montpellier, CNRS, Montpellier, France} \par
  \textit{E-mail}: \texttt{francois.fillastre@umontpellier.fr}
}
\bigskip{\footnotesize\par
  \textsc{University of Vienna, Faculty of Mathematics, Oskar-Morgenstern-Platz 1, A-1090 Vienna, Austria} \par
  \textit{E-mail}: \texttt{roman.prosanov@univie.ac.at}
}

\end{document}